\documentclass[11pt, reqno]{article}
\title{
Universality of high-dimensional spanning forests and sandpiles
}
\author{Tom Hutchcroft\footnote{Statslab, DPMMS, University of Cambridge}
}
\usepackage{amsmath,amssymb,amsfonts,amsthm}
\usepackage{color}

\usepackage[colorlinks=true,linkcolor=blue,citecolor=blue,urlcolor=black,pdfborder={0 0 0}]{hyperref}
\usepackage{graphicx}
\usepackage{cleveref}

\usepackage{caption}
\usepackage{thmtools}
\usepackage{commath}
\usepackage{mathrsfs}

\usepackage{enumitem}
\usepackage{eufrak}

\crefname{theorem}{Theorem}{Theorems}
\crefname{thm}{Theorem}{Theorems}
\crefname{lemma}{Lemma}{Lemmas}
\crefname{lem}{Lemma}{Lemmas}
\crefname{remark}{Remark}{Remarks}
\crefname{prop}{Proposition}{Propositions}
\crefname{proposition}{Proposition}{Propositions}
\crefname{defn}{Definition}{Definitions}
\crefname{corollary}{Corollary}{Corollaries}
\crefname{conjecture}{Conjecture}{Conjectures}
\crefname{question}{Question}{Questions}
\crefname{chapter}{Chapter}{Chapters}
\crefname{section}{Section}{Sections}
\crefname{figure}{Figure}{Figures}

\theoremstyle{plain}
\newtheorem{thm}{Theorem}[section]
\newtheorem{theorem}[thm]{Theorem}
\newtheorem{lemma}[thm]{Lemma}

\newtheorem{corollary}[thm]{Corollary}
\newtheorem{prop}[thm]{Proposition}
\newtheorem{proposition}[thm]{Proposition}

\theoremstyle{definition}

\theoremstyle{remark}
\newtheorem{remark}[thm]{Remark}

\numberwithin{equation}{section}

\renewcommand{\P}{\mathbb P}
\newcommand{\E}{\mathbb E}

\newcommand{\R}{\mathbb R}
\newcommand{\Z}{\mathbb Z}
\newcommand{\N}{\mathbb N}

\newcommand{\F}{\mathfrak F}

\newcommand{\cI}{\mathcal I}

\newcommand{\cW}{\mathcal W}

\newcommand{\sA}{\mathscr A}
\newcommand{\sB}{\mathscr B}
\newcommand{\sC}{\mathscr C}
\newcommand{\sD}{\mathscr D}

\newcommand{\sI}{\mathscr I}

\newcommand{\fB}{\mathfrak B}

\newcommand{\fF}{\mathfrak F}

\newcommand{\fP}{\mathfrak P}

\newcommand{\fT}{\mathfrak T}

\newcommand{\WUSF}{\mathsf{WUSF}}

\newcommand{\UST}{\mathsf{UST}}

\newcommand{\LE}{\mathsf{LE}}

\usepackage[numeric,initials,nobysame]{amsrefs}

\usepackage{setspace}
\usepackage{commath}
\usepackage{mathrsfs}
\usepackage{bbm}
\usepackage{mathtools}
 \usepackage[margin=1.1in]{geometry}
\usepackage{enumitem}

\newcommand{\Ceff}{\sC_{\mathrm{eff}}}
\newcommand{\eps}{\varepsilon}
\newcommand{\AB}{\mathsf{AB}}
\renewcommand{\Cap}{\mathrm{Cap}}

\newcommand{\past}{\mathrm{past}}
\newcommand{\diam}{\mathrm{diam}}
\newcommand{\ext}{\mathrm{ext}}
\newcommand{\inte}{\mathrm{int}}
\newcommand{\bP}{\mathbf P}
\newcommand{\bE}{\mathbf E}
\newcommand{\Reff}{\mathscr{R}_\mathrm{eff}}
\renewcommand{\Ceff}{\mathscr{C}_\mathrm{eff}}
\newcommand{\Var}{\operatorname{Var}}

\newcommand{\pstar}[1]{\|P^{#1}\|_{1\to\infty}}
\newcommand{\bubnorm}[1]{\|{#1}\|_{\mathrm{bub}}}

\newcommand{\preb}{\preceq}

\newcommand{\Eta}{\mathrm{H}}
\newcommand{\Av}{\operatorname{Av}}
\newcommand{\AvC}{\operatorname{AvC}}

\newcommand{\myfrac}[3][0pt]{\genfrac{}{}{}{}{\raisebox{#1}{$#2$}}{\raisebox{-#1}{$#3$}}}

\begin{document}

\newgeometry{margin=1.2in}

\maketitle

\begin{abstract}
\setstretch{1}
 We prove that the wired uniform spanning forest exhibits mean-field behaviour on a very large class of graphs, including every transitive graph of at least quintic volume growth and every bounded degree nonamenable graph. Several of our results are new even in the case of $\Z^d$, $d\geq 5$.
In particular, we prove that every tree in the forest has  spectral dimension $4/3$ and walk dimension $3$ almost surely, and that the critical exponents governing 
the 
intrinsic diameter and volume of the past of a vertex in the forest are $1$ and $1/2$ respectively. (The past of a vertex in the uniform spanning forest is the union of the vertex and the finite components that are disconnected from infinity  when that vertex is deleted from the forest.)
 We obtain as a corollary that the critical exponent governing the \emph{extrinsic} diameter of the past is $2$ on any transitive graph of at least five dimensional polynomial growth, and is $1$ on any bounded degree nonamenable graph.
   We deduce that the critical exponents describing the diameter and total number of topplings in an avalanche in the Abelian sandpile model are $2$ and $1/2$ respectively for any transitive graph with polynomial growth of dimension at least five, and are $1$ and $1/2$ respectively for any bounded degree nonamenable graph.

   In the case of $\Z^d$, $d\geq 5$, some of our results regarding critical exponents recover earlier results of Bhupatiraju, Hanson, and J\'arai (\emph{Electron.\ J.\ Probab.\ Volume 22 (2017), no.\ 85}). In this case, we improve upon their results by showing that the tail probabilities in question are described by the appropriate power laws to within constant-order multiplicative errors, rather than the polylogarithmic-order multiplicative errors present in that work.
\end{abstract}

% \newpage

\setstretch{1.1}

\newpage

\tableofcontents

\newgeometry{margin=1.08in}

\newpage

\section{Introduction}

The \textbf{uniform spanning forests} (USFs) of an infinite graph $G$ are defined 
as weak limits of the uniform spanning trees of finite subgraphs of $G$. These limits can be taken with respect to two extremal boundary conditions, yielding the \textbf{free uniform spanning forest} (FUSF) and \textbf{wired uniform spanning forest} (WUSF).  For transitive amenable graphs such as the hypercubic lattice $\Z^d$, the free and wired forests coincide and we speak simply of the USF. 
In this paper we shall be concerned exclusively with the wired forest.
Uniform spanning forests have played a central role in the development of probability theory over the last twenty years, and are closely related to several other topics in probability and statistical mechanics including electrical networks \cite{kirchhoff1847ueber,BurPe93,BLPS}, loop-erased random walk \cite{Lawler80,Wilson96,BLPS}, the random cluster model \cite{GrimFKbook,Hagg95}, domino tiling \cite{BurPe93,Ken00}, random interlacements \cite{Szni10,hutchcroft2015interlacements}, conformally invariant scaling limits \cite{Schramm00,LaSchWe04,MR3719057}, and the Abelian sandpile model \cite{Dhar90,MajDhar92,JarRed08,JarWer14}. Indeed, our results have important implications for the Abelian sandpile model, which we discuss in \cref{subsec:sandpileintro}.

 % . Indeed, the results of this paper have important consequences for the Abelian sandpile model, which we explore in \cref{subsec:sandpileintro}.
% 

Following the work of many authors, the basic qualitative features of the WUSF are firmly understood on a wide variety of graphs. In particular, it is known that every tree in the WUSF is recurrent almost surely on any graph \cite{Morris03}, that the WUSF is connected a.s.\ if and only if two independent random walks on $G$ intersect almost surely \cite{Pem91,BLPS}, and that every tree in the WUSF is one-ended almost surely whenever $G$ is in one of several large classes of graphs \cite{Pem91,BLPS,LMS08,H15,AL07,HutNach15b} including all transient transitive graphs. (An infinite tree is one-ended if it does not contain a simple bi-infinite path.)

The goal of this paper is to understand the geometry of trees in the WUSF at a more detailed, quantitative level, under the assumption that the underlying graph is high-dimensional in a certain sense. 
 Our results can be summarized informally as follows. 
 Let $G$ be a connected graph with bounded degrees, and suppose that the $n$-step return probabilities $p_n(v,v)$ for (discrete-time) random walk on $G$ satisfy $\sum_{n\geq 1} n \sup_{v\in V} p_n(v,v)<\infty$. In particular, this holds for $\Z^d$ if and only if $d \geq 5$, and more generally for any transitive graph of at least quintic volume growth. 
 Let $\F$ be the WUSF of $G$, and let $v$ be a vertex of $G$. The \textbf{past}\footnote{This terminology arises from the definition of the past in terms of an oriented version of $\F$, see \cref{subsec:introintrinsic}.} of $v$ in $\F$ is the union of $v$ with the finite connected components of $\F \setminus \{v\}$. The following hold:
\begin{enumerate}[leftmargin=*]
\itemsep0.5em
	\item
	\emph{The 
	   intrinsic geometry} of each tree in $\F$ is similar at large scales to that of a critical Galton-Watson tree with finite variance offspring distribution,  conditioned to survive forever. In particular, every tree has volume growth dimension $2$ (with respect to its intrinsic graph metric), spectral dimension $4/3$, and walk dimension $3$ almost surely. The latter two statements mean that the $n$-step return probabilities for simple random walk on the tree decay like $n^{-2/3+o(1)}$, and that the typical displacement of the walk (as measured by the intrinsic graph distance in the tree) is $n^{1/3+o(1)}$.
  These are known as the \emph{Alexander-Orbach} values of these dimensions~\cite{AlexOrb1982,KozNach09}.
  	% \item
  	%  \emph{The extrinsic geometry} of each tree $\F$ is similar at large scales to that of a critical branching random walk on $G$, with finite variance offspring distribution and conditioned to survive forever.
\item \emph{The intrinsic geometry of the past} of $v$ in $\F$ is similar in law to that of an \emph{unconditioned} critical Galton-Watson tree with finite variance offspring distribution. In particular, the probability that the past contains a path of length at least $n$ is of order $n^{-1}$, and the probability that the past contains more than $n$ points is of order $n^{-1/2}$.  That is, the \emph{intrinsic diameter exponent} and \emph{volume exponent} are $1$ and $1/2$ respectively.
% \end{enumerate}

% Under some additional technical assumptions on $G$ (which we show hold for several classes of examples including the Euclidean case), we obtain the following.

% \begin{enumerate}[leftmargin=*]
\item[3.] \emph{The extrinsic geometry of the past} of $v$ in $\F$ is similar in law to that of an \emph{unconditioned} critical branching random walk on $G$ with finite variance offspring distribution. In particular, the probability that the past of $v$ includes a vertex at extrinsic distance at least $n$ from $v$ depends on the rate of escape of the random walk on $G$. For example, it is of order $n^{-2}$ for $G=\Z^d$ for $d\geq 5$  and is of order $n^{-1}$ for $G$ a transitive nonamenable graph. This is related to the fact that the random walk on the ambient graph $G$ is diffusive in the former case and ballistic in the latter case.
 \end{enumerate}

All of these results apply more generally to \emph{networks} (a.k.a.\ weighted graphs); see the remainder of the introduction for details.

In light of the connections between the WUSF and the Abelian sandpile model,  these results imply related results for that model, to the effect that an avalanche in the Abelian sandpile model has a similar distribution to a critical branching random walk (see \cref{subsec:sandpileintro}). 
  Precise statements of our results and further background are given in the remainder of the introduction.

The fact that our results apply at such a high level of generality is a strong vindication of \emph{universality} for high-dimensional spanning trees and sandpiles, which predicts that the large-scale behaviour of these models should depend only on the dimension, and in particular should be insensitive to the microscopic structure of the lattice. In particular, our results apply not only to $\Z^d$ for $d\geq 5$, but also to non-transitive networks that are similar to $\Z^d$ such as the half-space $\Z^{d-1}\times \N$ or, say, $\Z^d$ with variable edge conductances bounded between two positive constants. Many of our results also apply to long-range spanning forest models on $\Z^d$ such as those associated with the fractional Laplacian $-(-\Delta)^\beta$ of $\Z^d$ for $d\geq 1$, $\beta\in(0,d/4 \wedge 1)$. Long-range models such as these are motivated physically as a route towards understanding low-dimensional models via the $\eps$-expansion \cite{wilson1972critical}, for which it is desirable to think of the dimension as a continuous parameter. (See the introduction of \cite{Slade2017} for an account of the $\eps$-expansion for mathematicians.)
\medskip

\noindent \textbf{About the proofs.}
  Our proof relies on the interplay between two different ways of sampling the WUSF. The first of these is \emph{Wilson's algorithm}, a method of sampling the WUSF by joining together loop-erased random walks which was introduced by David Wilson \cite{Wilson96} and extended to infinite transient graphs by Benjamini, Lyons, Peres, and Schramm \cite{BLPS}.  The second is the \emph{interlacement Aldous-Broder algorithm}, a method of sampling the WUSF as the set of first-entry edges of Sznitman's \emph{random interlacement process} \cite{Szni10}. This algorithm was introduced in the author's recent work \cite{hutchcroft2015interlacements} and extends the classical Aldous-Broder algorithm  \cite{Aldous90,broder1989generating} to infinite transient graphs. Generally speaking, it seems that Wilson's algorithm is the better tool for estimating the moments of random variables associated with the WUSF, while the interlacement Aldous-Broder algorithm is the better tool for estimating tail probabilities. 

  A key feature of the interlacement Aldous-Broder algorithm is that it enables us to think of the WUSF as the stationary measure of a natural continuous-time Markov chain. Moreover, the past of the origin evolves in an easily-understood way under these Markovian  dynamics. In particular, as we run time backwards, the past of the origin gets monotonically smaller except possibly for those times at which the origin is visited by an interlacement trajectory. Indeed, the central insight in the proof of our results is that \emph{static} tail events (on which the past of the origin is large) can be related to to \emph{dynamic} tail events (on which the origin is hit by an interlacement trajectory at a small time). Roughly speaking, we show that these two types of tail event tend to occur together, and consequently have comparable probabilities. We make this intuition precise using inductive inequalities similar to those used to analyze one-arm probabilities in high-dimensional percolation \cite{KozNach09,MR2748397,MR3418547}. 

  Once the critical exponent results are in place, the results concerning the simple random walk on the trees can be proven rather straightforwardly using the results and techniques of Barlow, J\'arai, Kumagai, and Slade \cite{BJKS08}.

\subsection{Relation to other work}
\begin{itemize}
 \item When $G$ is a regular tree of degree $k\geq 3$, the components of the WUSF are distributed exactly as augmented critical binomial Galton-Watson trees conditioned to survive forever, and in this case all of our results are classical \cite{kestensubdiff,MR1349164,BarKum06}.

 \item In the case of $\Z^d$ for $d\geq 5$, Barlow and J\'arai \cite{barlow2016geometry} established that the trees in the WUSF have quadratic volume growth almost surely. Our proof of quadratic volume growth uses similar methods to theirs, which were in turn inspired by related methods in percolation due to Aizenman and Newman \cite{MR762034}. 

 \item Also in the case of $\Z^d$ for $d\geq 5$, Bhupatiraju, Hanson, and J\'arai \cite{bhupatiraju2016inequalities} followed the strategy of an unpublished proof of Lyons, Morris, and Schramm \cite{LMS08} to prove that the probability that the past reaches extrinsic distance $n$ is $n^{-2} \log^{O(1)}n$ and that the probability that the past has volume $n$ is $n^{-1/2} \log^{O(1)} n$. Our results improve upon theirs in this case by reducing the error from polylogarithmic to constant order. Moreover, their proof relies heavily on transitivity and cannot be used to derive universal results of the kind we prove here.

 \item Peres and Revelle \cite{peres2004scaling} proved that the USTs of large $d$-dimensional tori converge under rescaling (with respect to the Gromov-weak topology) to Aldous's continuum random tree when $d\geq 5$. They also proved that their result extends to other sequences of finite transitive graphs satisfying a heat-kernel upper bound similar to the one we assume here. Later, Schweinsberg \cite{schweinsberg2009loop} established a similar result for four-dimensional tori. Related results concerning loop-erased random walk on high-dimensional tori had previosuly been proven by Benjamini and Kozma \cite{MR2172682}. While these results are closely related in spirit to those that we prove here, it does not seem that either can be deduced from the other.

\item For planar Euclidean lattices such as $\Z^2$, the UST is very well understood thanks in part to its connections to conformally invariant processes in the continuum \cite{Ken00,Schramm00,LaSchWe04,Masson09,MR3719057}. In particular, Barlow and Masson \cite{BarMass10,BarMass11} proved that the UST of $\Z^2$ has volume growth dimension $8/5$ and spectral dimension $16/13$ almost surely.  See also \cite{BCK17} for more refined results.

 \item
 In \cite{HutNach15b}, the author and Nachmias established that the WUSF of any transient proper plane graph with bounded degrees and codegrees has mean-field critical exponents provided that measurements are made using the hyperbolic geometry of the graph's circle packing rather than its usual combinatorial geometry. Our results recover those of \cite{HutNach15b} in the case that the graph in question is also uniformly transient, in which case it is nonamenable and the graph distances and hyperbolic distances are comparable. 

 \item A consequence of this paper is that several properties of the WUSF are insensitive to the geometry of the graph once the dimension is sufficiently large. In contrast, the theory developed in \cite{BeKePeSc04,hutchcroft2017component} shows that some other properties describing the adjacency structure of the trees in the forest continue to undergo qualitative changes every time the dimension increases.

 \item In forthcoming work with Sousi, we build upon the methods of this paper to analyze related problems concerning the uniform spanning tree in $\Z^3$ and $\Z^4$.
\end{itemize} 
 % Related results have previously been established in the case of $\Z^d$, $d\geq 5$ by  .

\subsection{Basic definitions}

In this paper, a \textbf{network} will be a connected graph $G=(V,E)$ (possibly containing loops and multiple edges) together with a function $c:E \to (0,\infty)$ assigning a positive \textbf{conductance} $c(e)$ to each edge $e \in E$ such that for each vertex $v\in V$, the \textbf{vertex conductance} $c(v):= \sum c(e)< \infty$, taken over edges incident to $v$, is finite. 
% The definitions of uniform spanning trees and forests of networks are given in [ref].
We say that the network $G$ has \textbf{controlled stationary measure} if there exists a positive constant $C$ such that $C^{-1} \leq c(v) \leq C$ for every vertex $v$ of $G$. Locally finite graphs can always be considered as networks by setting $c(e) \equiv 1$, and in this case have controlled stationary measure if and only if they have bounded degrees. We write $E^\rightarrow$ for the set of oriented edges of a network. An oriented edge $e$ is oriented from its tail $e^-$ to its head $e^+$ and has reversal $-e$.

 The \textbf{random walk} on a network $G$ is the process that, at each time step, chooses an edge emanating from its current position with probability proportional to its conductance, independently of everything it has done so far to reach its current position, and then traverses that edge.
We use $\mathbf{P}_v$ to denote the law of the random walk started at a vertex $v$, $\mathbf{E}_v$ to denote the associated expectation operator, and $P$ to denote the Markov operator $P:\ell^2(V,c) \to \ell^2(V,c)$, defined by
\[ (Pf)(v) = \mathbf{E}_v f(X_1) = \sum_{e^-=v} \frac{c(e)}{c(v)}f(e^+),\]
where $\ell^2(V,c)$ is the space of functions $f:V \to \R$ such that $\sum_v f(v)^2c(v) <\infty$.
Finally, for each two vertices $u$ and $v$ of $G$ and $n\geq 0$, we write $p_n(u,v) = \mathbf{P}_u(X_n = v)$ for the probability that a random walk started at $u$ is at $v$ at time $n$.
% \[p_n(u,v) = \mathbf{P}_u(X_n = v).\]

Given a graph or network $G$, a \textbf{spanning tree} of $G$ is a connected subgraph of $G$ that contains every vertex of $G$ and does not contain any cycles. 
The \textbf{uniform spanning tree} measure $\mathsf{UST}_G$ on a finite connected graph $G=(V,E)$ is the probability measure on subgraphs of $G$ (considered as elements of $E^{\{0,1\}}$) that assigns equal mass to each spanning tree of $G$. If $G$ is a finite connected \emph{network}, then the 
% The
 uniform spanning tree measure $\UST_G$
  % on a finite connected \emph{network} $G$ 
  of $G$ 
 is defined so that the mass of each tree is proportional to the product of the conductances of the edges it contains. That is,
\[\UST_G\left(\{\omega\}\right) = Z^{-1}_G \prod_{e\in \omega} c(e) \mathbbm{1}\left(\omega \text{ is a spanning tree of $G$}\right)\]
for every $\omega \in \{0,1\}^E$,  
where $Z_G$ is a normalizing constant. 

Let $G$ be an infinite network and let $\langle V_n \rangle_{n \geq1}$ be an \textbf{exhaustion} of $G$, that is, an increasing sequence of finite sets $V_n \subset V$ such that $\bigcup_{n\geq1} V_n = V$. 
For each $n \geq 1$, let $G^*_n$ be the network obtained from $G_n$ by contracting every vertex in $V \setminus V_n$ into a single vertex, denoted $\partial_n$, and deleting all the resulting self-loops from $\partial_n$. The \textbf{wired uniform spanning forest} measure on $G$, denoted $\WUSF_G$, is defined to be the weak limit of the uniform spanning tree measures on the finite networks $G_n^*$. That is,
\[\WUSF_G\left(\{\omega : S \subset \omega \}\right) = \lim_{n\to\infty}\UST_{G^*_n}\left(\{\omega : S \subset \omega\}\right)\]
for every finite set $S \subseteq E$. It follows from the work of Pemantle \cite{Pem91} that this limit exists and does not depend on the choice of exhaustion; See also \cite[Chapter 10]{LP:book}.
% We say that a random spanning forest $\F$ of $G$ is a wired uniform spanning forest of $G$ if it has law $\WUSF_G$.

In the limiting construction above, one can also orient the uniform spanning tree of $G_n^*$ towards the boundary vertex $\partial_n$, so that every vertex other than $\partial_n$ has exactly one oriented edge emanating from it in the spanning tree. If $G$ is transient, then the sequence of laws of these random oriented spanning trees converge weakly to the law of a random oriented spanning forest of $G$, which is known as the \textbf{oriented wired uniform spanning forest}, and from which we can recover the usual (unoriented) WUSF by forgetting the orientation. (This assertion follows from the proof of \cite[Theorem 5.1]{BLPS}.) It is easily seen that the oriented wired uniform spanning forest of $G$ is almost surely an \textbf{oriented essential spanning forest} of $G$, that is, an oriented spanning forest of $G$ such that every vertex of $G$ has exactly one oriented edge emanating from it in the forest (from which it follows that every tree is infinite).

\subsection{Intrinsic exponents}
\label{subsec:introintrinsic}
Let $\F$ be an oriented essential spanning forest of an infinite graph $G$. 
 We define the \textbf{past} of a vertex $v$ in $\F$, denoted $\fP(v)$, 
 to be the subgraph of $\F$ induced by the set of vertices $u$ of $\F$ such that every edge in the geodesic from $u$ to $v$ in $\F$ is oriented in the direction of $v$, where we also consider $v$ to be included in this set. (By abuse of notation, we will also use $\fP(v)$ to mean the vertex set of this subgraph.) Thus, a component of $\F$ is one-ended if and only if the past of each of its vertices is finite. The \textbf{future} of a vertex $v$ is denoted by $\Gamma(v,\infty)$ and is defined to be the set of vertices $u$ such that $v$ is in the past of $u$. 

In order to quantify the one-endedness of the WUSF, it is interesting to estimate the probability that the past of a vertex is large in various senses. Perhaps the three most natural such measures of largeness are given by the \emph{intrinsic diameter}, \emph{extrinsic diameter}, and \emph{volume} of the past. 
 Here, given a subgraph $K$ of a graph $G$, we define the \textbf{extrinsic diameter} of $K$, denoted $\diam_\ext(K)$, to be the supremal graph distance in $G$ between two points in $K$, and define the \textbf{intrinsic diameter} of $K$, denoted $\diam_\inte(K)$, to be the diameter of $K$. The \textbf{volume} of $K$, denoted $|K|$, is defined to be the number of vertices in $K$.

Generally speaking, for critical statistical mechanics models defined on Euclidean lattices such as $\Z^d$, many natural random variables arising geometrically from the model are expected to have power law  tails. The exponents governing these tails are referred to as \textbf{critical exponents}. 
For example, if $\F$ is the USF of $\Z^d$, we expect that for each $d\geq 2$ there exists $\alpha_d$ such that
\[ \P\left(\diam_\ext\left(\fP\left(0\right)\right) \geq R \right) = R^{-\alpha_d +o(1)},\]
in which case we call $\alpha_d$ the \emph{extrinsic diameter exponent} for the USF of $\Z^d$. 
% The \textbf{intrinsic diameter exponent} and \textbf{volume exponent} are defined similarly when they exist. 
Calculating and proving the existence of critical exponents is considered a central problem in probability theory and mathematical statistical mechanics.

 It is also expected that each model has an \textbf{upper-critical dimension}, denoted $d_c$, above which the critical exponents of the model stabilize at their so-called \emph{mean-field} values. For the uniform spanning forest, the upper critical dimension is believed to be four. Intuitively, above the upper critical dimension the lattice is spacious enough that different parts of the model do not interact with each other very much. This causes the model to behave similarly to how it behaves on, say, the $3$-regular tree or the complete graph, both of which have a rather trivial geometric structure.
 % \footnote{For the $3$-regular tree, the component of the root in the wired uniform spanning forest has the same distribution as the incipient infinite cluster for critical Bernoulli bond percolation. This suggests that for large $d$, the exponents for critical percolation and for the USF should be related.}.
Below the upper critical dimension, the geometry of the lattice affects the model in a non-trivial way, and the critical exponents are expected to differ from their mean-field values.  
 The upper critical dimension itself (which need not necessarily be an integer) is often characterised by the mean-field exponents holding up to a polylogarithmic multiplicative correction, which is not expected to be present in other dimensions. 

 For example, we expect that there exist constants $\gamma_2,\gamma_3,\gamma_{\mathrm{mf}}$ and $\delta$ such that
 \begin{align*} 
 \P\left(|\fP(0)| \geq R\right) \asymp
  \begin{cases} 
R^{-\gamma_2} & d=2\\
R^{-\gamma_3} & d=3\\
R^{-\gamma_{\mathrm{mf}}} \log^{-\delta} R & d=4\\
R^{-\gamma_{\mathrm{mf}}} & d\geq 5,
  \end{cases}
 \end{align*}
 where $\F$ is the uniform spanning forest of $\Z^d$ and $\asymp$ denotes an equality that holds up to positive multiplicative constants. Moreover, the values of the exponents $\gamma_2,\gamma_3,\gamma_{\mathrm{mf}},$ and $\delta$ should depend only on the dimension $d$, and not on the choice of lattice; this predicted insensitivity to the microscopic structure of the lattice is an instance of the phenomenon of \emph{universality}. 

Our first main result verifies the high-dimensional part of this picture for the intrinsic exponents. The corresponding results for the extrinsic diameter exponent are given in \cref{subsec:introextrinsic}.
Before stating our results,
% introducing the most general forms of these results,
 let us give some further definitions and motivation. 
For many models, an important signifier of mean-field behaviour is that certain \emph{diagrammatic sums} associated with the model are convergent. Examples include the \emph{bubble diagram} for self-avoiding walk and the Ising model, the \emph{triangle diagram} for percolation, and the \emph{square diagram} for lattices trees and animals; see e.g.\ \cite{MR2239599} for an overview. 
For the WUSF, the relevant signifier of mean-field behaviour on a network $G$ is the convergence of the \emph{random walk bubble diagram}
\[\sum_{x} \bigg(\sum_{n\geq 0} p_n(v,x)\bigg)^2,\] 
 where $p_n(\cdot,\cdot)$ denotes the $n$-step transition probabilities for simple random walk on $G$ and $v$ is a fixed root vertex. 
 Note that the value of the bubble diagram is exactly the expected number of times that two independent simple random walks started at the origin intersect. 
Using time-reversal, the convergence of the bubble diagram of a network with controlled stationary measure is equivalent to the convergence of the sum
\begin{equation*}
% \label{eq:bubble2}
% \sum_{x} \left(\sum_{n\geq 0} p_n(0,x)\right) \left(\sum_{n\geq 0} p_n(0,x)\right) \asymp
 \sum_{n \geq0} (n+1) p_n(v,v). \end{equation*}
 % It is sometimes taken as a definition of the upper 
 % We say that the bubble diagram \textbf{converges} if $\sum_{n\geq0}p_n(v,v)<\infty$, and that the bubble diagram \textbf{diverges logarithmically} if 
It is characteristic of the upper-critical dimension $d_c$ that the bubble diagram converges for all $d>d_c$, while at the upper-critical dimension itself we expect that the bubble diagram diverges logarithmically, as indeed is the case in our setting.
% Indeed for $\mathbb{Z}^4$ we have that 
 % \[ \sum_{n = 0}^N (n+1) p_n(v,v) \asymp \log N.\]

Our condition for mean-field behaviour of the 
WUSF will be that the bubble diagram converges uniformly in a certain sense. 
Let $G=(V,E)$ be a network, let $P$ be the Markov operator of $G$, and
  let $\|P^n\|_{1\to\infty}$ be the $1\to\infty$ norm of $P^n$, defined by
\[\|P^n\|_{1\to\infty} = \sup_{u,v\in V}p_n(u,v) .\]
We define the \textbf{bubble norm} 
% and \textbf{weak bubble norm}
 of $P$ to be
\[\bubnorm{P} := \sum_{n=0}^\infty (n+1) \pstar{n}.  \]
% \quad \text{ and } \quad  \bubnormw{P} := \sup_{n\geq 0} \, (n+1)^{2} \pstar{n}\]
% respectively. 
Thus, for transitive networks, $\bubnorm{P}<\infty$ is equivalent to convergence of the random walk bubble diagram. Here, a network is said to be \textbf{transitive} if for every two vertices $u,v\in V$ there exists a conductance-preserving graph automorphism mapping $u$ to $v$. Throughout the paper, we use $\asymp, \preceq$ and $\succeq$ to denote equalities and inequalities that hold to within multiplication by two positive constants depending only on the choice of network.
 % while $\bubnormw{P}<\infty$ implies that the bubble diagram diverges at most logarithmically. 
 
\vspace{0.1em}

\begin{theorem}[Mean-field intrinsic exponents]
	\label{thm:transitivemain}
	Let $G$ be a transitive network such that $\bubnorm{P} < \infty$,
	 % (e.g.\ $\Z^d$ for $d\geq 5$), 
	  and let $\F$ be the wired uniform spanning forest of $G$. 
	% Suppose that there exists a positive constant $c$ such that the volume of a ball of radius $n$ in $G$ is at least $c n^5$ for every $n \geq 1$.
	 Then
	\vspace{0.25em}
	\[ 
		\vspace{0.25em}
		\P\left(\diam_{\mathrm{int}}(\fP(v)) \geq R \right)  \asymp R^{-1} 
		\quad \text{ and } \quad
		\P\left(|\fP(v)| \geq R \right) \asymp R^{-1/2}
	\]
	for every vertex $v$ and every $R\geq 1$. In particular, the critical exponents governing the intrinsic diameter and volume of the past are $1$ and $1/2$ respectively. 
\end{theorem}

For transitive \emph{graphs}, it follows from work of Hebisch and Saloff-Coste \cite{HebSaCo93} that $\bubnorm{P}<\infty$ if and only if the graph has at least quintic volume growth, i.e., if and only if there exists a constant $c$ such that the number of points in every graph distance ball of radius $n$ is at least $cn^5$ for every $n\geq 1$. 
Thus, by Gromov's theorem \cite{Gromov1981},  Trofimov's theorem \cite{Trofimov1984}, and the Bass-Guivarc'h formula \cite{Bass1972,Guiv1973}, the class of graphs treated by \cref{thm:transitivemain}  includes all transitive graphs not rough-isometric to $\Z,\Z^2,\Z^3,\Z^4$, or the discrete Heisenberg group. As mentioned above, the theorem also applies for example to long-ranged transitive networks with vertex set $\Z^d$, a single edge between every two vertices, and with translation-invariant conductances given up to positive multiplicative constants by
\[ c\left(\{x,y\}\right) =c\left(x-y\right) \asymp \| x-y\|^{-d-\alpha}
\vspace{0.5em}
\]
provided that either $1 \leq d \leq 4$ and $\alpha \in (0,d/2)$ or $d\geq 5$ and $\alpha>0$. The canonical example of such a network is that associated with the fractional Laplacian $-(-\Delta)^\beta$ of $\Z^d$ for $\beta\in(0,d/4 \wedge 1)$. (See \cite[Section 2]{Slade2017}.)

\medskip

The general form of our result is similar, but has an additional technical complication owing to the need to avoid trivialities that may arise from the local geometry of the network. 
Let $G$ be a network, let $v$ be a vertex of $G$, let $X$ and $Y$ be independent random walks started at $v$, and let $q(v)$ be the probability that $X$ and $Y$ never return to $v$ or intersect each other after time zero.

\begin{theorem}
	\label{thm:generalexponents}
	Let $G$ be a network with controlled stationary measure such that $\bubnorm{P}<\infty$, and let $\F$ be the wired uniform spanning forest of $G$. Then 
	\vspace{0.2em}
	\[ 
	\vspace{0.2em}
	 q(v)  R^{-1}  
			\preceq
		\P\Bigl(\diam_\mathrm{int}\bigl(\fP(v)\bigr) \geq R \Bigr)
			\preceq
		  R^{-1} 
	\]
	and
	\vspace{0.2em}
	\[
	\vspace{0.2em}
		 q(v)^{5/2}  R^{-1/2}
			\preceq
		\P\Bigl(\bigl|\fP(v)\bigr| \geq R \Bigr)
			\preceq
		  R^{-1/2}
	\]
	for all $v\in V$ and $R\geq1$.
\end{theorem}

The presence of $q(v)$ in the theorem is required, for example, in the case that we attach a vertex by a single edge to the origin of $\Z^d$, so that the past of this vertex in the USF is necessarily trivial. (The precise nature of the dependence on $q(v)$ has not been optimized.)
However, in any network $G$ with controlled stationary measure and with $\bubnorm{P}<\infty$, there exist positive constants $\eps$ and $r$ such that for every vertex $v$ in $G$, there exists a vertex $u$ within distance $r$ of $G$ such that $q(u)>\eps$ (\cref{lem:qpositivity}). In particular, if $G$ is a transitive network with $\bubnorm{P}<\infty$ then $q(v)$ is a positive constant, so that \cref{thm:transitivemain} follows from \cref{thm:generalexponents}. 

Let us note that \cref{thm:generalexponents} applies in particular to any bounded degree nonamenable graph, or more generally to any network with controlled stationary measures satisfying a $d$-dimensional isoperimetric inequality for some $d>4$, see \cite[Theorem 3.2.7]{KumagaiBook}. In particular, it applies to $\Z^d$, $d\geq 5$, with any specification of edge conductances bounded above and below by two positive constants (in which case it can also be shown that $q(v)$ is bounded below by a positive constant). A further example to which our results are applicable is given by taking $G=H^d$ where $d\geq 5$ and $H$ is \emph{any} infinite, bounded degree graph.

% The extreme generality of the above theorems is a strong vindication of the predictions of universality and of mean-field theory for high dimensional models. The scope of application is far in excess of that which has been proven for any other statistical mechanics model thus far.

\subsection{Volume growth, spectral dimension, and anomalous diffusion}

The theorems concerning intrinsic exponents stated in the previous subsection also allow us to determine exponents describing the almost sure asymptotic geometry of the trees in the WUSF, and in particular allow us to compute the almost sure spectral dimension and walk dimension of the trees in the forest. 
See e.g.\ \cite{KumagaiBook} for background on these and related concepts.
Here, we always consider the trees of the WUSF as \emph{graphs}. 
One could instead consider the trees as networks with conductances inherited from $G$, and the same results would apply with minor modifications to the proofs. 

Let $G$ be an infinite, connected network and let $v$ be a vertex of $G$.
We define the \textbf{volume growth dimension} (a.k.a.\ \textbf{fractal dimension}) of $G$ to be
\begin{align*}
d_f(G) &:= \lim_{n\to\infty} \frac{\log |B(v,n)|}{\log n} &\qquad \text{ when this limit exists,}
\intertext{define the \textbf{spectral dimension} of $G$ to be}
d_s(G) &:= \lim_{n\to\infty} \frac{-2\log p_{2n}(v,v)}{\log n} &\qquad \text{ when this limit exists,}
\intertext{and define the \textbf{walk dimension}
 of $G$ to be}
d_w(G) &:= \lim_{n\to\infty} \frac{\log n}{\log \bE_v \max_{1\leq m \leq n} d(v,X_n)} &\qquad \text{ when this limit exists.}
\end{align*}
 In each case, the limit used to define the dimension 
 % \footnote{The reason $d_w$ is known as a `dimension' is that, heuristically, it measures the dimension of the graph of the walk $\{(n,X_n):n\geq 0\}\subseteq \N \times G$.}
  does not depend on the choice of root vertex $v$.  Our next theorem establishes the values of $d_f,d_s,$ and $d_w$ for the trees in the WUSF under the assumption that $\bubnorm{P}<\infty$. The results concerning $d_s$ and $d_w$ are new even in the case of $\Z^d$, $d\geq 5$, while the result concerning the volume growth was established for  $\Z^d$, $d\geq 5$, by Barlow and J\'arai \cite{barlow2016geometry}.

\begin{theorem}
\label{thm:AlexanderOrbach}
  Let $G$ be a network with controlled stationary measure and with $\bubnorm{P}<\infty$, 
  % let $P$ be the Markov operator on $G$, 
   % satisfying a uniform $d$-dimensional  on-diagonal heat kernel estimate \eqref{eq:heatkernel} for some $d>4$,
    and let $\F$ be the wired uniform spanning forest of $G$. Then almost surely, for every component $T$ of $\F$, the volume growth dimension, spectral dimension, and walk dimension of $T$ satisfy
    \[
d_f(T)=2, \quad d_s(T)=\frac{4}{3}, \quad \text{and} \quad d_w(T)=3.
    \]
     In particular, the limits defining these quantities are well-defined almost surely.
\end{theorem}

The values $d_f=2,d_s=4/3,$ and $d_w=3$ are known as the \emph{Alexander-Orbach} values of these exponents, following the conjecture due to Alexander and Orbach \cite{AlexOrb1982} that they held for high-dimensional incipient infinite percolation clusters. The first rigorous proof of Alexander-Orbach behaviour was due to Kesten \cite{kestensubdiff}, who established it for critical Galton-Watson trees conditioned to survive (see also \cite{BarKum06}). The first proof for a model in Euclidean space was due to Barlow, J\'arai, Kumagai, and Slade \cite{BJKS08}, who established it for high-dimensional incipient infinite clusters in \emph{oriented} percolation. Later, Kozma and Nachmias \cite{KozNach09} established the Alexander-Orbach conjecture for high-dimensional \emph{unoriented} percolation. 
See \cite{MR3206998} for an extension to long-range percolation, \cite{KumagaiBook} for an overview, and \cite{1609.03980} for  results regarding scaling limits of a related model. 

As previously mentioned, Barlow and Masson \cite{BarMass11} have shown that in the \emph{two-dimensional} uniform spanning tree, $d_f=8/5,$ $d_s=16/13$, and $d_w=13/5$.

\subsection{Extrinsic exponents}
\label{subsec:introextrinsic}

We now describe our results concerning the \emph{extrinsic} diameter of the past.  In comparison to the intrinsic diameter, our methods to study the extrinsic diameter are more delicate and require stronger assumptions on the graph in order to derive sharp estimates. Our first result on the extrinsic diameter concerns $\Z^d$, and improves upon the results of Bhupatiraju, Hanson, and J\'arai \cite{bhupatiraju2016inequalities} by removing the polylogarithmic errors present in their results.

\begin{theorem}[Mean-field Euclidean extrinsic diameter]
\label{thm:extrinsicZd}
Let $d\geq 5$, and
 % let $G$ be a network with vertex set $\Z^d$ and conductances and 
 let $\F$ be the wired uniform spanning forest of $\Z^d$. 
	 Then
	\vspace{0.25em}
	\[ 
		\vspace{0.25em}
		\P\left(\diam_{\mathrm{ext}}(\fP(0)) \geq R \right)  \asymp R^{-2}
		% 
		% \quad \text{ and } \quad
		% 
		% \P\left(|\fP(v)| \geq R \right) \asymp R^{-1/2}
	\]
	for every $R\geq 1$. 
\end{theorem}

We expect that it should be possible to generalize the proof of \cref{thm:extrinsicZd} to other similar graphs, and to long-range models, but we do not pursue this here. 

On the other hand, if we are unconcerned about polylogarithmic errors, it is rather straightforward to deduce various estimates on the extrinsic diameter from the analogous estimates on the intrinsic diameter, \cref{thm:transitivemain,thm:generalexponents}. The following is one such estimate of particular interest. For notational simplicity we will always work with the graph metric, although our methods easily adapt to various other metrics.
We say that a network $G$ with controlled stationary measure is \textbf{$d$-Ahlfors regular} if there exist positive constants $c$ and 
$C$ such that $ cn^d \leq B(v,n) \leq Cn^d$ for every vertex $v$ and $n\geq 1$. We say that $G$ satisfies \textbf{Gaussian heat kernel estimates} if there exist positive constants $c,c'$ such that
\begin{align} \frac{c}{|B(x,n^{1/2})|}e^{- d(x,y)^2/(cn)} \leq p_n(x,y) + p_{n+1}(x,y) \leq \frac{c'}{|B(x,n^{1/2})|}e^{-d(x,y)^2/(c' n)} 
\label{eq:GHKE}
\end{align}
for every $n\geq 0$ and every pair of vertices $x,y$ in $G$ with $d(x,y)\leq n$. It follows from the work of Hebisch and Saloff-Coste \cite{HebSaCo93} that every transitive graph of polynomial volume growth satisfies Gaussian heat-kernel estimates, as does every bounded degree  network with edge conductances bounded between two positive constants that is rough-isometric to a transitive graph of polynomial growth.

\begin{theorem}
\label{thm:extrinsic}
Let $G$ be a network with controlled stationary measure that is $d$-Ahlfors regular for some $d>4$  and that satisfies Gaussian heat kernel estimates.
% vertex set $\Z^d$ and conductances and let $\F$ be the wired uniform spanning forest of $\Z^d$. 
   Then
  \vspace{0.25em}
  \[ 
    \vspace{0.25em}
    q^2(v) R^{-2} \preceq \P\left(\diam_{\mathrm{ext}}(\fP(v)) \geq R \right)  \preceq R^{-2}\log R
    % 
    % \quad \text{ and } \quad
    % 
    % \P\left(|\fP(v)| \geq R \right) \asymp R^{-1/2}
  \]
  for every vertex $v$ and every $R\geq 1$. 
\end{theorem}

Note that the hypotheses of this theorem imply that $\bubnorm{P}<\infty$.

Finally, we consider networks in which the random walk is ballistic rather than diffusive. We say that a network $G$ is \textbf{uniformly ballistic} if there exists a constant $C$ such that
\begin{equation}
\label{eq:speedcondition}
\sup_{v\in V}\bE_v\left[
 \sup\left\{n\geq 0: d(v,X_n) \leq r\right\}
 % \sum_{n\geq 0} \mathbbm{1}\Bigl(X_n \in B(v,r)\Bigr)
 \right] \leq C r
\end{equation}
for every $r\geq 1$. 
% An observation of Lyons, Peres, and Sun \cite[Lemma 2.1]{lyons2017occupation} shows that a transitive network is uniformly ballistic if and only if $\liminf_{t\to\infty}\frac{1}{t}d(X_0,X_t)>0$ almost surely when $X$ is a random walk on the network. 
Every nonamenable network with bounded degrees and edge conductances bounded above is uniformly ballistic, as can be seen from the proof of \cite[Proposition 6.9]{LP:book}.

\begin{theorem}[Extrinsic diameter in the positive speed case]
\label{thm:extrinsicspeed}
Let $G$ be a uniformly ballistic network with controlled stationary measure and $\bubnorm{P}<\infty$, and let $\fF$ be the wired uniform spanning forest of $G$. 
   Then
  \vspace{0.25em}
  \[ 
    \vspace{0.25em}
    q^2(v)R^{-1} \preceq \P\left(\diam_{\mathrm{ext}}(\fP(v)) \geq R \right)  \preceq R^{-1}
    % 
    % \quad \text{ and } \quad
    % 
    % \P\left(|\fP(v)| \geq R \right) \asymp R^{-1/2}
  \]
  for every vertex $v$ and every $R\geq 1$. 
\end{theorem}

Note that the \emph{upper bound} of \cref{thm:extrinsicspeed} is a trivial consequence of \cref{thm:generalexponents}.

\subsection{Applications to the Abelian sandpile model}
\label{subsec:sandpileintro}

% The \textbf{Abelian sandpile model} was introduced by the physicist Dhar as a paradigmatic example of a model exhibiting \emph{self-organized criticality}.

The \textbf{Abelian sandpile model} was introduced by Dhar \cite{Dhar90} as an analytically tractable example of a system exhibiting \emph{self-organized criticality}. This is the phenomenon by which certain randomized dynamical systems tend to exhibit critical-like behaviour at equilibrium despite being defined without any parameters that can be varied to produce a phase transition in the traditional sense. The concept of self-organized criticality was first posited in the highly influential work of Bak, Tang, and Wiesenfeld \cite{bak1987self,bak1988self}, who proposed (somewhat  controversially \cite{Watkins2016}) that it may account for the occurrence of complexity, fractals, and power laws in nature. See \cite{MR3857602} for a detailed introduction to the Abelian sandpile model, and \cite{jensen1998self} for a discussion of self-organized criticality in applications.

We now define the Abelian sandpile model.
Let $G=(V,E)$ be a connected, locally finite graph and let $K \subseteq V$ be a set of vertices.  A \textbf{sandpile} on $K$ is a function $\eta: K \to \{0,1,\ldots\}$, which we think of as a collection of indistinguishable particles (grains of sand) located on the vertices of $K$. We say that $\eta$ is \textbf{stable} at a vertex $x$ if $\eta(x) < \deg(x)$, and otherwise that $\eta$ is \textbf{unstable} at $x$. We say that $\eta$ is stable if it is stable at every $x$, and that it is unstable otherwise. If $\eta$ is unstable at $x$, we can \textbf{topple} $\eta$ at $x$ to obtain the sandpile $\eta'$ defined by
\[
\eta'(y) = \begin{cases} \eta(x) - \deg(x) & y=x\\
\eta(y)+\#\{\text{edges between $x$ and $y$}\} & y \neq x
\end{cases}
\]
for all $y\in K$.
That is, when $x$ topples, $\deg(x)$ of the grains of sand at $x$ are redistributed to its neighbours, and grains of sand redistributed to neighbours of $x$ in $V\setminus K$ are lost. Dhar~\cite{Dhar90} observed that if $K$ is finite and not equal to $V$ then carrying out successive topplings will eventually result in a stable configuration and, moreover, that the stable configuration obtained in this manner does not depend on the order in which the topplings are carried out.
 (This property justifies the model's description as \emph{Abelian}.)

We define a Markov chain on the set of stable sandpile configurations on $K$ as follows: At each time step, a vertex of $K$ is chosen uniformly at random,  an additional grain of sand is placed at that vertex, and the resulting configuration is stabilized. Although this Markov chain is \emph{not} irreducible, it can be shown that chain has a unique closed communicating class, consisting of the \emph{recurrent} configurations, and that the stationary measure of the Markov chain is simply the uniform measure on the set of recurrent configurations. In particular, the stationary measure for the Markov chain is also stationary if we add a grain of sand to a \emph{fixed} vertex and then stabilize \cite[Exercise 2.17]{MR3857602}.  

The connection between sandpiles and spanning trees was first discovered by Majumdar and Dhar \cite{MajDhar92}, who described a bijection, known as the \emph{burning bijection}, between recurrent sandpile configurations and spanning trees. Using the burning bijection, Athreya and J\'arai \cite{MR2077255} showed that if $d\geq 2$ and $\langle V_n\rangle_{n\geq 1}$ is an exhaustion of $\Z^d$ by finite sets, then the uniform measure on recurrent sandpile configurations on $V_n$ converges weakly as $n\to\infty$ to a limiting measure on sandpile configurations on $\Z^d$. J\'arai and Werning \cite{JarWer14} later extended this result to any infinite, connected, locally finite graph $G$ for which every component of the WUSF of $G$ is one-ended almost surely. We call a random sandpile configuration on $G$ drawn from this measure a \textbf{uniform recurrent sandpile} on $G$, and typically denote such a random variable by $\Eta$ (capital $\eta$).

We are particularly interested in what happens during one step of the dynamics at equilibrium, in which one grain of sand is added to a vertex $v$ in a uniformly random recurrent configuration $\Eta$, and then topplings are performed in order to stabilize the resulting configuration. The multi-set of vertices counted according to the number of times they topple is called the \textbf{Avalanche}, and is denoted $\Av_v(\Eta)$. The \emph{set} of vertices that topple at all is called the \textbf{Avalanche cluster} and is denoted by $\AvC_v(\Eta)$. 

J\'arai and Redig \cite{JarRed08} showed that the burning bijection allows one to relate avalanches to the past of the WUSF, which allowed them to prove that avalanches in $\Z^d$ satisfy $\P( v \in \AvC_0(\Eta)) \asymp \|v\|^{-d+2}$ for $d\geq 5$. (The fact that the \emph{expected} number of times $v$ topples scales this way is an immediate consequence of Dhar's formula, see \cite[Section 3.3.1]{MR3857602}.)  Bhupatiraju, Hanson, and J\'arai \cite{bhupatiraju2016inequalities} built upon these methods to prove that, when $d\geq 5$, the probability that the diameter of the avalanche is at least $n$ scales as $n^{-2} \log^{O(1)} n$ and the probability that the total number of topplings in the avalanche is at least $n$ is between $c n^{-1/2}$ and $n^{-2/5+o(1)}$. Using the combinatorial tools that they developed, the following theorem, which improves upon theirs, follows straightforwardly from our results concerning the WUSF. (Strictly speaking, it also requires our results on the $v$-WUSF, see \cref{subsec:vWUSF}.)
 % If this causes the configuration to become unstable, then we perf

\begin{theorem}
\label{thm:sandpile}
Let $d\geq 5$ and let $\Eta$ be a uniform recurrent sandpile on $\Z^d$. Then 
% \begin{align*}
% \P\bigl(\operatorname{diam}_\ext\left(\AvC_0(\Eta)\right) \geq n\bigr) &\asymp n^{-2},\\
% \P\bigl(|\!\AvC_0(\Eta)| \geq n\bigr) &\asymp n^{-1/2},\\
% \P\bigl(|\!\Av_0(\Eta)| \geq n\bigr)&\asymp n^{-1/2}
% \end{align*}
\begin{align*}
\P\bigl(\operatorname{diam}_\ext\left(\AvC_0(\Eta)\right) \geq n\bigr) \asymp n^{-2} \quad \text{ and } \quad
\P\bigl(|\!\AvC_0(\Eta)| \geq n\bigr) \asymp 
\P\bigl(|\!\Av_0(\Eta)| \geq n\bigr)\asymp n^{-1/2}
\end{align*}
for all $n\geq 1$.
\end{theorem}

As with the WUSF, our methods also yield several variations on this theorem for other classes of graphs, the following of which are particularly notable. See \cref{subsec:introextrinsic} for the relevant definitions. With a little further work, it should be possible to remove the dependency on $v$ in the lower bounds of \cref{thm:sandpilepolynomial,thm:sandpilenonamenable}. The \emph{upper bounds} of \cref{thm:sandpilepolynomial} only require that $G$ has polynomial growth, see \cref{prop:polyupperext}.

\begin{theorem}
\label{thm:sandpilepolynomial}
Let $G$ be a bounded degree graph that is $d$-Ahlfors regular for some $d>4$  and that satisfies Gaussian heat kernel estimates, and let $\Eta$ be a uniform recurrent sandpile on $G$. Then
\hspace{1cm}
\[
\begin{array}{llcll}
% \begin{align*}
q(v)^2n^{-2} &\preceq &\P\Bigl(\operatorname{diam}_\ext\left(\AvC_v(\Eta)\right) \geq n\Bigr) &\preceq &n^{-2}\log n,\vspace{.25em}\\
% \item
q(v)^{5/2} n^{-1/2} &\preceq &\P\Bigl(|\AvC_v(\Eta)| \geq n\Bigr) &\preceq & n^{-1/2}\log^{1/2}n,\vspace{.25em}\\
q(v)^{5/2} n^{-1/2} &\preceq &\P\Bigl(|\Av_v(\Eta)| \geq n\Bigr) &\preceq & n^{-1/2}\log^{1/2}n
 %\text{ and } \vspace{.25em}
\end{array}
\]
for all $n\geq 1$.
\end{theorem}

Similarly, the following theorem concerning uniformly ballistic graphs can be deduced from \cref{thm:extrinsicspeed,thm:extrinsicspeedv}. 
% we have the following theorem in the uniformly ballistic context, which applies in particular to any bounded degree nonamenable graph.
Again, we stress that this result applies in particular to any bounded degree nonamenable graph. 
\begin{theorem}
\label{thm:sandpilenonamenable}
Let $G$ be a bounded degree, uniformly ballistic graph such that $\bubnorm{P}<\infty$, and let $\Eta$ be a uniform recurrent sandpile on $G$. Then
% \begin{align*}
% \P\Bigl(\operatorname{diam}_\ext\left(\AvC_0(\Eta)\right) \geq n\Bigr) &\asymp n^{-1} && n \nearrow\infty\\
% \P\Bigl(|\AvC_0(\Eta)| \geq n\Bigr) &\asymp n^{-1/2} && n \nearrow \infty\\
% \P\Bigl(|\Av_0(\Eta)| \geq n\Bigr)&\asymp n^{-1/2} && n \nearrow \infty.
% \end{align*}
\[
\begin{array}{llcll}
% \begin{align*}
q(v)^2n^{-1} &\preceq &\P\Bigl(\operatorname{diam}_\ext\left(\AvC_v(\Eta)\right) \geq n\Bigr) &\preceq &n^{-1},\vspace{.25em}\\
% \item
q(v)^{5/2} n^{-1/2} &\preceq &\P\Bigl(|\AvC_v(\Eta)| \geq n\Bigr) &\preceq & n^{-1/2},\vspace{.25em}\\
q(v)^{5/2} n^{-1/2} &\preceq &\P\Bigl(|\Av_v(\Eta)| \geq n\Bigr) &\preceq & n^{-1/2}
 %\text{ and } \vspace{.25em}
\end{array}
\]
for all $n\geq 1$.
\end{theorem}

\subsection*{Notation}
As previously discussed,  we use $\asymp, \preceq$ and $\succeq$ to denote equalities and inequalities that hold to within multiplication by two positive constants depending only on the choice of network.
Typically, but not always, these constants will only depend on a few important parameters such as $\inf_{v\in V} c(v),$  $\sup_{v\in V}c(v),$ and $\bubnorm{P}$. 

For the reader's convenience, we gather here several pieces of notation that will be used throughout the paper. Each such piece of notation is also defined whenever it first appears within the body of the paper. In particular, the $v$-wired uniform spanning forest is defined in \cref{subsec:vWUSF} and the interlacement and $v$-wired interlacement processes are defined in \cref{sec:Interlacements}.
\begin{itemize}[leftmargin=4cm]
  \item[$\F,\F_v$] A sample of the wired uniform spanning forest and $v$-wired uniform spanning forest respectively.
  \item[$\fT_v$] The tree containing $v$ in $\F_v$.
  % \item[$\F_v$] A sample of the $v$-wired uniform spanning forest.
  % \item[$\fP_v(u)$] The past of $u$ in $\F_v$.
  \item[$\fB(u,n),\fB_v(u,n)$] The intrinsic ball of radius $n$ around $u$ in $\F$ and $\F_v$ respectively.
   \item[$\partial \fB(u,n),\partial \fB_v(u,n)$] The set of vertices at distance exactly $n$ from $u$ in $\F$ and $\F_v$ respectively.
    \item[$\fP(u),\fP_v(u)$] The past of $u$ in $\F$ and $\F_v$ respectively.
      \item[$\past_{F}(u)$] The past of $u$ in the oriented forest $F$ (which need not be spanning).
     \item[$\fP(u,n),\fP_v(u,n)$] The intrinsic ball of radius $n$ around $u$ in the past of $u$ in $\F$ and $\F_v$ respectively.
     \item[$\Gamma(u,w), \Gamma_v(u,w)$] The path from $u$ to $w$ in $\F$ and $\F_v$ respectively, should these vertices be in the same component.
     \item[$\Gamma(u,\infty), \Gamma_v(u,\infty)$] The future of $u$ in $\F$ and $\F_v$ respectively.
      \item[$\partial \fP(u,n),\partial \fP_v(u,n)$] The set of vertices with intrinsic distance exactly $n$ from $u$ in the past of $u$ in $\F$ and $\F_v$ respectively.
        \item[$\sI,\sI_v$] The interlacement process and $v$-wired interlacement process respectively.
  \item[$\cI_{[a,b]},\cI_{v,[a,b]}$] The set of vertices visited by the interlacement process and the $v$-wired interlacement process in the time interval $[a,b]$ respectively.
  % \item[$\fB_v(u,n)$] The intrinsic ball of radius $n$ around $u$ in $\F_v$/ 
  % \item[$\Cap(K)$] The capacity in $G$ of the set of vertices $K$.
\end{itemize}
\section{Background}

% \subsection{The wired uniform spanning forest}

\subsection{Loop-erased random walk and Wilson's algorithm}

Let $G$ be a network. For each $-\infty \leq n \leq m \leq \infty$ we define $L(n,m)$ to be the line graph with vertex set $\{i \in \Z : n \leq i\leq m\}$ and with edge set $\{\{i,i+1\}: n\leq i \leq m-1\}$. A \textbf{path} in $G$ is a multigraph homomorphism from $L(n,m)$ to $G$ for some $-\infty \leq n \leq m \leq \infty$. We can consider the random walk on $G$ as a path by keeping track of the edges it traverses as well as the vertices it visits. Given a path $w : L(n,m)\to G$ we will use  $w(i)$ and $w_i$ interchangeably to denote the  vertex visited by $w$ at time $i$, and use $w(i,i+1)$ and $w_{i,i+1}$ interchangeably to denote the oriented edge crossed by $w$ between times $i$ and $i+1$.

Given a path in $w : L(0,m) \to G$ for some $m\in [0,\infty]$ that is transient in the sense that it visits each vertex at most finitely many times, we define the sequence of times $\ell_n(w)$ by $\ell_0(w)=0$ and $\ell_{n+1}(w) = 1+\max\{ k : w_k = w_{\ell_n}\}$, terminating the sequence when $\max\{ k : w_k = w_{\ell_n}\}=m$ in the case that $m<\infty$. The \textbf{loop-erasure}  $\LE(w)$ of $w$ is the path defined by
\[\LE(w)_i = w_{\ell_i(w)} \qquad \LE(w)_{i,i+1} = w_{\ell_{i+1}-1,\ell_{i+1}}.\]
In other words, $\LE(w)$ is the path formed by erasing cycles from $w$ chronologically as they are created. The loop-erasure of simple random walk is known as \textbf{loop-erased random walk} and was first studied by Lawler \cite{Lawler80}.

\textbf{Wilson's algorithm} \cite{Wilson96} is a method of sampling the UST of a finite graph by recursively joining together loop-erased random walk paths. It was extended to sample the WUSF of infinite transient graphs by Benjamini, Lyons, Peres, and Schramm \cite{BLPS}. See also \cite[Chapters 4 and 10]{LP:book} for an overview of the algorithm and its applications.

Wilson's algorithm can be described in the infinite transient case as follows. Let $G$ be an infinite transient network, and let $v_1,v_2,\ldots$ be an enumeration of the vertices of $G$. Let $\F^{0}$ be the empty forest, which has no vertices or edges.
 Given $\F^{n}$ for some $n\geq 0$, start a random walk at $v_{n+1}$. Stop the random walk if and when it hits the set of vertices already included in $\F^n$, running it forever otherwise. Let $\F^{n+1}$ be the union of $\F^n$ with the set of edges traversed by the loop-erasure of this stopped path.
Let $\F=\bigcup_{n\geq 0} \F^n$.
Then the random forest $\F$ has the law of the wired uniform spanning forest of $G$. If we keep track of direction in which edges are crossed by the loop-erased random walks when performing Wilson's algorithm, we obtain the oriented wired uniform spanning forest. 
The algorithm works similarly in the finite and recurrent cases, except that we start by taking $\F^{0}$ to contain one vertex and no edges.

\subsection{The $v$-wired uniform spanning forest and stochastic domination}
\label{subsec:vWUSF}

In this section we introduce the $v$-wired uniform spanning forest ($v$-WUSF), which was originally defined by J\'arai and Redig \cite{JarRed08} in the context of their work on the sandpile model (where it was called the WSF$_o$). The $v$-WUSF is a variation of the WUSF of $G$ in which, roughly speaking, we consider $v$ to be `wired to infinity'. The $v$-WUSF serves two useful purposes in this paper: its stochastic domination properties allow us to ignore interactions between different parts of the WUSF, and the control of the $v$-WUSF that we obtain will be applied to prove our results concerning the Abelian sandpile model in \cref{sec:sandpile}.

Let $G$ be an infinite network and let $v$ be a vertex of $G$.  Let $\langle V_n \rangle_{n\geq 1}$ be an exhaustion of $G$ and, for each $n\geq1$, let $G^{*v}_n$ be the graph obtained by identifying $v$ with $\partial_n$ in the graph $G^*_n$. The measure $\WUSF_v$ is defined to be the weak limit
\[\WUSF_v(S \subset \F; G) = \lim_{n\to\infty}\UST(S \subset T ; G^{*v}_n).\]
The fact that this limit exists and does not depend on the choice of exhaustion of $G$ is proved similarly to the corresponding statement for the WUSF, see \cite{LMS08}.  As with the WUSF, we can also define the \textbf{oriented $v$-wired uniform spanning forest} by orienting the uniform spanning tree of $G_n^{*v}$ towards $\partial_n$ (which is identified with $v$) at each step of the exhaustion before taking the weak limit. 
It is possible to sample the $v$-WUSF by running Wilson's algorithm rooted at infinity, but starting with $\F^{0}_v$ as the forest that has vertex set $\{v\}$ and no edges (as we usually would in the finite and recurrent cases). Moreover, if we orient each edge in the direction in which it is crossed by the loop-erased random walk when running Wilson's algorithm, we obtain a sample of the oriented $v$-WUSF.

The following lemma makes the $v$-WUSF extremely useful for studying the usual WUSF, particularly in the mean-field setting. It will be the primary means by which we ignore the interactions between different parts of the forest. (Indeed, it plays a role analogous to that played by the BK inequality in Bernoulli percolation.) We denote by $\past_F(v)$ the past of $v$ in the oriented forest $F$, which need not be spanning. We write $\fT_v$ for the tree containing $v$ in $\F_v$, and write $\Gamma(u,\infty)$ and $\Gamma_v(u,\infty)$ for the future of $u$ in $\F$ and $\F_v$ respectively, as defined in \cref{subsec:introintrinsic}.

\begin{lemma}[Stochastic Domination] 
\label{lem:domination}
Let $G$ be an infinite network,  let $\F$ be an oriented wired uniform spanning forest of $G$, and for each vertex $v$ of $G$ let $\F_v$ be an oriented $v$-wired uniform spanning forest of $G$.
% , and let 
Let $K$ be a finite set of vertices of $G$, and define $F(K) =\bigcup_{u\in K} \Gamma(u,\infty)$ and $F_v(K) =\bigcup_{u\in K} \Gamma_v(u,\infty)$. Then for every $u\in K$ and every increasing event $\sA \subseteq \{0,1\}^E$ we have that 
% for every $u,v \in V$, every simple path $\gamma$ from $u$ to $v$, every vertex $w$ in the trace of $\gamma$, and every 
\begin{align}\P\Bigl( \past_{\F \setminus F(K)}(u) \in \sA \mid F(K) \Bigr) &\leq \P\bigl(\fT_u \in \sA \bigr),
\label{eq:dom1}
% \]
\intertext{and similarly}
% \[
\P\Bigl( \past_{\F_v\setminus F_v(K)}(u) \in \sA \mid F_v(K) \Bigr) &\leq \P\bigl(\fT_u \in \sA \bigr).
\label{eq:dom2}
\end{align}
% \]
% \intertext{whenever the events being conditioned on have positive probability.
% Similarly, if 
% $\Gamma(v,\infty)$ is the future of $v$ in $\F$, then}
% \P\Bigl( \past_{\F\setminus \Gamma(v,\infty)}(u) \in \sA \mid \Gamma(v,\infty) \Bigr) &\leq  \P\bigl(\past_{\F_u}(u) \in \sA \bigr).
% \end{align}
\end{lemma}

Note that when $K$ is a singleton, \eqref{eq:dom1} follows implicitly from \cite[Lemma 2.3]{LMS08}. The proof in the general case is also very similar to theirs, but we include it for completeness. Given a network $G$ and a finite set of vertices $K$, we write $G/K$ for the network formed from $G$ by identifying all the vertices of $K$.

\begin{lemma}
\label{lem:domaux}
Let $G$ be a finite network, let $K_1 \subseteq K_2$ be sets of vertices of $G$. For each spanning tree $T$ of $G$, let $S(T,K_2)$ be the smallest subtree of $T$ containing all of $K_2$. Then the uniform spanning tree of $G/K_1$ stochastically dominates $T \setminus S(T,K_2)$, where $T$ is a uniform spanning tree of $G$.
\end{lemma}

% \begin{proof}
% % The case $|K|=1$ is trivial, and \cite[Lemma 2.3]{LMS08} establishes the case $|K|=2$. We will prove the claim by induction on $|K|$. 
% % 
% % Let $|K|\geq 3$, let $v\in K$ and let $K'=K\setminus\{v\}$.  

\begin{proof}
It follows from the spatial Markov property of the UST that, conditional on $S(T,K_2)$, the complement $T\setminus S(T,K_2)$ is distributed as the UST of the network $G/S(T,K_2)$ constructed from $G$ by identifying all the vertices in the tree $S(T,K_2)$, see \cite[Section 2.2.1]{HutNach15b}. On the other hand, it follows from the negative association property of the UST \cite[Theorem 4.6]{LP:book} that if $A \subseteq B$ are two sets of vertices, then the UST of $G / A$ stochastically dominates the UST of $G / B$. This implies that the claim holds when we condition on $S(T,K_2)$, and we conclude by averaging over the possible choices of $S(T,K_2)$.
\end{proof}

\begin{proof}[Proof of \cref{lem:domination}]
The claim \eqref{eq:dom1}
 follows from \cref{lem:domaux} by considering the finite networks $G_n^*$ used in the definition of the WUSF, taking $K_1 = \{u,\partial_n\}$ and $K_2 = K \cup \{\partial_n\}$, and taking the limit as $n\to\infty$. 

 We now prove \eqref{eq:dom2}. 
If $u=v$ then the claim follows by applying \cref{lem:domaux} to the finite networks $G_n^{*v}$, taking $K_1 = \emptyset$ and $K_2 = K$, and taking the limit as $n\to\infty$. Now suppose that $u\neq v$. 
 Let  $G/\{u,v\}$ be the network obtained from $G$ by identifying $u$ and $v$ into a single vertex $x$, and let $\F'$ be the $x$-wired uniform spanning forest of $G/\{u,v\}$. We consider $\F'$ as a subgraph of $G$, and let $\mathfrak{T}'$ be the component of $u$ in $\F'$. It follows from the negative association property of the UST and an obvious limiting argument that $\F'$ is stochastically dominated by $\F_u$, and hence that $\mathfrak{T}'$ is stochastically dominated by $\mathfrak{T}_u$. On the other hand, applying 
\cref{lem:domaux} to the finite networks $G_n^{*v}$, taking $K_1 = \{u,v\}$ and $K_2 = K \cup \{v\}$, and taking the limit as $n\to\infty$ 
% shows that . Now suppose that $u\neq v$. Apply the same argument but with $G_n^{*v}$ instead of $G_n^*$
   yields that the conditional distribution of $\past_{\F_v\setminus F_v(K)}(u)$ given $F_v(K)$ is stochastically dominated by $\mathfrak{T}'$ and hence by $\mathfrak{T}_u$.
\end{proof}

\section{Interlacements and the Aldous-Broder algorithm}
\label{sec:Interlacements}

The \textbf{random interlacement process} is a Poissonian soup of doubly-infinite random walks that was introduced by Sznitman \cite{Szni10} and generalized to arbitrary transient graphs by Texeira \cite{Teix09}. Formally, the interlacement process $\sI$ on the transient graph $G$ is a Poisson point process on $\cW^* \times \R$, where $\cW^*$ is the space of bi-infinite paths in $G$ modulo time-shift, and $\R$ is thought of as a time coordinate. In \cite{hutchcroft2015interlacements}, we showed that the random interlacement process can be used to generate the WUSF via a generalization of the Aldous-Broder algorithm. By shifting the time coordinate of the interlacement process, this sampling algorithm also allows us to view the WUSF as the stationary measure of a Markov process; this dynamical picture of the WUSF, or more precisely its generalization to the $v$-WUSF, is of central importance to the proofs of the main theorems of this paper.

We now begin to define these notions formally.
We must first define the space of trajectories $\cW^*$.
Recall that for each $-\infty \leq n \leq m \leq \infty$ we define $L(n,m)$ to be the line graph with vertex set $\{i \in \Z : n \leq i\leq m\}$ and with edge set $\{\{i,i+1\}: n\leq i \leq m-1\}$. Given a graph $G$, we define $\cW(n,m)$ to be the set of multigraph homomorphisms from $L(n,m)$ to $G$ that are transient in the sense that the preimage of each vertex of $G$ is finite. We define the set $\cW$ to be the union
\[ \cW:= \bigcup\left\{\cW(n,m): -\infty \leq n \leq m \leq \infty\right\}.\]
 The set $\cW$ can be made into a Polish space in such a way that local times at vertices and first and last hitting times of finite sets are continuous, see \cite[Section 3.2]{hutchcroft2015interlacements}.
We define the \textbf{time shift}  $\theta_k:\cW\to \cW$ by
$\theta_k : \cW(n,m) \longrightarrow \cW(n-k,m-k)$,
\[ \theta_k(w)(i)=w(i+k), \quad  \theta_k(w)(i,i+1)= w(i+k,i+k+1),\]
and define the space $\cW^*$ to be the quotient 
\[\cW^* = \cW / \sim\,, \text{ where } w_1\sim w_2 \text{ if and only if } w_1 = \theta_k (w_2) \text{ for some k}.\]
Let $\pi : \cW \to \cW^*$ denote the associated quotient function. We equip the set $\cW^*$ with the quotient topology (which is Polish) and associated Borel $\sigma$-algebra. 
 An element of $\cW^*$ is called a \textbf{trajectory}. 

We now define the intensity measure of the interlacement process. 
Let $G$ be a transient network. 
Given $w \in \cW(n,m)$, let $w^\leftarrow \in \cW(-m,-n)$ be the reversal of $w$, which is defined by setting $w^\leftarrow(i)=w(-i)$ for all $-m \leq i \leq -n$ and setting $w^\leftarrow(i,i+1)=w(-i,-i-1)$ for all $-m \leq i \leq -n-1$. For each subset $\sA \subseteq \cW$, let $\sA^\leftarrow$ denote the set 
\[\sA^\leftarrow:= \{w \in \cW: w^\leftarrow \in \sA \}.\]
For each set $K \subseteq V$, we let $\cW_K(n,m)$ be the set of $w\in \cW(n,m)$ such that there exists $n\leq i\leq m$ such that $w(i)\in K$, and similarly define $\cW_K$ to be the union $\cW_K=\bigcup \{\cW_K(n,m): -\infty \leq n \leq m \leq \infty\}$. Let $\tau^+_K$ be the first positive time that the walk visits $K$, where we set $\tau^+_K=\infty$ if the walk does not visit $K$ at any positive time. We 
define a measure $Q_K$ on $\cW_K$ by setting \[Q_K(\{w\in \cW: w(0)\notin K\})=0\] and, for each $u\in K$ and each two Borel subsets $\sA,\sB\subseteq \cW$,  
\begin{multline*}
Q_K\left( \{w \in \cW: w|_{(-\infty,0]} \in \sA,\, w(0) = u \text{ and } w|_{[0,\infty)} \in \sB \}\right)\\= 
 c(u)\mathbf{P}_u\big( X \in \sA^\leftarrow \text{ and } \tau_K^+ =\infty\big) \mathbf{P}_u\big(X \in \sB \big).
\end{multline*}
For each set $K \subseteq V$, let $\cW^*_K= \pi(W_K)$ be the set of trajectories that visit $K$.
It follows from the work of Sznitman \cite{Szni10} and Teixeira \cite{Teix09} that there exists a unique $\sigma$-finite measure $Q^\ast$ on $\cW^*$ such that for every Borel set $\sA \subseteq \cW^*$ and every finite $K\subset V$,
\begin{equation}\label{eq:ildefn} Q^*(\sA \cap \cW^*_K) =  Q_K\left(\pi^{-1}(\sA)\right). \end{equation}
We refer to the unique such measure $Q^*$ as the  \textbf{interlacement intensity measure}, and define the \textbf{random interlacement process} $\sI$ to be the Poisson point process on $\cW^*\times \R$ with intensity measure $Q^* \otimes \Lambda$, where $\Lambda$ is the Lebesgue measure on $\R$. For each $t\in\R$ and $A \subseteq \R$, we write $\sI_t$ for the set of $w\in \cW^*$ such that $(w,t)\in\sI$, and  write $\sI_{A}$ for the intersection of $\sI$ with $\cW^* \times A$.

See \cite[Proposition 3.3]{hutchcroft2015interlacements} for a limiting construction of the interlacement process from the random walk on an exhaustion with wired boundary conditions.

In \cite{hutchcroft2015interlacements}, we proved that the WUSF can be generated from the random interlacement process in the following manner. Let $G$ be a transient network, and let $t\in \R$. For each vertex $v$ of $G$, let $\sigma_t(v)$ be the smallest time greater than $t$ such that there exists a trajectory $W_{\sigma_t(v)} \in \sI_{\sigma_t(v)}$ that hits $v$, and note that the trajectory $W_{\sigma_t(v)}$ is unique for every $t\in \R$ and $v\in V$ almost surely. We define $e_t(v)$ to be the oriented edge of $G$ that is traversed by the trajectory $W_{\sigma_t(v)}$ as it enters $v$ for the first time, and define 
\[
\AB_t(\sI):=\Bigl\{ -e_t(v): v \in V \Bigr\}.
\]
\cite[Theorem 1.1]{hutchcroft2015interlacements} states that $\AB_t(\sI)$ has the law of the oriented wired uniform spanning forest of $G$ for every $t\in \R$. Moreover, \cite[Proposition 4.2]{hutchcroft2015interlacements} states that $\langle \AB_t(\sI)\rangle_{t\in \R}$ is a stationary, ergodic, mixing, stochastically continuous  Markov process.
% \subsection{The interlacement Aldous-Broder algorithm}
% \subsection{The }

\subsection{$v$-wired variants}

% \medskip

In this section, we introduce a variation on the interlacement process in which a vertex $v$ is wired to infinity, which we call the \textbf{$v$-wired interlacement process}. We then show how the $v$-wired interlacement process can be used to generate the $v$-WUSF in the same way that the usual interlacement process generates the usual WUSF.
% We now develop a version of the interlacement Aldous-Broder algorithm for the WUSF$_\rho$.

Let $G$ be a (not necessarily transient) network and let $v$ be a vertex of $G$. We denote by $\tau_v$ the first time that the random walk visits $v$, and denote by $\tau^+_K$ the first positive time that the random walk visits $K$. 
We write $X^T$ for the random walk ran up to the (possibly random and/or infinite) time $T$, which is considered to be an element of $\cW(0,T)$. In particular, if $X$ is started at $v$ then $X^{\tau_v}$ is the path of length zero at $v$. 
For each finite set $K \subset V$ we define a measure $Q_{v,K}$ on $\cW$ by $Q_{v,K}(\{w \in \cW : w(0)\notin K\})=0$,
% \begin{multline*}
% Q_{v,K}\left( \{w \in \cW: w|_{(-\infty,0]} \in \sA,\, w(0) = u \text{ and } w|_{[0,\infty)} \in \sB \}\right)\\= \begin{cases}
%  c(u)\bP_u\big( X^{\tau_v} \in \sA^\leftarrow \text{ and } \tau_K^+ >\tau_v \big) \bP_u\big(X^{\tau_v} \in \sB \big)
% c(v)\mathbbm{1}(w^0 \in \sA) \bP_v\big(\langle X_k \rangle^{\tau_{v}}_{k\geq0} \in \sB \big)
% +c(v)\bP_v\big( \langle X_{k} \rangle^{\tau_{v}}_{k\geq0}\in \sA^\leftarrow \text{ and } \tau_K^+ = \infty \big) \mathbbm{1}(w_0 \in \sB) & u=v
%  ,
%  \end{cases} 
% \end{multline*}
% where we consider $\tau^+_K > \tau_v$ in the case that $\tau^+_K = \tau_v = \infty$. 
% 
\begin{multline*}
Q_{v,K}\left( \{w \in \cW: w|_{(-\infty,0]} \in \sA,\, w(0) = u \text{ and } w|_{[0,\infty)} \in \sB \}\right)\\= 
% \begin{cases}
 c(u)\bP_u\big( X^{\tau_v} \in \sA^\leftarrow \text{ and } \tau_K^+ >\tau_v \big) \bP_u\big(X^{\tau_v} \in \sB \big) 
 % \end{cases} 
\end{multline*}
(with the convention that $\tau^+_K>\tau_v$ in the case that both hitting times are equal to $\infty$) for every $u\in K\setminus \{v\}$ and every two Borel sets $\sA,\sB \subseteq \cW$, and 
\begin{multline*}
Q_{v,K}\left( \{w \in \cW: w|_{(-\infty,0]} \in \sA,\, w(0) = v \text{ and } w|_{[0,\infty)} \in \sB \}\right)\\=
c(v)\mathbbm{1}(w_0 \in \sA^\leftarrow) \bP_v\big(X^{\tau_v} \in \sB \big)
+c(v)\bP_v\big( X^{\tau_v} \in \sA^\leftarrow \text{ and } \tau_K^+ = \infty \big) \mathbbm{1}(w_0 \in \sB)
\end{multline*}
for every two Borel sets $\sA,\sB \subseteq \cW$ if $v\in K$, 
where we write $w_0\in \cW(0,0)$ for the path of length zero at $v$.

As with the usual interlacement intensity measure, we wish to define a measure $Q^*_v$ on $\cW^*$ via the consistency condition
 \begin{equation}
 \label{eq:rootedintensitydef}
 Q^*_v(\sA \cap \cW^*_K) = Q_{v,K}(\pi^{-1}(\sA))
 \end{equation}
for every finite set $K \subset V$ and every Borel set $\sA \subseteq \cW^*$, and define the $v$-rooted interlacement process to be the Poisson point process on $\cW^*\times \R$ with intensity measure $Q_v^* \otimes \Lambda$, where $\Lambda$ is the Lebesgue measure on $\R$.

We will deduce that such a measure exists via the following limiting procedure, which also gives a direct construction of the $v$-rooted interlacement process. 
Let $N$ be a Poisson point process on $\R$ with intensity measure $(c(\partial_n)+c(v))\Lambda$. Conditional on $N$, for each $t\in N$, let $W_t$ be a random walk on $G_n^{*v}$  started at $\partial_n$ (which is identified with $v$) and stopped when it first returns to $\partial_n$, where we consider each $W_t$ to be an element of $\cW^*$. We define $\sI^n_v$ to be the point process
$\sI^n_v:=\left\{(W_t,t) : t \in N\right\}$.

\begin{proposition}\label{Prop:intexhaust}
Let $G$ be an infinite network, let $v$ be a vertex of $G$, and let $\langle V_n \rangle_{n\geq 0}$ be an exhaustion of $G$. Then the Poisson point processes $\sI^n_v$ converge in distribution as $n\to\infty$ to a Poisson point process $\sI_v$ on $\cW^* \times \R$ with  intensity measure of the form
$Q^*_v \otimes \Lambda$,
where $\Lambda$ is the Lebesgue measure on $\R$ and $Q^*_v$ is a $\sigma$-finite measure on $\cW^*$ such that \eqref{eq:rootedintensitydef} is satisfied 
for every finite set $K \subset V$ and every event $\sA \subseteq \cW^*$.
\end{proposition}

The proof is very similar to that of \cite[Proposition 3.3]{hutchcroft2015interlacements}, and is omitted.

\begin{corollary}
Let $G$ be an infinite network and let $v$ be a vertex of $G$. Then there exists a unique $\sigma$-finite measure $Q^*_v$ on $\cW^*$ such that \eqref{eq:rootedintensitydef} is satisfied 
for every finite set $K \subset V$ and every event $\sA \subseteq \cW^*$.
\end{corollary}

\begin{proof}
The existence statement follows immediately from \cref{Prop:intexhaust}. 
The uniqueness statement is immediate since sets of the form $\sA \cap \cW^*_K$ are a $\pi$-system generating the Borel $\sigma$-algebra on $\cW^*$.
\end{proof}

 We call $\sI_v$ the \textbf{$v$-wired interlacement process}. Note that it may include trajectories that are either doubly infinite, singly infinite and ending at $v$, singly infinite and starting at $v$, or finite and both starting and ending at $v$.

% \[\AB_{v,t}(\sI_v):=\left\{-e_t(u):u \in V \setminus \{v\} \right\}\]

We have the following $v$-rooted analogue of \cite[Theorem 1.1 and Proposition 4.2]{hutchcroft2015interlacements}, whose proof is identical to those in that paper.

\begin{proposition}
Let $G$ be an infinite network, let $v$ be a vertex of $G$, and let $\sI_v$ be the $v$-rooted interlacement process on $G$. Then 
\[\AB_{v,t}(\sI_v):= \bigl\{-e_t(u):u \in V \setminus \{v\} \bigr\} \]has the law of the oriented $v$-wired uniform spanning forest of $G$ for every $t\in \R$. Moreover, the process $\langle \AB_{v,t}(\sI_v) \rangle_{t\in \R}$ is a stationary, ergodic, stochastically continuous Markov process.
\end{proposition}

\subsection{Relation to capacity}

In this section, we record the well-known relationship between the interlacement intensity measure $Q^*$ and the \textbf{capacity} of a set, and extend this relationship to the $v$-rooted interlacement intensity measure $Q^*_v$. 
Recall that the \textbf{capacity} (a.k.a.\ conductance to infinity \cite[Chapter 2]{LP:book}) of a finite set of vertices $K$ in a network $G$ is defined to be
\[
\Cap(K):= Q^*(\cW^*_K)=\sum_{v\in K} c(v)\bP_v(\tau^+_K=\infty),
\]
where $\tau^+_K$ is the first positive time that the random walk visits $K$ and the second equality follows by definition of $Q^*$.
% and observe that, by definition of $Q^*$,
% \[
% Q^*(\cW^*_K) = \Cap(K).
% \]
Similarly, we define the \textbf{$v$-wired capacity} of a finite set $K$ to be
\[
\Cap_v(K) := Q_v^*(\cW^*_K) = \sum_{u\in K \setminus v} c(u)\bP_u(\tau^+_K>\tau_v) + c(v)\mathbbm{1}(v\in K)\left[1+\mathbf{P}_v(\tau^+_K=\infty)\right],
\]
with the convention that $\tau^+_K>\tau_v$ in the case that both hitting times are equal to $\infty$. Note that if $v\notin K$ then $\Cap_v(K)$ is the effective conductance between $K$ and $\{\infty,v\}$ 
(see \cref{sec:AlexanderOrbach} or \cite[Chapter 2]{LP:book} for background on effective conductances). Thus, the number of trajectories in $\sI_{[a,b]}$ that hit $K$ is a Poisson random variable with parameter $|a-b| \Cap(K)$, while the number of trajectories in $\sI_{v,[a,b]}$ that hit $K$ is a Poisson random variable with parameter $|a-b| \Cap_v(K)$.
% \[
% Q^*_v(\cW^*_K) = \Ceff\bigl(K \leftrightarrow \{\infty,v\} \bigr)
% \]
% so that in particular
% \[
% \Cap(K) \leq Q^*_v(\cW^*_K) \leq \Cap(K) + c(v).
% \]
% In particular, for any uniformly transient network with controlled stationary measure we have
% \[
% \P(\cI_{[a,b]} \cap K) \asymp 
% \]

In our setting the capacity and $v$-rooted capacity of a set will always be of the same order: 
The inequality $\Cap(K)\leq \Cap_v(K)$ is immediate, while on the other hand we have that
\[
\Cap_v(K) \leq \Cap(K) + 2c(v) + \sum_{u\in K\setminus \{v\}} c(u) \bP_u( \tau_v < \tau^+_K < \infty).
\]
By time-reversal we have that 
\begin{multline*}c(u)\bP_u( \tau_v < \tau^+_K < \infty )\leq c(u) \bP_u\bigl( \{\tau_v < \tau^+_K < \infty\} \cup \{\tau_v < \infty, \tau^+_K=\infty\}\bigr) 
\\= c(v)\bP_v(\tau^+_K<\infty, X_{\tau^+_K} =u)\end{multline*} for each $u\in K \setminus \{v\}$, and summing over $u\in K \setminus \{v\}$ we obtain that
\begin{equation}
\label{eq:CapvCapComparison}
\Cap_v(K) \leq \Cap(K) + c(v)\bP_v(\tau^+_K<\infty) + 2c(v) \leq \Cap(K) + 3c(v).
\end{equation}
Thus, in networks with bounded vertex conductances, the capacity and and $v$-rooted capacity agree to within an additive constant. Furthermore, the assumption that $\bubnorm{P}<\infty$ implies that $G$ is uniformly transient and hence that $\Cap(K)$ is bounded below by a positive constant for every non-empty set $K$.

\subsection{Evolution of the past under the dynamics}

The reason that the dynamics induced by the interlacement Aldous-Broder algorithm are so useful for studying the past of the origin in the WUSF is that the past itself  evolves in a very natural way under the dynamics. Indeed, if we run time backwards and compare the pasts $\fP_0(v)$ and $\fP_{-t}(v)$ of $v$ in $\F_0$ and $\F_{-t}$, we find that the past can become larger only at those times when a trajectory visits $v$. At all other times, $\fP_{-t}(v)$ decreases monotonically in $t$ as it is `sliced into pieces' by newly arriving trajectories. This behaviour is summarised in the following lemma, which is adapted from \cite[Lemma 5.1]{hutchcroft2015interlacements}. Given a set $A \subseteq \R$, we write $\cI_{A}$ for the set of vertices that are hit by some trajectory in $\sI_{A}$, and write $\fP_t(v)$ for the past of $v$ in the forest $\F_t$.

\begin{lemma}
\label{lem:PastDynamics}
Let $G$ be a transient network, let $\sI$ be the interlacement process on $G$, and let $\langle \F_t\rangle_{t\in \R}=\langle \AB_t(\sI) \rangle_{t\in \R}$. Let $v$ be a vertex of $G$, and let $s<t$. If $v \notin \cI_{[s,t)}$, then $\fP_s(v)$ is equal to the component of $v$ in the subgraph of $\fP_t(v)$ induced by $V \setminus \cI_{[s,t)}$.
\end{lemma}

\begin{proof}
Suppose that $u$ is a vertex of $V$, and let $\Gamma_s(u,\infty)$ and $\Gamma_t(u,\infty)$ be the futures of $u$ in $\F_s$ and $\F_t$ respectively.  Let $u=u_{0,s},u_{1,s},\ldots$ and $u=u_{0,t},u_{1,t},\ldots$ be, respectively, the vertices visited by $\Gamma_s(u,\infty)$ and $\Gamma_t(u,\infty)$ in order. 
Let $i_0$ be the smallest $i$ such that $\sigma_s(u_{i,s}) <t$. Then it follows from the definitions that $\Gamma_s(u,\infty)$ and $\Gamma_t(u,\infty)$ coincide up until step $i_0$, and that $\sigma_s(u_{i,s}) <t$ for every $i\geq i_0$. (Indeed, $\sigma_s(u_{i,s})$ is decreasing in $i$.) On the other hand, if $v\notin \cI_{[s,t)}$ then $\sigma_s(v)>t$, and the claim follows readily.
\end{proof}

Similarly, we have the following lemma in the $v$-wired case, whose proof is identical to that of \cref{lem:PastDynamics} above. Given $A \subseteq \R$, we write $\cI_{v,A}$ for the set of vertices that are hit by some trajectory in $\sI_{v,A}$, and write $\fP_{v,t}(u)$ for the past of $u$ in the forest $\F_{v,t}$.

\begin{lemma}
\label{lem:vPastDynamics}
Let $G$ be a network, let $v$ be a vertex of $G$, let $\sI_v$ be the $v$-wired interlacement process on $G$, and let $\langle \F_{v,t}\rangle_{t\in \R}=\langle \AB_{v,t}(\sI_v) \rangle_{t\in \R}$. Let $u$ be a vertex of $G$, and let $s<t$. If $u \notin \cI_{v,[s,t)}$, then $\fP_{v,s}(u)$ is equal to the component of $u$ in the subgraph of $\fP_{v,t}(u)$ induced by $V \setminus \cI_{v,[s,t)}$.
\end{lemma}

\section{Lower bounds for the diameter}

In this section, we use the interlacement Aldous-Broder algorithm to derive the lower bounds on the tail of the intrinsic and extrinsic diameter of \cref{thm:transitivemain,thm:generalexponents,thm:generalexponentsv,thm:extrinsicZd,thm:extrinsicv}.

\subsection{Lower bounds for the intrinsic diameter}

Recall that $\mathfrak{P}(v)$ denotes the past of $v$ in the WUSF, that $\mathfrak{T}_v$ denotes the component of $v$ in the $v$-WUSF, and that $q(v)$ is the probability that two independent random walks started at $v$ do not return to $v$ or intersect each other at any positive time.

\begin{proposition}
\label{prop:intlower}
Let $G$ be a transient network. Then for each vertex $v$ of $G$ we have that
\begin{equation}
\label{eq:intdiamlower}
\P\Bigl(\diam_\mathrm{int}\bigl(\fP(v)\bigr)\geq r\Bigr) \geq \left(\frac{q(v)\Cap(v) }{e \sup_{u\in V}\Cap(u)}\right)\frac{1}{r+1}
\end{equation}
and similarly
\begin{equation}
\label{eq:intdiamlowerv}
\P\Bigl(\diam_\mathrm{int}\bigl(\fT_v\bigr) \geq r\Bigr) \geq \left(\frac{\Cap(v)}{e \sup_{u\in V} \Cap(u)}\right) \frac{1}{r+1}
\end{equation}
for every $r\geq 0$.
\end{proposition}

Note that \eqref{eq:intdiamlowerv} gives a non-trivial lower bound for every transitive network, and can be thought of as a mean-field lower bound. (For recurrent networks, the tree $\fT_v$ contains every vertex of the network almost surely, so that the bound also holds degenerately in that case.)

\begin{proof}
We prove \eqref{eq:intdiamlower}, the proof of \eqref{eq:intdiamlowerv} being similar.
Let $\sI$ be the interlacement process on $G$ and let $\F=\AB_0(\sI)$. 
Given a path $X \in \cW(0,\infty)$ and a vertex $u$ of $G$ visited by the path after time zero, we define $e(X,u)$ to be the oriented edge pointing into $u$ that is traversed by $X$ as it enters $u$ for the first time, and define
\[\AB(X) = \{-e(X,u) : u \text{ is visited by $X$ after time zero}\}.\]
Note that, by definition of $\cW(0,\infty)$, $X$ visits infinitely many vertices and $\AB(X)$ is an infinite tree oriented towards $X_0$. In particular, $\AB(X)$ contains an infinite path starting at $X_0$, whose edges are oriented towards $X_0$. (If $X$ is a random walk then this infinite path is unique and can be interpreted to be the loop-erasure of the time-reversal of $X$. We will not need to use these properties here.)
% Note that if the trace of $W$ is infinite then $\AB(W)$ is an infinite tree, and in particular contains an infinite path starting at $W_0$.
% Note that if $W$ is doubly-infinite with $\langle $, then the intersection of $$
% Now suppose that $W \in \cW^*$ is doubly infinite and hits $v$. Parameterize $W$ so that it hits $v$ for the first time at time zero, and let $X=W|_{[0,\infty)}$. One may verify from the definitions that the past of $v$ in $\AB(W)$ contains the portion of the loop-erasure of $\langle W_n \rangle_{n\geq 0}$ up until the first positive time that this loop-erasure intersects $\langle W_n \rangle_{n \leq 0}$.

For each $\eps>0$, let $\sA_{\eps}$ be the event that $v$ is hit by exactly one trajectory in $\sI_{[0,\eps]}$, that this trajectory hits $v$ exactly once, and that the parts of the trajectory before and after hitting $v$ do not intersect each other. (In the $v$-wired case, in the proof of \eqref{eq:intdiamlowerv}, one would instead take $\sA_\eps$ to be the event that $v$ is hit by exactly one trajectory in $\sI_{v,[0,\eps]}$, and that this trajectory is half-infinite and begins at $v$.) It follows from the definition of the interlacement intensity measure that
\begin{equation}\label{eq:Aeps}\P(\sA_\eps) = \eps q(v) \Cap(v) e^{-\eps\Cap(v)}.\end{equation}
Given $\sA_\eps$, let $\langle W_n \rangle_{n \in \Z}$ be the unique representative of this trajectory that has $W_0=v$, let $X=W|_{[0,\infty)}$, let $Z$ be an infinite path starting at $v$ in $\AB(X)$ (chosen in some measurable way if there are multiple such paths), and let $\eta$ be the set of vertices visited by the first $r$ steps of $Z$, not including $v$ itself. Let $\sB_{r,\eps}\subseteq\sA_\eps$ be the event that $\sA_\eps$ occurs and that 
 $\eta$ is not hit by any trajectories in $\sI_{[0,\eps]}$ other than $W$.
% \end{enumerate}
On the event $\sB_{r,\eps}$ we have by definition of $\AB_0(\sI)$ that the reversals of the first $r$ edges traversed by $Z$ are all contained in $\F$ and oriented towards $v$, and it follows that
 % Thus, we have that
\[ \P\left(\operatorname{diam}_\textrm{int}(\fP(v)) \geq r\right) \geq \P(\sB_{r,\eps})\]
for every $r\geq 1$ and $\eps>0$. On the other hand, 
by the splitting property of Poisson processes, for every set $K\subset V$, the number of trajectories of $\sI_{[0,\eps]}$ that hit $K$ but not $v$ is independent of the set of trajectories of $\sI_{[0,\eps]}$ that hit $v$. We deduce that
 \begin{equation}
\P(\sB_{r,\eps} \mid \sA_{\eps}) \geq \E\left[ e^{-\eps \Cap(\eta)} \mid \sA_{\eps}\right].
 \end{equation}

% and
Let $M = \sup_{u \in V} \Cap(u)$. We may assume that $M<\infty$, since the claim is trivial otherwise. 
 By the subadditivity of the capacity we have that
 $\Cap(\eta) \leq Mr$, 
so that $\P(\sB_{r,\eps} \mid \sA_{\eps}) \geq e^{-\eps Mr}$
and hence that
\begin{equation}\label{eq:Beps}
\P(\sB_{r,\eps}) \geq q(v)\, \Cap(v) \,  \eps e^{-\eps\Cap(v)}e^{-\eps M r} \geq q(v) \Cap(v)  \eps e^{-\eps M(r+1)}.
\end{equation}
The claimed inequality now follows by taking $\eps= 1/(M(r+1))$.
\end{proof}

The following lemma shows that the lower bound of \eqref{eq:intdiamlower} is always meaningful provided that $\bubnorm{P}<\infty$. We write
$\mathbf{G}(u,v) = \bE_u\left[ \sum_{n\geq0} \mathbbm{1}(X_n=v) \right]$ for the Greens function.

\begin{lemma} 
\label{lem:qpositivity}
Let $G$ be a network with controlled stationary measure and with $\bubnorm{P}<\infty$. Then there exists a positive constant $\eps$ such that for every $v\in V$ there exists $u\in V$ with $\mathbf{G}(v,u)\geq \eps$, $d(v,u) \leq \eps^{-1}$, and $q(u)>\eps$. In particular, if $G$ is transitive then $q(v)$ is a positive constant.
\end{lemma}

Note that the statement concerning the graph distance may hold degenerately on networks that are not locally finite, but that the Greens function lower bound remains meaningful in this setting.

\begin{proof}
Let $X$ and $Y$ be independent random walks started at $v$, and observe that
\begin{multline}
\bE_v\bigl|\{(i,j) : i \geq n, j \geq m, X_i=Y_j \}\bigr| = \sum_{i=n}^\infty\sum_{j=m}^\infty \sum_{w\in V} p_i(v,w)p_j(v,w) \\ \preceq \sum_{i=n}^\infty\sum_{j=m}^\infty p_{i+j}(v,v) \preceq \sum_{\ell=n+m}^\infty (\ell+1)\|P^\ell\|_{1\to\infty} \label{eq:qpositivity1}
\end{multline}
for every $n,m\geq 0$.

 Let $I=\{(i,j) : i \geq 0, j \geq 0, X_i=Y_j \}$ be the set of intersection times, which is almost surely finite by \eqref{eq:qpositivity1}, and let $(\tau_1,\tau_2)$ be the unique element of $I$ that is lexicographically maximal in the sense that every $(i,j)\in I$ either has $i<\tau_1$ or $i=\tau_1$ and $j \leq \tau_2$.
Since $\bubnorm{P}<\infty$ the right hand side of \eqref{eq:qpositivity1} tends to zero as $n+m\to\infty$, and it follows by Markov's inequality that there exists $k_0<\infty$ such that $\bP_v(\tau_1,\tau_2 \leq k_0) \geq 1/2$ for every $v\in V$. Thus, we deduce that
\begin{multline*}
\frac{1}{2} \leq \sum_{u \in V} \sum_{i=0}^{k_0} \sum_{j=0}^{k_0} \bP_v(\tau_1=i,\tau_2=j, X_{i}=Y_{j}=u)\\ \leq \sum_{u \in V} \sum_{i=0}^{k_0} \sum_{j=0}^{k_0} \bP_v(X_{i}=Y_{j}=u,\text{ $\{X_\ell : \ell > i\}$, $\{Y_r : r > j\}$, and $\{u\}$ pairwise disjoint})\\
= \sum_{u \in V} \sum_{i=0}^{k_0} \sum_{j=0}^{k_0} \bP_v(X_{i}=Y_{j}=u)q(u) \leq (k_0+1)^2\sup_{u \in V} \mathbf{G}(v,u)q(u) \mathbbm{1}(d(u,v)\leq k_0).
\end{multline*}
On the other hand, it follows from the definitions that $\sup_{u,v\in V} \mathbf{G}(u,v) \leq \bubnorm{P}$, and the claim follows. 
\end{proof}

% \begin{remark}
% A very similar proof using the Aldous-Broder algorithm for the $v$-WUSF yields that
% \begin{equation}
% \P\Bigl(\diam_\mathrm{int}\bigl(\fP_v(v)\bigr) \geq n\Bigr) \geq \left(\frac{\Cap(v)}{e \sup_{u\in V} \Cap(u)}\right) \frac{1}{n+1}
% \end{equation}
% for \emph{any} transient network with vertex conductances bounded above by $M$. (In infinite recurrent networks $\diam_\mathrm{int}(\fT_v)=\infty$ almost surely.) This can be thought of as a mean-field lower bound.
% \end{remark}

\subsection{Lower bounds for the extrinsic diameter}

In this section we apply a similar method to that used in the previous subsection to prove a lower bound on the tail of the \emph{extrinsic} diameter. The method we use is very general and, as well as being used in the proof of \cref{thm:extrinsicZd,thm:extrinsic,thm:extrinsicspeed} and \ref{thm:extrinsicv}, can also be used to deduce similar lower bounds for e.g.\ long-ranged models.

Let $G$ be a network. For each $r\geq 0$, we define 
\[
L(r) =\sup_{v\in V} \bE_v\left[ \sup \left\{n\geq 0 : X_n \in B(v,r) \right\}\right]
\]
to be the maximum expected final visit time to a ball of radius $r$. 
It is easily seen that every transitive graph of polynomial growth of dimension  $d>2$ has $L(r)\asymp r^2$, and the same holds for any Ahlfors regular network with controlled stationary measure satisfying Gaussian heat kernel estimates, see \cref{lem:polyLr} below. On the other hand, uniformly ballistic networks have by definition that $L(r)\asymp r$.

\begin{proposition}
\label{prop:extlower}
Let $G$ be a transient network. Then for each vertex $v$ of $G$ we have that
\begin{equation}
\label{eq:extlower}
\P\Bigl(\diam_\mathrm{ext}\bigl(\fP(v)\bigr)\geq r \Bigr) \geq \left(\frac{q(v)^2 \Cap(v) }{4e \sup_{u\in V}\Cap(u)}\right)\frac{1}{L(r)+1}
\end{equation}
for every $r\geq 1$, and similarly
\begin{equation}
\label{eq:extlowerv}
\P\Bigl(\diam_\mathrm{ext}\bigl(\fT_v\bigr) \geq r\Bigr)\geq
\left(\frac{\Cap(v) }{4e \sup_{u\in V}\Cap(u)}\right)\frac{1}{L(r)+1}
\end{equation}
for every $r\geq 1$.
\end{proposition}

\begin{proof}[Proof of \cref{prop:extlower}]
We prove \eqref{eq:extlower}, the proof of \eqref{eq:extlowerv} being similar.
% 
%  We begin with a brief preliminary concerning the fastness property.
% Let $c>0$ be sufficiently small that
% \[
% \liminf_{n\to\infty} \bP_v\Bigl( d(v,X_m) \geq f(c n)  \text{ for all $m\geq n$}\Bigr) \geq \frac{3}{4} + \frac{q(v)}{2},
% \]
% and let $N$ be sufficiently large that
% \[
% \bP_v\Bigl( d(v,X_m) \geq f(c n)  \text{ for all $m\geq n$}\Bigr) \geq \frac{1+q(v)}{2}
% \]
% for all $n\geq N$. Then, by the union bound, we have that
% \begin{multline}
% \label{eq:fastness12}
% \bP_v\Bigl( X,Y, \text{do not return to $v$ or intersect after time zero}, d(v,X_m)\geq f(cn) \text{ for all $n\geq N$}\Bigr)\\ \geq 
% \frac{1}{2}
% \end{multline}
% for all $n\geq N$ when $X$ and $Y$ are independent random walks started at $v$. 
% 
% 
% 
We continue to use the notation $\sA_{\eps}$, $\sB_{R,\eps}$, $W$, $X$, $Z$, $\eta$, and $M$ defined in the proof of \cref{prop:intlower} but with the variable $R$ replacing the variable $r$ there, so that $\eta$ is the set of vertices visited by the first $R$ steps of $Z$.  
We also define $\sC_{r,R,\eps} \subseteq \sA_{\eps}$ to be the event in which $\sA_{\eps}$ occurs and the distance in $G$ between 
$v$ and the endpoint of $\eta$ is at least $r$, and define $\sD_{r,R,\eps} = \sC_{r,R,\eps} \cap \sB_{R,\eps}$. 
 Note that, since the path $Z$ is oriented towards $v$ in $\AB(X)$, it follows from the definition of $\AB(X)$ that the $R$th point visited by $Z$ is visited by $X$ at a time greater than or equal to $R$. Thus, we have by the definition of the interlacement intensity measure, the union bound (in the form $\P(A \setminus B) \geq \P(A)-\P(B)$), and Markov's inequality that, letting $Y^1$ and $Y^2$ be independent random walks started at $v$,
\begin{align*}
\P(&\sC_{r,R,\eps})\\ &\geq
% \P\left(\begin{array}{l} v \text{ hit by exactly one}\\ \text{trajectory in $\sI_{[0,\eps]}$}\end{array}\right)
\eps \Cap(v) e^{-\eps \Cap(v)} \mathbf{P}_v\left(
\begin{array}{l}   \{Y^1_i : i > 0\},\, \{Y^2_i : i >0\}, \text{ and } \{v\} \text{ are mutually}\\ \text{ disjoint and }\sup\{n \geq 0 : Y^1_n \in B(v,r) \} < R
% X \text{ and } Y \text{ do}\\ \text{not intersect or return to $v$ after time zero}
\end{array} 
% \begin{array}{l} \sup\{n \geq 0 : Y^1_n \in B(v,r) \} \geq R,\;  \{Y^1_i : i > 0\},\\\{Y^2_i : i >0\}, \text{ and } \{v\} \text{ are mutually disjoint}
% % X \text{ and } Y \text{ do}\\ \text{not intersect or return to $v$ after time zero}
% \end{array}
 \bigg|\; 
% \begin{array}{l} 
v \notin \{Y^2_i : i >0\}
 % \text{ does not return}\\ \text{to $v$ after time zero}\end{array}
 \right)
\\&\geq
\eps \Cap(v) e^{-\eps \Cap(v)} \mathbf{P}_v\left(
\begin{array}{l}   \{Y^1_i : i > 0\},\, \{Y^2_i : i >0\}, \text{ and } \{v\} \text{ are mutually}\\ \text{ disjoint and }\sup\{n \geq 0 : Y^1_n \in B(v,r) \} < R
% X \text{ and } Y \text{ do}\\ \text{not intersect or return to $v$ after time zero}
\end{array} 
 % \text{ does not return}\\ \text{to $v$ after time zero}\end{array}
 \right)
\\&\geq
 \eps \Cap(v) e^{-\eps \Cap(v)}\left[q(v)-\frac{L(r)}{R}\right].
\end{align*}
Thus, it follows by the same reasoning as in the proof of \cref{prop:intlower} that
\[
\P(\diam_\ext(\fP(v))\geq r)\geq 
\P(\sD_{r,R,\eps}) \geq e^{-\eps M R} \P(\sC_{r,R,\eps}) \geq \eps \Cap(v) e^{-\eps M(R+1)}\left[q(v)-\frac{L(r)}{R}\right]
\]
for every $R,r\geq 1$ and $\eps >0$. We conclude by taking $\eps=1/M(R+1)$ and $R= \lceil 2 L(r)/q(v) \rceil$.
\end{proof}

\begin{lemma}
\label{lem:polyLr}
Let $G=(V,E)$ be a network with controlled stationary measure that is $d$-Ahlfors regular for some $d>2$ and satisfies Gaussian heat kernel estimates. Then $L(r)\preceq r^2$.
\end{lemma}

\begin{proof}
Let $v$ be a vertex of $G$ and let $T_r= \sup\{n \geq 0: X_n \in B(v,r)\}$.
It follows from the definitions that 
\begin{align*}
% \sum_{u \in B(v,r)} |B(v,n^{1/2})|^{-1/2} e^{-d(v,u)^2/(cn)} \preceq 
\mathbf{P}_v\bigl(X_n \in B(v,r)\bigr) \preceq \sum_{u \in B(v,r)} |B(v,n^{1/2})|^{-1} \preceq n^{-d/2} r^d
\end{align*}
for every $n\geq 1$ and $r\geq 1$, and also that there exists a positive constant $c$ such that
\begin{align*}
\mathbf{P}_v\bigl(X_n \in B(v,r)\bigr)
 &\succeq \sum_{u \in B(v,r)} \frac{ e^{-d(v,u)^2/(cn)}}{|B(v,n^{1/2})|} 
&\succeq \sum_{u \in B(v,r \wedge n^{1/2})} \frac{1}{|B(v,n^{1/2})|}
\succeq 1 \wedge r^d n^{-d/2}
\end{align*}
for every $n\geq 1$ and $r\geq 1$. We deduce that
\begin{align}
\label{eq:MartinsDetails}
\mathbf{P}_v\bigl(X_n \in B(v,r)\bigr) \asymp \begin{cases} 1 & n\leq r^2\\
 n^{-d/2}r^d &n > r^2.
 \end{cases}
\end{align}
for every $n\geq 1$ and $r\geq 1$, and hence that, since $d>2$,
\begin{equation}
\label{eq:Occupation}
\bE_v\left[ \sum_{m\geq n} \mathbbm{1}\bigl(X_m \in B(v,2r)\bigr)  \right] \preceq  
\begin{cases} r^2 & n\leq r^2\\
 r^d n^{-d/2+1} &n > r^2.
 \end{cases}
\end{equation}
On the other hand, it follows by the strong Markov property that
\[
\bE_v\left[ \sum_{m\geq n} \mathbbm{1}\bigl(X_m \in B(v,2r)\bigr) \;\Bigg |\; T_r \geq n \right] \geq \inf_{u\in V} \mathbf{E}_u\left[ \sum_{m\geq 0} \mathbbm{1}\bigl(X_m \in B(u,r)\bigr) \right] \succeq r^2,
\]
where the final inequality follows from \eqref{eq:MartinsDetails}, and we deduce that 
\[\mathbf{P}_v(T_r \geq n) \preceq r^{d-2}n^{-d/2+1}\]
for every $n\geq r^2$. The claim follows by summing over $n$.
\end{proof}

\section{The length and capacity of the loop-erased random walk}
\label{sec:LERWCap}

In this section, we study the length and capacity of loop-erased random walk. In particular, we prove that in a network with controlled stationary measure and $\bubnorm{P}<\infty$, an $n$-step loop-erased random walk has capacity of order $n$ with high probability. The estimates we derive are used extensively throughout the remainder of the paper. In the case of $\Z^d$, these estimates are closely related to classical estimates of Lawler, see \cite{MR2985195} and references therein.

\subsection{The number of points erased}

Recall that when $X$ is a path in a network $G$, the times $\langle \ell_n(X) : 0 \leq n \leq |\LE(X)|\rangle$ are defined to be the times contributing to the loop-erasure of $X$, defined recursively by $\ell_0(X)=0$ and $\ell_{n+1}(X) = 1+ \max\{m : X_m = X_{\ell_n(X)}\}$. We also define 
\[\rho_n(X) = \max\{ k : \ell_k \leq n\}\]
for each $0 \leq n \leq |X|$, so that $\rho_n(X)$ is the number of times between $1$ and $n$ that contribute to the loop-erasure of $X$. The purpose of this section is to study the growth of $\ell$ and $\rho$ when $X$ is a random walk on a network with controlled stationary measure satisfying $\bubnorm{P}<\infty$.

Recall that we write $X^T$ for the random walk ran up to the (possibly random) time $T$, and use similar notation for other paths such as $\LE(X)$.

\begin{lemma}

	\label{lem:LoopLength2}
	Let $G$ be a transient network, let $v$ be a vertex of $G$, and let $X$ be a random walk on $G$ started at $v$. Then the following hold.
	\begin{enumerate}[leftmargin=*]
	\itemsep0.5em
		\item  
			The random variables $\langle \ell_{n+1}(X)-\ell_{n}(X) \rangle_{n\geq0}$ are independent conditional on $\LE(X)$, and the estimate
			\begin{equation}
				\label{eq:ellloops1}
				\bP_v\left[ \ell_{n+1}(X)-\ell_n(X) -1 =m \mid \LE(X) \right] \leq \pstar{m}
			\end{equation}
			holds for every $n\geq 0$ and $m \geq 0$.
		\item
			If $\bubnorm{P}<\infty$, then
			 % satisfies \eqref{eq:MFC}, then 
			\begin{equation}
				\label{eq:MFCLoopLength2}
				\bE_v \left[ \ell_n(X) \mid  \LE\left(X\right)
				 \right] \leq   \bubnorm{P}  \, n
			\end{equation}
			and
			\begin{equation}
				\label{eq:MFCLoopLength3}
				\bP_v \left[\rho_n(X) \leq \lambda^{-1} n \mid \LE(X) \right] \leq \bubnorm{P} \, \lambda^{-1}
			\end{equation}
			for every $n\geq 1$ and $\lambda > 0$.
\end{enumerate}
\end{lemma}

Note that, in the other direction, we have the trivial inequalities $\ell_n \geq n$ and $\rho_n \leq n$.

\begin{proof}
	
	We first prove item 1. Observe that the conditional distribution of $\langle X_i \rangle_{i \geq \ell_n}$ given  $X^{\ell_n}$  is equal to the distribution of a random walk started at $X_{\ell_n}$ and conditioned never to return to the set of vertices visited by $\LE(X)^{n-1}$. (This is the same observation that is used to derive the \emph{Laplacian random walk} representation of loop-erased random walk, see  \cite{Lawler87}.)
	Thus, the conditional distribution of $X$ given $\LE(X) = \gamma$ can be described as follows. 
	For each finite path $\eta$ in $G$, let $w(\eta)$ be its random walk weight
	\[w(\eta) = \bP_{x_0}( X^{|\eta|} = \eta) = \prod_{i=0}^{|\eta|-1} \frac{c(\eta_{i,i+1})}{c(\eta_i)}.\]
	For each time $n \geq 0$, let 
	$L_n=L_n(\LE(X))$ be the set of finite loops in $G$ that start and end at $\LE(X)_n$ and do not hit the trace of $\LE(X)^{n-1}$ (which we consider to be the empty set if $n=0$). In particular, $L_n$ includes the loop of length zero at $\LE(X)_n$ for each $n \geq 0$. Then the random walk segments $\langle X_i \rangle_{i=\ell_n}^{\ell_{n+1}-1}$ are conditionally independent given $\LE(X)$, and have law given by
	% t $Y^i$ be independent random elements of $L_i(\gamma)$, where
	% 
	\begin{equation}
		\label{eq:conditionalloops1}
		\bP_v(\langle X_i \rangle_{i=\ell_n}^{\ell_{n+1}-1} = \eta \mid \mathsf{LE}(X)) 
		= \frac{w(\eta)}{\sum_{\eta' \in L_n}w(\eta')}\mathbbm{1}(\eta \in L_n).
	\end{equation}
	% 
	% Note that the denominator is finite since
	% % 
	% \[
	% 	\sum_{y \in L_i(\gamma)} w(y) \leq \sum_{y} w(y) 
	% 	= \sum_{n\geq0} p^v_n(\gamma_i,\gamma_i).
	% \]
	% % 
	% The right hand side is the expected number of returns to $\gamma_i$ by a random walk started at $\gamma_i$ before hitting $v$, and we have
	% % 
	% \[ 
	% 	\sum_{n\geq0} p^v_n(\gamma_i,\gamma_i) = \bP_{\gamma_i}(\tau_{\gamma_i}^+ = \infty \text{ or } \tau_{\gamma_i}^+ > \tau_v)^{-1} <\infty.
	% \]
	% 
	The contribution of the loop of length zero ensures that the denominator in \eqref{eq:conditionalloops1} is at least one, so that
	\begin{equation*}
	\bP_v\left(\langle X_i \rangle_{i=\ell_n}^{\ell_{n+1}-1} = \eta  \mid  \mathsf{LE}(X)\right) \leq  w(\eta)\mathbbm{1}(\eta \in L_n)
	\end{equation*}
	and hence, summing over $\eta \in L_n$ of length $m$,
	\begin{equation}
		\label{eq:looplengthestimate}
		\bP_v\left(\ell_{n+1}-\ell_{n} -1 = m  \mid \LE(X) \right) \leq  p_{m}\left(\LE(X)_n,\LE(X)_n\right) \leq \pstar{m} 
	\end{equation}
	for all $m \geq 0$, establishing item 1.

	For item 2, \eqref{eq:MFCLoopLength2} follows immediately from \eqref{eq:ellloops1}. Furthermore, $\rho_n \leq \lambda^{-1} n$ if and only if $\ell_{\lfloor \lambda^{-1} n \rfloor} \geq n$, so that \eqref{eq:MFCLoopLength3} follows from \eqref{eq:MFCLoopLength2} and Markov's inequality.\qedhere

\end{proof}

We remark that \cref{lem:LoopLength2} together with the strong law of large numbers for independent, uniformly integrable random variables \cite[Theorem 2.19]{hall2014martingale} has the following easy corollary. Since we do not require the result for the remainder of the paper, the proof is omitted.

\begin{corollary}
	\label{cor:quenchedrhogrowth}
	Let $G$ be a transient network, and let $X$ be a random walk on $G$. 
			If $\bubnorm{P}<\infty$, then
			\[ 
				\limsup_{n\to\infty} \frac{1}{n}\ell_n(X) \leq \limsup_{n\to\infty} \frac{1}{n}\E\ell_n(X) \leq \bubnorm{P}
			\]
			almost surely.
\end{corollary}

% \begin{proof}
% 	Item (1) follows from item (1) of \cref{lem:LoopLength2} and the strong law of large numbers of independent, uniformly integrable random variables. Item (2) is similar, except that we apply the strong law of large numbers to the sequence of random variables
% 	\[
% 		\frac{\ell_{n+1}-\ell_n+1}{\left(\log(\ell_{n+1}-\ell_n +1)\right)^{1+\eps}}.
% 	\]
% 	It follows from item (1) of \cref{lem:LoopLength} that, conditional on $\LE(X)$, these random variables are independent, have bounded mean and are uniformly integrable, and so
% 	\[
% 		\limsup_{n\to\infty} \frac{\ell_{n+1}-\ell_n+1}{\left(\log(\ell_{n+1}-\ell_n +1)\right)^{1+\eps}} <\infty
% 	\]
% 	almost surely for every $\eps>0$ by the strong law of large numbers. The claim now follows easily from the elementary inequality
% 	\[ 
% 		\sum_{i =1}^n \frac{a_i}{\log(a_i)^{1+\eps}} \geq \frac{\sum_{i=1}^n a_i}{\left(\log\left(\sum_{i=1}^n a_i\right)\right)^{1+\eps}},
% 	\]
% 	which holds for any $\eps \in \R$ and any sequence $a_1,\ldots,a_n$ such that $a_i \geq 1$ for every $1 \leq i\leq n$. \qedhere

% \end{proof}

% \begin{remark}
% What happens in $\Z^4$?
% \end{remark}

	The following variation of \cref{lem:LoopLength2}, applying to the loop-erasure of a random walk stopped upon hitting a vertex $v$, is proved similarly.

\begin{lemma}

	\label{lem:LoopLength}
	Let $G$ be a network. Let $u$ and $v$ be distinct vertices of $G$, let $X$ be a random walk started at $u$, and let $\gamma$ be a simple path connecting $u$ to $v$.  Then the following hold.
	\begin{enumerate}[leftmargin=*]
		\itemsep0.5em
		\item 
			The random variables $\langle \ell_{n+1}(X^{\tau_v})-\ell_{n}(X^{\tau_v}) \rangle_{n=0}^{|\gamma|-1}$ are independent conditional on the event that $\tau_v<\infty$ and $\LE(X^{\tau_v})=\gamma$, and the estimate
			\vspace{0.2em}
			\begin{equation}
			\vspace{0.2em}
				\bP_u\left(\ell_{n+1}(X^{\tau_v})-\ell_n(X^{\tau_v}) -1 =m \mid \tau_v < \infty,\, \LE(X^{\tau_v})=\gamma \right) \leq \pstar{m}
			\end{equation}
			holds for every vertex $1 \leq n \leq |\gamma|-1$ and every $m\geq 0$.
		
		\item
			If $\bubnorm{P}<\infty$, then 
			\vspace{0.2em}
			\begin{equation}
			\vspace{0.2em}
				\label{eq:MFCLoopLength1}
				\bE_u \left[ \tau_v \mid  \tau_v < \infty,\, \LE(X^{\tau_v})=\gamma
				 \right] \leq \bubnorm{P} |\gamma|. 
			\end{equation}

\end{enumerate}
\end{lemma}

\begin{proof}

	Write $\ell_n=\ell_n(X^{\tau_v})$. 
	 Observe that the conditional distribution of $\langle X_i \rangle_{i=\ell_n}^{\tau_v}$ given the random variable $X^{\ell_n}$ and the event $\ell_n<\tau_v<\infty$ is equal to the distribution of a simple random walk started at $X_{\ell_n}$ and conditioned to hit $v$ before hitting the set of vertices visited by $\LE(X)^{n-1}$. The rest of the proof is very similar to that of \cref{lem:LoopLength2}. \qedhere

\end{proof}

\subsection{The capacity of loop-erased random walk}

Given a transient path $X$ in a network, we define
$\eta_n(X) = \max\{\ell_k(X) :  k \geq 0, \ell_k(X) \leq n\}$ for each $n\geq 0$. The  time $\eta_n(X)$ is defined so that $\LE(X^{\eta_n}) = \LE(X^n)^{\rho_n} = \LE(X)^{\rho_n}$, and in particular, every edge traversed by $\LE(X^{\eta_n})$ is also traversed by both $\LE(X)$ and $\LE(X^n)$.
The goal of this subsection is to prove the following estimate, which will play a fundamental role in the remainder of our analysis.

\begin{proposition}
\label{prop:capacities}
Let $G$ be a network with controlled stationary measure. 
 If $\bubnorm{P}<\infty$, then 
	% there exists a constant $C=C(\bubnorm{P},M)$ such that
	\vspace{0.6em}
	\begin{equation}
	\vspace{0.3em}
	 \bP_v \left( \Cap\left(\LE(X^n)\right) \leq \lambda^{-1} n \right) \leq 
   \bP_v \left( \Cap\left(\LE(X^{\eta_n})\right) \leq \lambda^{-1} n \right) \preb \lambda^{-1/3} 
	\end{equation}
	and similarly
	\vspace{0.4em}
	\begin{equation}
  \label{eq:smallcapLERWprob}
	\vspace{1em}
	 \bP_v \left( \Cap\left(\LE(X^{\ell_n})\right) \leq \lambda^{-1} n \right) \preb \lambda^{-1/2} 
	\end{equation}
	for every vertex $v$ of $G$, every $n \geq 1$, and every $\lambda \geq 1$. 

\end{proposition}

We do not expect these bounds to be optimal.

Our primary means of estimating capacity will be the following lemma. Given a network $G$, we write
$|A|_c=\sum_{v\in A} c(v)$ for the total conductance of a set of vertices $A$, write
\[\mathbf{G}(u,v) = \bE_u\left[ \sum_{n\geq0} \mathbbm{1}(X_n=v) \right]\]
for the Greens function on $G$,
and define for each finite set of vertices $A$ of $G$ the quantity
\[\mathbf{I}(A) =  \sum_{u,v\in A} c(u) \mathbf{G}(u,v)=\sum_{u\in A} c(u) \bE_u\bigg[ \sum_{n \geq 0} \mathbbm{1}\left(X_n \in A\right)\bigg]. \]
Note that for any two sets of vertices $A \subseteq B$, we have $\mathbf{I}(A) \leq \mathbf{I}(B)$.
When $G$ has controlled stationary measure, the ratio $\mathbf{I}(A)/|A|_c$ is comparable to the expected number of steps a random walk spends in $A$ when started from a uniform point of $A$.
 % returns to $A$.

\begin{lemma} 
	\label{lem:maincap}
	Let $G$ be a transient network. Then 
	\[
		\Cap(A) \geq \frac{|A|_c^2}{  \mathbf{I}(A)}
	\]
	for every finite set of vertices $A$ in $G$.
\end{lemma}

\begin{proof} Recall the following variational formula for the capacity of a finite set $A$ \cite[Lemma 2.3]{JainOrey-properties}:
\[
\Cap(A)^{-1} = \inf\left\{ \sum_{u,v\in A} \frac{\mathbf{G}(u,v)}{c(v)} \mu(u)\mu(v) : \mu \text{ is a probability measure on $A$} \right\}.
\]
The claim follows by taking $\mu$ to be the measure 
$\mu(v)=c(v)/|A|_c$.
\end{proof}

 A useful feature of \cref{lem:maincap} is that once one has an upper bound on $\mathbf{I}(A)$ for some set $A$, one also obtains lower bounds on the capacity of all \emph{subsets} of $A$ in terms of their size. 
In particular,  \cref{lem:maincap} yields that
\[\Cap\left( \LE(X^{\eta_n}) \right) \geq \left[\inf_{u \in V} c(u)\right]^2 \frac{(\rho_n+1)^2}{  \mathbf{I} ( \LE(X^{\eta_n}))}  \geq  \left[\inf_{u \in V} c(u)\right]^2 \frac{(\rho_n+1)^2}{  \mathbf{I} (X^{n})}, \]
so that to lower bound $\Cap( \LE(X^{\eta_n}) )$ it will suffice to lower bound $\rho_n$ and upper bound $\mathbf{I}(X^n)$. Moreover, for our purposes, it will suffice to control the expectation of $\mathbf{I}(X^n)$. 

\begin{lemma}
	\label{lem:Imean}
	Let $G$ be a network. Then 
	\[
		 \bE_v \left[ \mathbf{I}(X^n) \right] 
		 \leq 2(n+1)\left[ \sup_{u \in V} c(u)\right] \cdot  \left[  \sum_{m=0}^n (m+1) \pstar{m} + (n+1)\sum_{m=n+1}^\infty \pstar{m} \right]
	\]
	for every $v\in V$ and $n \geq 0$.
\end{lemma}

\begin{remark}
\cref{lem:maincap,lem:Imean} give the correct order of magnitude for the capacity of the random walk on $\Z^d$ for all $d\geq 3$, which is order $\sqrt{n}$ when $d=3$, order $n/\log n$ when $d=4$, and order $n$ when $d\geq 5$. See \cite{MR3852443} and references therein for more detailed results.
\end{remark}

\cref{lem:Imean} has the following immediate corollary.

\begin{corollary}
\label{cor:Imean2}
Let $G$ be a network with controlled stationary measure. 
	If $\bubnorm{P}<\infty$ then 
	% there exists a constant $C=C(\bubnorm{P},\inf_{u \in V} c(u))$ such that
	\[ \bE_v\left[\mathbf{I}(X^n)\right] \preb n\]
	for every vertex $v$ of $G$ and every $n \geq 1$.
\end{corollary}

Before proving \cref{lem:Imean}, let us use it, together with \cref{lem:LoopLength2}, to deduce \cref{prop:capacities}.

\begin{proof}[Proof of \cref{prop:capacities} given \cref{cor:Imean2}] 
	It follows from \cref{lem:maincap} and a union bound that
	\begin{equation}
	\label{eq:Iunionbound1}
		\bP_v\left( \Cap( \LE(X^{\eta_n}) ) \leq \lambda^{-1} n \right) 
		\leq \bP_v \left( (\rho_n+1) \leq \left[\inf_{u\in V} c(u)\right]^{-1} \lambda^{-1/3} n \right) + \bP_v \left( \mathbf{I}(X^n) \geq \lambda^{1/3} n \right)
	\end{equation}
	and similarly
	\begin{equation}
	\label{eq:Iunionbound2}
	\bP_v\left( \Cap( \LE(X^{\ell_n}) ) \leq \lambda^{-1} n \right) 
	\leq 
	\bP_v \left(\ell_n \geq \lambda^{1/2} n \right) + \bP_v \left( \mathbf{I}\left(X^{\lfloor \lambda^{1/2} n \rfloor}\right) \geq \left[\inf_{u\in V} c(u)\right]^{-2}\lambda n \right).
	\end{equation}
	The claim now follows immediately from \cref{lem:LoopLength2}, \cref{cor:Imean2} and Markov's inequality. 
\end{proof}

\begin{proof}[Proof of \cref{lem:Imean}]
 Conditional on $X$, let $Y^i$ be a random walk started at $X_i$ for each $i\geq0$, writing $\P$ for the joint law of $X$ and the walks $\langle Y^i \rangle_{i \geq0}$.  Then we have
	\[
		\bE_v\left[ \mathbf{I}(X^n) \right] 
		\leq 
		 \sum_{j = 0}^{n} \sum_{i=0}^n \sum_{k \geq0} \E\left[  c(X_i) \mathbbm{1}(Y^i_k = X_j) \right].
	\]
	% 
	% We can express
	(This is an inequality rather than an equality because the right-hand side counts vertices with multiplicity according to how often they are visited by $X$.)
	We split this sum into two parts according to whether $i\leq j$ or $i > j$.
	For the first sum, we have that
	\begin{align*}
		\sum_{i=0}^n \sum_{j = i}^{n} \sum_{k \geq0} \E\left[ c(X_i) \mathbbm{1}(Y^i_k = X_j) \right]
		&= \sum_{i=0}^n \sum_{j = i}^{n} \sum_{k \geq0} \sum_{u,w \in V} p_i(v,u)p_{k}(u,w)c(u)p_{j-i}(u,w).
	\end{align*}
	Reversing time and rearranging yields that
	\begin{align*}
		 % \sum_{i=0}^n \sum_{j = i}^{n} &\sum_{k \geq0} \P(Y^i_m = X_j) 
		 \sum_{i=0}^n \sum_{j = i}^{n} \sum_{k \geq0} \E\left[ c(X_i) \mathbbm{1}(Y^i_k = X_j) \right]
		&= \sum_{i=0}^n \sum_{j = i}^{n} \sum_{k \geq0} \sum_{u,w \in V} p_i(v,u)p_{k}(u,w)p_{j-i}(w,u) c(w)\\
		% \nonumber
		% \\
		% & \leq  \sum_{i=0}^n \sum_{j = i}^{n} \sum_{k \geq0} \sum_{u \in V} p_i(v,u) p_k(u,w) p_{j-i} (w,u) c(u)^{-1}\\
		&\leq  \left[\sup_{u \in V} c(u)\right] \sum_{i=0}^n \sum_{j = i}^{n} \sum_{k \geq0} \pstar{k+j-i}.
		% (n+1) \sum_{j = 0}^{n} \sum_{k \geq0} \pstar{k+j}.
	\end{align*}
    Similarly, for the second sum, we have that
  \begin{align*}
    \sum_{i=0}^n \sum_{j = 0}^{i-1} \sum_{k \geq0} \E\left[ c(X_i)  \mathbbm{1}(Y^i_k = X_j) \right]
    % &= \sum_{i=0}^n \sum_{j = i}^{n} \sum_{k \geq0} \sum_{u,w \in V} p_j(v,u)p_{i-j+k}(u,u)\\
    &\leq \left[\sup_{u \in V} c(u)\right] \sum_{i=0}^n \sum_{j=0}^{i-1} \sum_{k \geq0} \pstar{i-j+k}, 
  \end{align*}
  and summing these two bounds we obtain that
  \[
\bE_v\left[ \mathbf{I}(X^n) \right] 
    \leq \left[\sup_{u \in V} c(u)\right] \sum_{i=0}^n \sum_{j=0}^{n} \sum_{k \geq0} \pstar{|i-j|+k}.
  \]
Using the substitutions $\ell=|i-j|$ and $m=k+\ell$ and noting that there at most $2(n+1)$ choices of $i$, $j$ and $k$ corresponding to each $\ell$ and $m$ with $\ell \leq m$, we deduce that
  \begin{multline*}
\frac{1}{2}\left[\sup_{u \in V} c(u)\right]^{-1} \bE_v\left[ \mathbf{I}(X^n) \right] 
    \leq (n+1) \sum_{m=0}^\infty \sum_{\ell =0}^{m \wedge n} \pstar{m}\\ = (n+1) \sum_{m=0}^n (m+1)\pstar{m} + (n+1)^2 \sum_{m=n+1}^\infty \pstar{m}
  \end{multline*}
  as claimed. \qedhere

\end{proof}

\section{Volume bounds}
\label{sec:volume}

In this section, we study the volume of balls in both the WUSF and $v$-WUSF. In \cref{subsec:volupper} we prove upper bounds on the moments of the volumes of balls, while in \cref{subsec:volume_lower} we prove lower bounds on moments and upper bounds on the probability that the volume is atypically small. Together, these estimates will imply that $d_f(T)=2$ for every component $T$ of $\F$ almost surely. The estimates in this section will also be important in \cref{sec:exponents,sec:AlexanderOrbach}.

\subsection{Upper bounds}
\label{subsec:volupper}

The goal of this subsection is to obtain tail bounds on the probability that an intrinsic ball in the WUSF contains more than $n^2$ vertices.
The upper bounds we obtain are summarized by the following two propositions, which are generalisations of  \cite[Theorem 4.1]{barlow2016geometry}. 
% which the majority of this subsection is devoted to proving. 

Recall that $\fB(v,n)$ denotes the intrinsic ball of radius $n$ around $v$ in the WUSF $\F$, and $\fB_v(v,n)$ denotes the intrinsic ball of radius $n$ around $v$ in the $v$-WUSF $\F_v$.
We define the constant
  \begin{equation}
    \label{eq:alphadefn}
    \alpha=\alpha(G)  = 4 \frac{\sup_{u \in V} c(u)}{\inf_{u \in V} c(u)} \bubnorm{P}.
  \end{equation}

\begin{proposition}
  \label{thm:moments}
  Let $G$ be a network with controlled stationary measure such that $\bubnorm{P}<\infty$, and let
    $\F$ be a wired uniform spanning forest of $G$.
   Then the  estimates
  \begin{align}
    \label{eq:thmmomentsunrooted}
  \vspace{0.1em}
    \E \left[|\fB(v,n)|^k \right] \leq e\, (k+1)!\, \alpha^k \, (n+1)^{2k}
  \end{align}
  and
  \vspace{0.3em}
  \begin{align}
    \label{eq:thmmomentsunrootedexp}
    \E\left[\exp\left( \frac{t}{\alpha(n+1)^2} |\fB(v,n)| \right)\right] \leq \frac{1}{1-t}
    \vspace{1em}
  \end{align}
  hold for every $v\in V, n\geq 0$, $k\geq 1$ and $0 \leq t <1 $.
\end{proposition}

We also obtain the following variation of this proposition applying to the $v$-WUSF.

\begin{proposition}
  \label{thm:vmoments}
  Let $G$ be a network with controlled stationary measure such that $\bubnorm{P}<\infty$, let $v$ be a vertex of $G$ and let
    $\F_v$ be a $v$-wired uniform spanning forest of $G$.
  % , and for each vertex $v$ of $G$ let $\F_v$ be a $v$-wired uniform spanning forest of $G$. Let $\fB(v,n)$ and $\fB_v(v,n)$ denote the ball of radius $n$ about $v$ in $\F$ and $\F_v$ respectively. 
   Then the  estimates
  \begin{align}
    \label{eq:thmmomentsrooted}
  \vspace{0.1em}
    \E \left[|\fB_v(v,n)|^k \right] \leq (k-1)!\, \alpha^k \, (n+1)^{2k-1}
  \end{align} 
  and
    \vspace{0.1em}
  \begin{align}
    \label{eq:thmmomentsrootedexp}
  \vspace{0.1em}
    \E \left[\exp\left( \frac{t}{\alpha(n+1)^2} |\fB_v(v,n)| \right)\right] \leq 1 - \frac{\log(1-t)}{n+1}
  \end{align}
  hold for every $v\in V, n\geq 0$, $k\geq 1$ and $0 \leq t <1 $.
\end{proposition}

(To prove our main theorems it suffices to have just the first and second moment bounds of \cref{thm:moments,thm:vmoments}. We include the exponential moment bounds for future application since they are not much more work to derive.)

Before proving 
\cref{thm:moments,thm:vmoments}, we note the following important corollaries.

\begin{corollary}
  \label{cor:exponentialproblargevolume}
   Let $G$ be a network with controlled stationary measure and with $\bubnorm{P}<\infty$, and let $\F$ be a wired uniform spanning forest of $G$. Then
        \vspace{0.3em}
        \[
        \vspace{0.3em}
          \P\left(|\fB(v,n)| \geq \lambda \alpha (n+1)^2\right) \leq \lambda e^{-\lambda +1}
        \]
        for all $v\in V$, $n\geq0$, and $\lambda \geq 1$.
\end{corollary}

\begin{proof}
 By Markov's inequality, we have that
  \begin{align*}
    \P\left(|\fB(v,n)| \geq \lambda \alpha (n+1)^2 \right) 
    \leq e^{-t\lambda} \E\left[\exp\left( \frac{t}{\alpha(n+1)^2}|\fB(v,n)| \right)\right] 
    \leq \frac{e^{-t\lambda}}{1-t}
  \end{align*}
  for every $v\in V$, $n\geq 0$, and $0 \leq t <1$. The claim follows by taking $t=1-\lambda^{-1}$.
\end{proof}

\begin{corollary}
  \label{cor:quenchedvolumeupper}
Let $G$ be a network with controlled stationary measure and with $\bubnorm{P}<\infty$, and let $\F$ be a wired uniform spanning forest of $G$. Then
      \[
      \vspace{0.3em}
        \limsup_{n\to\infty} \frac{|\fB(v,n)|}{n^2 \, \log \log n}
        \leq \alpha
      \]
      almost surely for every vertex $v$ of $G$.      
\end{corollary}

\begin{proof}
  Let $a>1$ and let $n_k= \lceil a^k \rceil$. Then for every $\eps>0$, $v\in V$ and all $k$ sufficiently large we have that, by \cref{cor:exponentialproblargevolume},
  \[
    \P\left(|\fB(v,n_k)| \geq  (1+\eps) \alpha (n_k+1)^2 \log\log n_k \right) \leq \frac{e(1+\eps)\log(k\log a) }{(k\log a)^{1+\eps}}.   
  \]
  The right hand side is summable in $k$ whenever $\eps>0$, and it follows by Borel-Cantelli that
  \[
  \limsup_{k \to\infty} \frac{|\fB(v,n_k)|}{\alpha(n_k+1)^2  \log\log n_k} 
  \leq 1
  \]
  almost surely.
  Since $|\fB(v,n)|$ is increasing and for every $n$ there exists $k$ such that $n_k \leq n \leq a n_k$, it follows that 
  \[
    \limsup_{k\to\infty} \frac{|\fB(v,n)|}{\alpha(n+1)^2\log \log n} 
    \leq a^2 \limsup_{k \to\infty} \frac{|\fB(v,n_k)|}{\alpha(n_k+1)^2  \log\log n_k} 
    \leq a^2
  \]
  almost surely, and the claim follows since $a>1$ was arbitrary.
\end{proof}

\begin{remark}
  \cite[Proposition 2.8]{BarKum06}\footnote{That work studies the IIC on the $3$-regular tree, rather than the WUSF. We recall however that the IIC and the component of the origin in the WUSF have the same distribution on a $k$-regular tree, namely that of (the unimodular version of) a critical Binomial  Galton-Watson tree conditioned to survive forever.} shows that \cref{cor:quenchedvolumeupper} is sharp in the sense that, when $G$ is a $3$-regular tree, $\log \log n$ cannot be replaced with $(\log \log n)^{1-\eps}$ for any $\eps>0$.
\end{remark}

\medskip

We now begin working towards the proof of \cref{thm:moments,thm:vmoments}.
We begin with a first moment estimate. 
 % with the following preliminary lemma about loop-erased random walk.

% We next obtain first moment estimates on the expected size of balls in the WUSF$_o$.

\begin{lemma}
  \label{lem:expectedballgrowth1} 
  Let $G$ be a network with controlled stationary measure and with $\bubnorm{P}<\infty$. Then
        \[
          \E |\fB_v(v,n)| \leq \alpha (n+1) 
        \]
        for every $v\in V$ and $n\geq 0$.
\end{lemma}

\begin{proof}
Let $u\in V$, and consider sampling the $v$-rooted uniform spanning forest $\F_v$ using Wilson's algorithm rooted at $v$, starting with a random walk $X$ with $X_0=u$.  Then $u \in \mathfrak{B}(v,n)$ if and only if the random walk started at $u$ hits $v$ and the loop-erasure of the random walk path stopped when it first hits $v$ has length at most $n$. Denote this event $\sA_n(u,v)$, so that
  \[
    \E|\fB_v(v,n)| = \sum_{u \in V} \bP_u(\sA_n(u,v)).
  \]
  % Let $\sB_n(u,v) = \bigcup_{m=0}^n \sA_m(u,v)$.
  % Let $C= 1 + \sum_{m\geq0} m\pstar{m}$.

  If $\bubnorm{P}<\infty$, then for every two vertices $u$ and $v$ in $G$ and every simple path $\gamma$ from $u$ to $v$ we have that, by the estimate \eqref{eq:MFCLoopLength1} of \cref{lem:LoopLength} and Markov's inequality,
  \[ 
    \bP_u\left( \tau_v \geq 2 \bubnorm{P} \cdot |\gamma| \; \mid \;  \tau_v < \infty,\, \LE(X^{\tau_v})=\gamma\right) \leq \frac{1}{2}.
  \]
   Taking expectations over $\LE(X)$ conditional on the event $\sA_n(u,v)$ yields that
  \[
    \bP_u\left(\tau_v \leq 2 \bubnorm{P} \, n \mid \sA_n(u,v)\right) \geq 1/2 
  \]
  and hence, by Bayes' rule
  \[
    \bP_u\left(\sA_n(u,v)\right) \leq 2 \bP_u\left(\tau_v \leq 2 \bubnorm{P} \, n\right) \leq 2\sum_{k=0}^{\lfloor 2\bubnorm{P} \, n \rfloor}p_k(u,v). 
  \]
  Reversing time we have
  \[
    \bP_u(\sA_n(u,v)) \leq \frac{2 c(v)}{c(u)} \sum_{k=0}^{\lfloor 2\bubnorm{P} \, n \rfloor}p_k(v,u),
  \]
  from which the claim may immediately be derived by summing over $u\in V$.
\end{proof}

We next use an inductive argument to control the higher moments of $|\fB_v(v,n)|$.
 % in the $v$-wired uniform spanning forest.

\begin{lemma}

  \label{lem:ball2ndmoment}
Let $G$ be a network. Then
  % \vspace{0.1em}
  \[
  \vspace{0.4em}
    \sup_{v \in V} \E  \Bigl[ |\fB_v(v,n)|^k \Bigr] 
    \leq (k-1)!\, (n+1)^{k-1} \, \sup_{v \in V} \E  \Bigl[ |\fB_v(v,n)| \Bigr]^k 
  \]
  for every $n\geq 0$ and $k\geq 1$.

\end{lemma}

\begin{proof}

  By induction, it suffices to prove that the inequality
  \vspace{0.2em}
  \begin{equation}
  \vspace{0.2em}
    \label{eq:momentinduction}
    \E \left[|\fB_v(v,n)|^k\right] \leq (k-1)(n+1) \, \E \left[|\fB_v(v,n)|^{k-1}\right] \, \sup_{w \in V}\E \left[|\fB_w(w,n)|\right]
  \end{equation}
  holds for every vertex $v$ of $G$ and every $k \geq 1$. 
  To this end, let $v$ be a vertex of $G$ and let $u_1,\ldots,u_{k-1}$ be vertices of $G$ such that $\P(u_1,\ldots,u_{k-1} \in \fB_v(v,n))>0$. 
It follows from Wilson's algorithm that, conditional on the event that $u_1,\ldots,u_{k-1} \in \fB_v(v,n)$ and on the paths $\Gamma_v(u_1,v),\ldots,\Gamma_v(u_{k-1},v)$ connecting each of the vertices $u_i$ to $v$ in $\F_v$, the probability that a vertex $w$ is contained in $\fB_v(v,n)$ is at most the probability that a random walk started at $w$ hits one of the paths $\Gamma_v(u_i,v)$ and that the loop-erasure of this stopped path has length at most $n$. Thus, we obtain 
  \begin{multline*}
    \P\left(u_k \in \fB_v(v,n) \mid u_1,\ldots,u_{k-1} \in \fB_v(v,n),\, \langle \Gamma_v(u_i,v)\rangle_{i=1}^{k-1}\right) \\
    \leq \sum_{i=1}^{k-1} \sum_{u \in \Gamma_v(u_i,v)} \bP_{u_k}(\tau_u < \infty,\; |\LE(X^{\tau_u})| \leq n)
    = \sum_{i=1}^{k-1} \sum_{u \in \Gamma_v(u_i,v)}  \P(u_k \in \fB_u(u,n))
    % 
     % \leq (k-1)(n+1) \sup_{u \in V}\E \left[|\fB^u(u,n)|\right]. 
   \end{multline*}
   and hence, summing over $u_k$ and taking expectations over $\langle \Gamma(u_i,v)\rangle_{i=1}^{k-1}$ we obtain that
   \begin{equation*}
    \E\left[|\fB_v(v,n)| \mid u_1,\ldots,u_{k-1} \in \fB_v(v,n)\right] \\
     \leq (k-1)(n+1) \sup_{u \in V}\E \left[|\fB_u(u,n)|\right]. 
   \end{equation*}
(This inequality can also be deduced using \cref{lem:domination}.)  The inequality \eqref{eq:momentinduction} now follows from this together with the inequality
  % Using the inequality
  \begin{multline*}
    \E \left[|\fB_v(v,n)|^k\right] \\
    = \sum_{u_1,\ldots,u_k \in V} \P\bigl(u_k \in \fB_v(v,n) \mid u_1,\ldots,u_k \in \fB_v(v,n)\bigr) \P\bigl(u_1,\ldots,u_{k-1} \in \fB_v(v,n)\bigr) \\
     \leq  \E \left[|\fB_v(v,n)|^{k-1}\right] \sup_{u_1,\ldots,u_{k-1}} \E \Bigl[|\fB_v(v,n)| \mid u_1,\ldots,u_{k-1} \in \fB_v(v,n)\Bigr],
  \end{multline*}
  where the supremum is taken over all collections of vertices $u_1,\ldots,u_{k-1} \in V$ such that the probability $\P(u_1,\ldots,u_{k-1} \in \fB_v(v,n))$ is positive. \qedhere

\end{proof}

% Let us also record the following lemma, the proof of which is identical to that of equation \eqref{eq:momentinduction} and which will be applied in \cref{sec:exponents}.

% \begin{lemma}
%   \label{lem:past2ndmoment}
% Let $G$ be a network. Then
%   % \vspace{0.1em}
%   \[
%   \vspace{0.4em}
%     \E  \Bigl[ |\fP(v,n)|^k \Bigr] 
%   % 
%     \leq (k-1) (n+1) \, \E  \Bigl[ |\fP(v,n)|^{k-1} \Bigr] \sup_{u \in V} \E  \Bigl[ |\fB_u(u,n)| \Bigr]
%   \]
%   for every $v \in V$, $n\geq 0$ and $k\geq 1$. 
% \end{lemma}

Next, we control the moments of the volume of balls in the WUSF in terms of the moments in the $v$-WUSF.

\begin{lemma}
  \label{lem:unrootedmoments}
  Let $G$ be a transient network. Then
  \begin{equation}
  \label{eq:unrootedmoment}
 \sup_{v\in V} \E\left[|\fB(v,n)|^k\right] \leq 
  % \adjustlimits 
  \sum_{\substack{k_0,\,\ldots,\,k_n \geq 0:\\ k_0+ \cdots + k_n=k}} \, \prod_{i=0}^n\; \sup_{v\in V}\E\left[|\fB_v(v,n)|^{k_i}\right]
  \end{equation}
  for every $n\geq 0$ and $k\geq 1$.
\end{lemma}

\begin{proof}
  Let $v \in V$ and let $\Gamma(v,\infty)$ be the future of $v$ in $\F$. 
   Let $v=u_0,\ldots,u_n$ be the first $n+1$ vertices in the path $\Gamma(v,\infty)$. 
  For each $0\leq i \leq n$, let $W_i=\{w_{i,1},\ldots,w_{i,m_i}\}$ be a finite (possibly empty) collection of vertices of $G$, and let $\sA_i$ be the event that for every vertex $w \in W_i$, $w$ is in $\fB(v,n)$ and that the path connecting $w$ to $v$ first meets $\Gamma(v,\infty)$ at $u_i$.

  Let $\{ X^{i,j} : 0 \leq i \leq n, m_i\neq 0, 1\leq j \leq n \}$ be a collection of independent random walks, independent of $\Gamma(v,\infty)$, such that $X^{i,j}_0=w_{i,j}$ for each $0 \leq i \leq n$ such that $m_i \neq 0$ and each $1 \leq j \leq m_i$. For each $0\leq i \leq n$ such that $m_i \neq 0,$ let $\sB_i$ be the event that, if we sample $\F_{u_i}$ using Wilson's algorithm, starting with the random walks $X^{i,1},\ldots,X^{i,m_i}$, then every vertex in $W_i$ is connected to $u_i$ in $\F_{u_i}$ by a path of length at most $n$. Observe that if we sample $\F$ conditional on $\Gamma(v,\infty)$ using Wilson's algorithm starting with $X^{0,1},\ldots,X^{0,m_0}$, then $X^{1,1},\ldots,X^{1,m_1}$, and so on, then we have the containment $\sA_i \subseteq \sB_i$. We deduce that
  \[\P\big(\cap_{i=0}^n\sA_i \mid \Gamma(v,\infty)\big) \leq \P\big(\cap_{i=1}^n \sB_i \mid \Gamma(v,\infty) \big) = \prod_{i=0}^n \P(\sB_i \mid \Gamma(v,\infty)).\] 
  % and hence that
  Summing over the possible choices of the sets $W_i$ such that $\sum_{i=0}^n |W_i| =k$, we obtain that
  \[\E \left[|\fB(v,n)|^k \mid \Gamma(v,\infty) \right] \leq \sum_{\substack{k_0,\,\ldots,\,k_n \geq 0:\\ k_0+ \cdots + k_n=k}} \prod_{i=0}^n \E \left[|\fB_{u_i}(u_i,n)|^{k_i}\right],\]
  and the claim follows. \qedhere
\end{proof}

We now prove \cref{thm:moments}.

\begin{proof}[Proof of \cref{thm:moments,thm:vmoments}]
  The moment estimate \eqref{eq:thmmomentsrooted} follows immediately from \cref{lem:expectedballgrowth1,lem:ball2ndmoment}. 
  In order to prove the moment generating function estimates \eqref{eq:thmmomentsrootedexp} and \eqref{eq:thmmomentsunrootedexp}, define
  \begin{align*}
    \Phi(n,t) &= \sum_{k = 0}^{\infty} \frac{1}{k!} \left(\frac{t}{\alpha(n+1)^2}\right)^k \sup_{v\in V} \E \left|\fB(v,n)\right|^k
  % \]
  \intertext{and}  
  % \[
    \Psi(n,t) &= \sum_{k = 0}^{\infty} \frac{1}{k!}\left(\frac{t}{\alpha(n+1)^2}\right)^k \sup_{v\in V} \E \left|\fB_v(v,n)\right|^k.
  \end{align*}
    The moment estimate \eqref{eq:thmmomentsrooted} implies that
    \[
     \Psi(n,t) \leq 1+ \frac{1}{n+1} \sum_{k=1}^\infty \frac{t^k}{k}  = 1 - \frac{\log(1-t)}{n+1}
     \]
     for every $n \geq 0$ and $t\in [0,1)$. The moment generating function estimate 
    % The moment generating function estimate
     \eqref{eq:thmmomentsrootedexp} follows immediately.
      % from the moment estimates [ref] by taking power series expansions.
       Next, \cref{lem:unrootedmoments} implies that
  \[
  \Phi(n,t) \leq \sum_{k=0}^\infty \frac{1}{k!} \left(\frac{t}{\alpha(n+1)^2}\right)^k \sum_{\substack{k_0,\,\ldots,\,k_n \geq 0:\\ k_0+ \cdots + k_n=k}} \, \prod_{i=0}^n\; \sup_{v\in V}\E \left[ |\fB_v(v,n)|^{k_i} \right].
    % \sup_{v\in V} \E\left[ \exp\left(\frac{t}{\alpha(n+1)^2} |\fB(v,n)| \right)\right]
    % \leq
    % \left(\sup_{v\in V} \E \left[\exp\left(\frac{t}{\alpha(n+1)^2} |\fB(v,n)| \right)\right]\right)^{n+1},
  \]
  On the other hand, we have that
  \[
  \Psi(n,t)^{n+1}
  = \sum_{k=0}^\infty \sum_{\substack{ k_0,\ldots,k_n \geq 0 \\ k_0+ \cdots + k_n=k}} \prod_{i=0}^n \frac{1}{k_i!}\left(\frac{t}{\alpha(n+1)^2}\right)^{k_i} \sup_{v\in V}\E\left[\left| \fB_v(v,n)\right|^{k_i} \right],
  \]
  and since $\prod_{i=0}^n(k_i)! \leq k!$ whenever $k_0,\ldots,k_n$ are non-negative integers summing to $k$, we deduce that
  \[
  \Phi(n,t) \leq \Psi(n,t)^{n+1} 
  \]
  for every $n \geq 0$ and $t\geq 0$. Thus, it follows that
  \[
  \Phi(n,t) \leq \left( 1 - \frac{\log(1-t)}{n+1}\right)^{n+1} \leq \frac{1}{1-t}
  \]
  for every $n \geq 0$ and $t\in [0,1)$, where the final inequality follows from the elementary inequality $1-x \leq e^{-x}$.  
  This yields the moment generating function estimate \eqref{eq:thmmomentsunrootedexp}.

   Finally, to deduce \eqref{eq:thmmomentsunrooted}, we use the fact that, since $|\fB(v,n)|$ is non-negative,
  \[
     \E\left[ \exp\left(\frac{t}{\alpha(n+1)^2} |\fB(v,n)| \right)\right] \geq \frac{t^k \, \E\bigl[|\fB(v,n)|^k\bigr]}{\alpha^k (n+1)^{2k}k!}
  \]
  and hence
  \[
    \E\left[|\fB(v,n)|^k\right] \leq  \frac{k!\alpha^k (n+1)^{2k}}{t^k(1-t)} 
  \]
  for all $k\geq 1$ and $t\in [0,1)$. Optimizing by taking $t=k/(k+1)$ and using that $(1+x^{-1})^x \leq e$ for every $x\geq 0$ yields \eqref{eq:thmmomentsunrooted}. \qedhere

\end{proof}

\subsection{Lower bounds}
\label{subsec:volume_lower}

% In this section we prove upper bounds on the probability that the volume
In this section, we give lower bounds on the first moment of the volume of the past, and derive upper bounds on the probability that the volume of an intrinsic ball is atypically small.

We begin with the following simple lower bounds on the first moments.
We write $\fP(v,n)$ for the ball of radius $n$ around $v$ in the past of $v$ in $\F$, and write $\partial \fP(v,n)$ for the set of points that are in the past of $v$ in $\F$ and have intrinsic distance exactly $n$ from $v$.
% , which we can prove immediately.

% $\mathfrak{P}$

\begin{lemma}
\label{lem:expectedballgrowth2}
 Let $G$ be a transient network, and let $\F$ be the wired uniform spanning forest of $G$. Then % \begin{enumerate}[leftmargin=*]
% \itemsep1em
  % \item If $G$ is transient, then 
  %  \[\E |\partial \mathfrak{B}_v(v,n)| \geq  \frac{c(v)}{\sup_{u\in V}c(u)}\bP_v(\tau^+_v=\infty) \]
  % % for all $v\in V$ and $n\geq 1$.
  % and similarly
  \[\E |\partial \fP(v,n)| \geq  \frac{q(v)c(v)}{\sup_{u\in V}c(u)} \]
  for every $v\in V$ and $n\geq 0$. Similarly, if $\F_v$ is the $v$-wired uniform spanning forest of $G$ then
  \[
\E|\partial \fB_v(v,n)| \geq \frac{c(v)}{\sup_{u\in V}c(u)}.
  \]
  % \item If $\bubnorm{P}<\infty$ then 
\end{lemma}

\begin{remark}
If $G$ is a transitive unimodular graph, the mass-transport principle yields the exact equality $\E |\partial \mathfrak{P}(v,n)| =1$ for every $n\geq 1$.
\end{remark}

\begin{proof}
We prove the first claim, the proof of the second being similar. 
Let $u\in V$, and let $X$ be a random walk started at $u$. Consider sampling $\F$ using Wilson's algorithm, starting with the walk $X$.  
Let $\sA_n(u,v)$ be the event that $X$ hits $v$, that the sets $\{ X_m : 0\leq m < \tau_v \}$ and $\{ X_m : m \geq \tau_v \}$ are disjoint, and that $|\LE(X^{\tau_v})|=n$, so that $u\in \partial \fP(v,n)$ on the event $\sA_n(u,v)$ and hence 
\[
\E\Bigl[\left|\partial \fP(v,n)\right|\Bigr] \geq \sum_{u\in V}\P\left(\sA_n(u,v)\right).
\]
Let $Y$ be an independent random walk started at $v$. By time-reversal, we have that
\[\bP_u(\sA_n(u,v)) \geq \frac{c(v)}{c(u)}\bE_v\left[\#\{k: X_k=u,\, |\LE(X^k)| = n\} \mathbbm{1}(\{X_i : i > 0 \} \cap \{Y_i : i \geq 0\} =\emptyset)\right]
\]
% \bP_u(\sA(v)) \geq \frac{c(v)\bP_v(\tau^+_v=\infty)}{c(u)}\sum_{k=0}^np_k(v,u)\]
% \medskip
and hence, summing over $u \in V$,
\begin{align*}
\E |\partial \fP(v,n)| &\geq \frac{c(v)q(v)}{\sup_{u\in V} c(u)}
\mathbf{E}_v\Bigl[\#\{k: |\LE(X^k)| = n\} \mid \{X_i : i > 0 \} \cap \{Y_i : i \geq 0\} =\emptyset\Bigr]
\\
&\geq \frac{c(v)q(v)}{\sup_{u\in V} c(u)}
% \sum_{u\in V}\sum_{k=0}^np_k(v,u) = \frac{c(v)\bP_v(\tau^+_v=\infty)}{\sup_{u\in V}c(u)} n\]
\end{align*}
as claimed, where the second inequality follows since there must be at least one time $k$ such that $|\LE(X^k)|=n$, namely the time $\ell_n$.
\end{proof}

% \subsection{Lower bound on volume growth}

Our next goal is to prove the following lemma.

\begin{lemma}
\label{lem:fullvolumelower}
Let $G$ be a network with controlled stationary measure and $\bubnorm{P}<\infty$, and let $\F$ be the wired uniform spanning forest of $G$. 
 Then
  \[ \P\bigl(|\fB(v,n)| \leq \lambda^{-1}n^2\bigr) \preb \lambda^{-1/8}\]
  for every vertex $v$, every $n\geq 0$ and every $\lambda \geq 1$.
\end{lemma}

\cref{lem:fullvolumelower} has the following immediate corollary, which is proved similarly to \cref{cor:quenchedvolumeupper} and which together with \cref{cor:quenchedvolumeupper} establishes that $d_f(T)=2$ for every component of $\F$ almost surely (as claimed in \cref{thm:AlexanderOrbach}). We remark that Barlow and J\'arai \cite{barlow2016geometry} established much stronger versions of \cref{lem:fullvolumelower,cor:quenchedvolumelowerbound} in the case of $\Z^d$, $d\geq 5$. 

\begin{corollary}
\label{cor:quenchedvolumelowerbound}
Let $G$ be a network with controlled stationary measure with $\bubnorm{P}<\infty$, and let $\F$ be the wired uniform spanning forest of $G$. Then
 % satisfying \eqref{eq:heatkernel} for some $d>4$. 
 % Then for every $\eps>0$ we have that
\[\liminf_{n\to\infty} \frac{\log^{8+\eps} n}{n^2} |\fB(v,n)| > 0 \qquad \]
almost surely for every $v\in V$ and $\eps>0$.
\end{corollary}

% \begin{proof}
% Suppose we generate $\F$ using Wilson's algorithm, starting with a random walk 
% $X$ started at $v$, so that $\Gamma=\LE(X)$. Since $\rho(n/2) \leq n/2$ for every $n\geq 1$, it follows from \cref{lem:aslinearcapgrowth} that
% % By [ref], 
% there exists a positive constant $c=c(G)$ such that
% \begin{equation}
% \label{eq:quenchedvolumelowerboundcapcopy}
% \liminf_{n\to\infty} \frac{\Cap(\Gamma^{n/2},\Gamma)}{n} \geq c 
% \end{equation}
% almost surely. 
% Let $n_k=2^k$. Applying \cref{lem:volumelowerboundquant} we have that
% \begin{align*}
% \P\left(|\fB(v,n_k)| \leq \frac{n_k^2}{(\log n_k)^{1+\eps}} \; \mid \; \Cap(\Gamma^{n/2},\Gamma) \geq c/2 \right)
% \preceq \frac{1}{k^{1+\eps}},
% \end{align*}
% and so it follows by \eqref{eq:quenchedvolumelowerboundcapcopy} and Borel-Cantelli that
% \[\liminf_{k\to\infty} \frac{(\log n_k)^{1+\eps}}{n_k}|\fB(v,n_k)| \geq 1 \]
% almost surely.
% The claim follows since $|\fB(v,n)|$ is increasing, and for every $n\geq 0$ there exists $k\geq 0$ such that $n_k \leq n \leq 2 n_k$.
% \end{proof}

% \begin{proof}[Proof of \cref{thm:volumegrowth}]
% The annealed lower bound follows from \cref{lem:volumelowerboundquant,eq:lemcapquant}, while the annealed upper bound follows from [ref].
% The quenched lower bound is provided by \cref{cor:quenchedvolumelowerbound}, while the upper bound is provided by \cref{cor:quenchedvolumeupper}.
% % 
% \end{proof}

% We now turn to the proof of \cref{lem:fullvolumelower}. 

 In order to prove \cref{lem:fullvolumelower}, we first show that the volume of the tree can be lower bounded with high probability in terms of quantities related to the capacity of the spine, and then show that these quantities are large with high probability.

Given two sets of vertices $A \subseteq B$ in a transient graph and $k \in [0,\infty]$, we define 
% $\Cap_k(A,B)$ to be
\[\Cap_k(A,B):=\sum_{v\in A} c(v) \bP_v(\tau^+_B \geq k),\]
so that $\Cap(A)=\Cap_\infty(A,A)$.

% Recall that $\ell_i$ are the times of the random walk that contribute to the infinite loop-erasure. 

\begin{lemma}
  \label{lem:volumelowerboundquant}
  Let $G$ be a network with controlled stationary measure. Let $v$ be a vertex of $G$ and let $\Gamma=\Gamma(v,\infty)$ be the future of $v$ in $\F$.

  \begin{enumerate}[leftmargin=*]
  \itemsep1em
    \item
    The estimate
       % Then there exists a constant $C=C(M)$ such that
      \begin{equation}
        \label{eq:expectedvolumelowerboundfulltree}
        \E\Big[|\fB(v,2n)| \;\Big|\; \Gamma\Big] \succeq \, (k+1) \, \Cap_k\bigl(\Gamma^{n},\Gamma\bigr)
      \end{equation}
      holds for every $n\geq 0$ and $ 0 \leq k \leq n$.

    \item 
      If $\bubnorm{P}<\infty$, then
      \begin{equation}
        \label{eq:smallvolumeprobability}
        \P\biggl(|\fB(v,2n)| \leq \frac{k\Cap_k(\Gamma^{n},\Gamma)}{2 \sup_{u\in V} c(u)} \; \bigg| \; \Gamma \biggr) 
         \preceq  \frac{(k+1)(n+1)}{\Cap_k(\Gamma^{n},\Gamma)^2} 
       \end{equation}
       for every $1\leq k \leq n$.
  \end{enumerate}
\end{lemma}

\begin{proof}
Let $v=v_0,v_1,\ldots$ be the vertices visited by the path $\Gamma$. 
Let $\fT_i(k)$ be the set of vertices that are connected to $v$ in $\F$ by a path that first meets $\Gamma$ at $v_i$, and such that the path connecting them to $\Gamma$ has length at most $k$. Clearly
\[|\fB(v,2n)| \geq \sum_{i=0}^{n} |\fT_i(k)|\]
for every $n\geq 0$ and every $0\leq k \leq n$.
% 
% Let $\beta \in [0,1/2]$, and let $V(n,\beta)$ be the number of points that are connected to $v$ in $\F$ by a path that takes at most $m$ steps outside of $\Gamma$, and at most $n/2$ steps inside $\Gamma$. Clearly $V(n,\beta)\leq |\fB(v,n)|$. 
% We  claim that
% \begin{equation}
  % \label{eq:Tbetamean}
  % \E \left[ \sum_{i=1}^{n} |\fT_i(m)| \; \mid \; \Gamma \right] \succeq m \Cap_m(\Gamma^{n},\Gamma)
% \end{equation}
% for all $v\in V$, and $0\leq m \leq k \leq n$. 
% , which already suffices to establish th.
% Indeed, 
Let $u$ be a vertex of $G$, and suppose that we sample $\F$ using Wilson's algorithm, starting with the vertices $v$ and $u$. For $u$ to be included in $\fT_i(k)$, it suffices for the random walk started at $u$ to hit $\Gamma$ for the first time at $v_i$, and to do so within time $k$. By a time-reversal argument similar to that used in the proof of \cref{lem:expectedballgrowth1}, it follows that
% \vspace{0.4em}
\begin{align}
\label{eq:Tbetamean}
% \vspace{0.4em}
\E\Bigg[ \sum_{i=0}^{n} |\fT_i(k)| \; \Bigg| \; \Gamma \Bigg] \geq \frac{\sum_{i=0}^{n} c(v_i)(k+1)\bP_{v_i}(\tau_{\Gamma}^+ \geq k+1)}{\sup_{u\in V} c(u)}  \geq \frac{(k+1) \Cap_{k+1}(\Gamma^{n},\Gamma)}{\sup_{u\in V} c(u)}.
\end{align}
The estimate \eqref{eq:expectedvolumelowerboundfulltree} follows immediately. 

% \medskip

% On the other hand,  for every $u_1,u_2 \in V$ and $0 \leq i < j \leq n$ we have that
% \[\P\left( u_1 \in \fT_i(m) \mid \Gamma,\, u_2 \in \fT_j(m)\right) \leq \P\left( u_1 \in \fT_i(m) \mid \Gamma\right). \]
Let $u_1,u_2 \in V$ and $0 \leq i < j \leq n$.
By running Wilson's algorithm starting first with $v$ and then with $u_1$, we see that the probability that $u_1 \in \fT_i(m)$ given $\Gamma$ is equal to the probability that a random walk started at $u_1$ hits $\Gamma$ for the first time at $v_i$, and that the loop-erasure of the walk stopped at this time has length at most $k$. On the other hand, by running Wilson's algorithm starting with $v$, $u_2$, and $u_1$ (in that order), we see that the conditional probability that $u_1 \in \fT_i(k)$ given $\Gamma$, $\Gamma(u_2,\infty)$, and the event that $u_2 \in \fT_j(k)$  is equal to the probability that a random walk started at $u_1$ hits \emph{the union of $\Gamma$ and $\Gamma(u_2,\infty)$} for the first time at $v_i$ and that the loop-erasure of the walk stopped at this time has length at most $k$. This second conditional probability is clearly smaller than the first, and we deduce that
\[\P\left( u_1 \in \fT_i(k) \mid \Gamma,\, u_2 \in \fT_j(k)\right) \leq \P\left( u_1 \in \fT_i(k) \mid \Gamma\right). \]
 % walk must also avoid the path connecting $u_2$ to $\Gamma$.
% 
It follows that
\[\E\left[|\fT_i(k)| \cdot |\fT_j(k)| \mid \Gamma\right] \leq \E\left[|\fT_i(k)|  \mid \Gamma\right] \E\left[|\fT_j(k)|  \mid \Gamma\right]\]
for every $0\leq i < j \leq n$, and hence that
\begin{align*}
\Var\Biggl[ \sum_{i=0}^{n} |\fT_i(k)| \; \Bigg| \; \Gamma \Biggr] &\leq \sum_{i=0}^{n}\E\left[ |\fT_i(k)|^2 \mid \Gamma \right] \leq 
\sum_{u \in V} \mathbbm{1}(u \in \{v_0,v_1,\ldots,v_n\}) \E\left[ |\past_{\F \setminus \Gamma}(u,k)|^2 \mid \Gamma \right]\\
&\leq \sum_{u \in V} \mathbbm{1}(u \in \{v_0,v_1,\ldots,v_n\}) \sup_{w\in V} \E\left[|\mathfrak{B}_w(w,k)|^2\right]\\ &=
(n+1) \cdot \sup_{w\in V} \E\left[|\mathfrak{B}_w(w,k)|^2\right],
\end{align*}
where  we have applied \cref{lem:domination} (more specifically, \eqref{eq:dom1} with $K=\{v,u\}$) in the third inequality.
If $\bubnorm{P}<\infty$, we deduce from \cref{thm:vmoments} that
% and hence that
\begin{align*}
\Var\Biggl[ \sum_{i=0}^{n} |\fT_i(k)| \; \Bigg| \; \Gamma \Biggr]
% \leq  8 M^2 \bubnorm{P}^2 \beta^3 (n+1)^{3}\\
\preceq (k+1)^3 (n+1),
\end{align*}
% where 
 % now implies that
% 
and 
applying Chebyshev's inequality yields that
\[
\P\left(\sum_{i=0}^{n} |\fT_i(k)| \leq \frac{(k+1) \Cap_k(\Gamma^{n},\Gamma)}{2 \sup_{u\in V} c(u)} \;\Bigg|\; \Gamma\right) \leq \myfrac[0.5em]{4 \Var\Bigl[ \sum_{i=0}^{n} |\fT_i(k)| \mid \Gamma\Bigr]}{\E\Bigl[ \sum_{i=0}^{n} |\fT_i(k)| \mid \Gamma\Bigr]^2}
\preceq \frac{(k+1)(n+1)}{\Cap_k(\Gamma^{n},\Gamma)^2}.
\]
for all $1 \leq k \leq n$ as claimed.
% for $\beta \leq k/n$. The estimate \eqref{eq:smallvolumeprobability} follows by taking $\beta = ((2\lambda^{-1} \sup_{u \in V} c(u))^{-1})\wedge 1) k/n$. 
 \qedhere
\end{proof}

\begin{lemma}
\label{lem:relativecapLERW}
Let $G$ be a network with controlled stationary measure and with $\bubnorm{P}<\infty$. Let $v$ be a vertex of $G$, let $X$ be a random walk started at $v$ and let $\Gamma=\LE(X)$. Then for every $1 \leq k \leq n$ we have that
\begin{equation}
\label{eq:relativecapexpectation}
\E\Bigl[\Cap_k\bigl(\Gamma^{n}\bigr)-\Cap_k\bigl(\Gamma^{n},\Gamma\bigr)\Bigr] \preceq k^{1/2}n^{1/2}
\end{equation}
and that 
\begin{equation}
\label{eq:relativecapprob}
\P\left(\Cap_k\bigl(\Gamma^{n},\Gamma\bigr) \leq \lambda^{-1} n\right) \preceq \lambda^{-1/2}+\lambda k^{1/2} n^{-1/2}.
\end{equation}
for every $1 \leq k \leq n$ and $\lambda \geq 1$.
\end{lemma}

\begin{proof}
Let $m\leq n$, and let 
 $A = \{\Gamma_i : 0 \leq i \leq n\}$, $B= \{\Gamma_i : 0 \leq i \leq n-m\}$, and $C = \{\Gamma_i : i \geq n+1\}$. Considering the contribution to $\Cap_k\bigl(\Gamma^n\bigr)-\Cap_k\bigl(\Gamma^{n},\Gamma)$ of  the last $m$ steps and the first $n-m$ steps of $\Gamma^n$ separately, we obtain that
\begin{multline*}
\Cap_k\bigl(\Gamma^n\bigr)-\Cap_k\bigl(\Gamma^{n},\Gamma)
=\sum_{u \in A}c(u)\bP_u(\tau_C < k)
\leq \sum_{u \in A \setminus B} c(u) + \sum_{u \in B}c(u)\bP_u(\tau_C<k) \\
\preceq  
m+\sum_{u\in B}\sum_{w\in C} \sum_{\ell=0}^k p_\ell(u,w).
\end{multline*}
% 
% \leq \sum_{u\in V} \mathbbm{1}\bigl(u\in \{\Gamma_i : 0 \leq i \leq n\}\bigr) c(u)\bP_{u}\bigl(\text{hit }\{\Gamma_j : j \geq n+1\}\bigr)
% \\\leq \sum_{i=0}^{n-m} c(\Gamma_i)\bP_{\Gamma_i}\bigl(\text{hit }\{\Gamma_j : j \geq n+1\}\bigr) + \sum_{i=n-m+1}^n c(\Gamma_i),
% \end{multline*}
% and hence that
% \[
% \Cap_k\bigl(\Gamma^n\bigr)-\Cap_k\bigl(\Gamma^{n},\Gamma) \leq \left(\sup_{u\in V} c(u)\right)\left[\sum_{i=0}^{n-m} \sum_{j = n+1}^\infty \sum_{\ell=0}^k  p_\ell(\Gamma_i,\Gamma_j)+m\right].
% \]
Taking expectations over $\Gamma$, we obtain that
\begin{equation*}\E\left[\Cap_k(\Gamma^n)-\Cap_k(\Gamma^{n},\Gamma)\right] \preceq m  + 
\sum_{u,w \in V} \P(u \in B, w \in C) \sum_{\ell=0}^k p_\ell(u,w).
% \sum_{0\leq i \leq m} \sum_{j > n} \sum_{\ell=0}^k  p_\ell(\Gamma_i,\Gamma_j),
\end{equation*}
On the event $u \in B$ we have that $\E[\,\tau_u\mid \Gamma\,] \leq \E[\,\ell_n \mid \Gamma\,]$, and hence by 
\cref{lem:LoopLength} that $\E[\,\tau_u \mid \Gamma\,]\leq \bubnorm{P} n$ on the event that $u\in B$. It follows by Markov's inequality that $\tau_u \leq 2\|P \|_\mathrm{bub} n$ with probability at least $1/2$  
conditional on $\Gamma$ and the event that $u\in B$, and in particular 
that $\tau_u \leq 2\|P \|_\mathrm{bub} n$ with probability at least $1/2$
conditional on the event that $u\in B$ and $w\in C$. Moreover, on this event we must have that $w$ is visited (not necessarily for the first time) by $X$ some time at least $m$ steps after $\tau_{u}$.
Thus, we obtain that
\begin{align*}\E\left[\Cap_k(\Gamma^n)-\Cap_k(\Gamma^{n},\Gamma)\right] &\preceq m  + 
 \sum_{u\in V} \P\Bigl(\tau_{u} \leq 2 \bubnorm{P}n \Bigr) \sum_{w\in V}\sum_{r\geq m} \sum_{\ell=0}^k p_r(u,w)  p_\ell(u,w)\\
 &=
 m  + 
 \sum_{u\in V} \P\Bigl(\tau_{u} \leq 2 \bubnorm{P}n \Bigr) \sum_{r\geq m} \sum_{\ell=0}^k \|P^{r+\ell}\|_{1\to\infty}\\
&\preceq m  + 
 k n \sum_{j\geq m} \|P^j\|_{1\to\infty} \preceq m + knm^{-1}.
% \leq m \sup_{v\in V} c(v) + 
% 2 k \sum_{u\in V} \P\Bigl(\tau_{u} \leq \bubnorm{P}m \Bigr) \sum_{r\geq m} \|P^r\|_{1\to\infty}
\end{align*}
 Taking $m=\lceil k^{1/2}n^{1/2} \rceil$ concludes the proof of \eqref{eq:relativecapexpectation}.

To obtain \eqref{eq:relativecapprob}, simply take a union bound and apply \eqref{eq:relativecapexpectation} and \eqref{eq:smallcapLERWprob} to get that, since $\Cap_k(\Gamma^n) \geq \Cap(\Gamma^n)$,
\begin{align*}\P\left(\Cap_k\bigl(\Gamma^{n},\Gamma\bigr) \leq \lambda^{-1} n \right)&\leq \P\left(\Cap(\Gamma^n) \leq 2 \lambda^{-1} n \right) + \P\left(\Cap_k(\Gamma^n)-\Cap_k(\Gamma^n,\Gamma) \geq \lambda^{-1} n \right)
\\
&\preceq \lambda^{-1/2} + \lambda k^{1/2} n^{-1/2}. \qedhere
 \end{align*}
\end{proof}

\begin{proof}[Proof of \cref{lem:fullvolumelower}]
% If $k\leq n$ and $1\leq \mu \leq \lambda$ are such that $\lambda k \geq \mu n$, 
Take $k=\lfloor \lambda ^{-3/4} n \rfloor$ and 
\[\eps = \frac{2n \sup_{u\in V}c(u)}{\lambda k} \asymp \lambda^{-1/4}, \qquad \text{so that } \qquad \frac{k \eps n}{2 \sup_{u \in V} c(u)} = \lambda^{-1}n^2.\] 
Then it follows from \cref{lem:volumelowerboundquant} and \cref{lem:relativecapLERW} that if $\lambda$ is sufficiently large that $\eps \leq 1$ then
\begin{align*}
\P\Bigl(|\fB(v,2n)|\leq \lambda^{-1} n^2 \Bigr) &\leq \P\left(\Cap_k(\Gamma^{n},\Gamma) \leq \eps n \right) + \P\left(|\fB(v,2n)|\leq  \lambda^{-1}n^2, \Cap_k(\Gamma^{n},\Gamma) \geq \eps n\right)\\
&\preceq 
\eps^{1/2} + \eps^{-1}k^{1/2}n^{-1/2} + \eps^{-2}kn^{-1} \preceq \lambda^{-1/8},
% \\
% &\preceq 
% \mu^{-1/2}+\mu k  n^{-1} + \mu^3 \lambda^{-1}
% \lambda^{-1/2}k^{-1/2}n^{1/2} + \lambda k^{3/2} n^{-3/2} +  \lambda^2 k^3 n^{-2} \preceq \lambda^{-1/8}
 % \left(\frac{\lambda k}{\mu n}\right)^{-1} \frac{nk}{\mu^{-2} n^2}.
\end{align*}
% Taking $k/n=\mu/\lambda$, we get that
% \begin{align*}
% \P\bigl(|\fB(v,n)|\leq \lambda^{-1} n^2 \bigr) \preceq \mu^{-1/2} + \mu^{2}\lambda^{-1} + \mu^3 \lambda^{-1} \preceq \lambda^{-1/2}k^{-1/2}n^{1/2} + \lambda k^2 n^{-2} +  \lambda^2 k^3 n^{-2} \preceq \lambda^{-1/8}
% \end{align*}
% and the claim follows by taking $\mu=\lambda^{2/5}$.
and the claim follows easily.
\end{proof}
% \medskip

\section{Critical exponents}
\label{sec:exponents}

% \subsection{An exponent lower bound}

In this section we apply the estimates obtained in \cref{sec:LERWCap,sec:volume} to complete the proofs of \cref{thm:generalexponents,thm:extrinsicZd,thm:extrinsic,thm:extrinsicspeed}.

\medskip

We will also prove the following extensions of these theorems to the $v$-wired case.

\begin{theorem}
  \label{thm:generalexponentsv}
  Let $G$ be a network with controlled stationary measure such that $\bubnorm{P}<\infty$, let $v$ be a vertex of $G$ and let $\F_v$ be the $v$-wired uniform spanning forest of $G$. Then 
  \vspace{0.2em}
  \[ 
  \vspace{0.2em}
    \P\Bigl(\diam_\mathrm{int}\bigl(\fT_v\bigr) \geq R \Bigr)
      \asymp R^{-1}
  \quad\text{and}\quad
    \P\Bigl(\bigl|\fT_v\bigr| \geq R \Bigr) \asymp R^{-1/2}
  \]
  for all $v\in V$ and $R\geq1$.
\end{theorem}

\begin{theorem}
\label{thm:extrinsicZdv}
Let $d\geq 5$, and let  $\F_0$ be the $0$-wired uniform spanning forest of $\Z^d$. 
   Then
  \vspace{0.25em}
  \[ 
    \vspace{0.25em}
    \P\left(\diam_{\mathrm{ext}}(\fT_0) \geq R \right)  \asymp R^{-2}
    % 
    % \quad \text{ and } \quad
    % 
    % \P\left(|\fP(v)| \geq R \right) \asymp R^{-1/2}
  \]
  for every $R\geq 1$. 
\end{theorem}

\begin{theorem}
\label{thm:extrinsicv}
Let $G$ be a network with controlled stationary measure that is $d$-Ahlfors regular for some $d>4$  and that satisfies Gaussian heat kernel estimates.
% vertex set $\Z^d$ and conductances and let $\F$ be the wired uniform spanning forest of $\Z^d$. 
   Let $v\in V$ and let $\fF_v$ be the wired uniform spanning forest of $G$. 
   Then
  \vspace{0.25em}
  \[ 
    \vspace{0.25em}
    R^{-2} \preceq \P\left(\diam_{\mathrm{ext}}(\fT_v) \geq R \right)  \preceq R^{-2}\log R
    % 
    % \quad \text{ and } \quad
    % 
    % \P\left(|\fP(v)| \geq R \right) \asymp R^{-1/2}
  \]
  for every vertex $v$ and every $R\geq 1$. 
\end{theorem}

\begin{theorem}
\label{thm:extrinsicspeedv}
Let $G$ be a uniformly ballistic network with controlled stationary measure and $\bubnorm{P}<\infty$, let $v\in V$ and let $\fF_v$ be the wired uniform spanning forest of $G$. 
   Then
  \vspace{0.25em}
  \[ 
    \vspace{0.25em}
   \P\left(\diam_{\mathrm{ext}}(\fT_v) \geq R \right)  \asymp R^{-1}
    % 
    % \quad \text{ and } \quad
    % 
    % \P\left(|\fP(v)| \geq R \right) \asymp R^{-1/2}
  \]
  for every vertex $v$ and every $R\geq 1$. 
\end{theorem}

\subsection{The intrinsic diameter: upper bounds}

% In this section we prove the upper bound of the  \cref{thm:generalexponents}.
 The key estimate is provided by the following lemma, which will allow us to prove the upper bound on the probability of a large intrinsic diameter by an inductive argument. 
 % The remaining estimates will follow by combining this with the first and second moment upper bounds from \cref{thm:moments} and the first moment lower bounds from \cref{lem:volumelowerboundquant}.
For non-negative integers $n$, we define
\[Q(n) =
     \sup_{v\in V}\P \left( |\partial \fB_v(v,n)| \neq \emptyset\right),
\]
to be the supremal probability that the component of $v$ in the $v$-WUSF survives to intrinsic distance $n$.
% , and more generally define $Q(t)=Q(\lfloor t \rfloor)$ for $t\geq 0$.
\begin{lemma} 
	\label{lem:exponentsinduction}
	Let $G$ be a network with controlled stationary measure satisfying $\bubnorm{P}<\infty$. Then there exist positive constants $N$ and $C$ such that
	\begin{equation}
	\label{eq:exponentsinduction}
	Q(3n) \leq \frac{C}{n} + \frac{1}{6} Q(n). 
	\end{equation}
	for all $n\geq N$.
\end{lemma}

Before proving this lemma, let us establish the following corollary of it.

\begin{corollary} 
	\label{cor:intupperbound}
	Let $G$ be a network with controlled stationary measure satisfying $\bubnorm{P}<\infty$, and let $Q(n)$ be as above. Then 
		$Q(n) \preceq n^{-1}$
	for all $n\geq 1$.
\end{corollary}

\begin{proof}
	 Let $N=N(G)$ and $C=C(G)$ be as in \cref{lem:exponentsinduction}. We may assume that $C,N\geq 1$. Since $Q(n)$ is a decreasing function of $n$, it suffices to prove that
	\begin{equation}
		\label{eq:Qnbound}
		Q(3^k) \leq 6 C N 3^{-k}
	\end{equation}
	for every $k\geq 1$. We do this by induction on $k$. When $3^k \leq N$ the claim holds trivially. If $3^k>N$ and the claim holds for all $\ell<k$, then we have by \eqref{eq:exponentsinduction} and the induction hypothesis that
	\begin{align*}
		Q(3^k) \leq 
		% Cn^{-1} + Cn^{-1} \log (n \frac{C'}{\lceil n/3\rceil}) + \frac{C}{6\lceil n/3 \rceil} \\
		 C3^{-k+1} + CN 3^{-k+1} 
     \leq 6CN 3^{-k},
	\end{align*}
	completing the induction. 
  % where the second inequality follows from the choice of $C'$. 
\end{proof}

\medskip

We now turn to the proof of \cref{lem:exponentsinduction}. We will require the following estimate. Recall that $\Gamma_v(u,\infty)$ denotes the future of $u$ in $\F_v$, which is equal to the path from $u$ to $v$ if $u\in \fT_v$.

\begin{lemma}
	\label{lem:ln}
	Let $G$ be a network with controlled stationary measure such that $\bubnorm{P}<\infty$. Then
	\begin{equation*}
		E(n,\delta):= \sup_{v\in V}\E  \left|\Bigl\{ u \in \fB_v(v,2n) \setminus \fB_v(v,n) : \Cap(\Gamma_{v}(u,\infty)) \leq \delta  n \Bigr\}\right| \preceq \delta n
	\end{equation*}
  for every $n\geq 1$.
\end{lemma}

\begin{proof}
	By Wilson's algorithm we have that, for each two vertices $u$ and $v$ of $G$,
	\begin{multline*}
		\P\left(u \in \fB_v(v,2n) \setminus \fB_v(v,n) \text{ and } \Cap(\Gamma_v(u,\infty)) \leq \delta n \right)
		\\
		= \bP_u\left(\tau_v < \infty,\, n \leq |\LE (X^{\tau_v})| \leq 2n,\, \text{ and } \Cap\left(\LE (X^{\tau_v})\right) \leq \delta n\right),
\end{multline*}
and applying \cref{lem:LoopLength,lem:maincap} we deduce that, letting $r=\lceil 4\bubnorm{P} n\rceil$,
  \begin{multline*}
    \P\left(u \in \fB_v(v,2n) \setminus \fB_v(v,n) \text{ and } \Cap(\Gamma_v(u,\infty)) \leq \delta n \right)
    \\
    \leq 2\bP_u\left(\tau_v \leq r,\, n \leq |\LE (X^{\tau_v})| \leq 2n,\, \text{ and } \Cap\left(\LE (X^{\tau_v})\right) \leq \delta n\right),
    \\
    \leq 
    2\bP_u\left(\tau_v \leq r \text{ and } \mathbf{I}(X^{\tau_v}) \geq  \left[ \inf_{w \in V}c(w)\right]^2 \frac{n}{\delta}\right)\\ \leq 2\sum_{\ell=0}^r\bP_u\left(X_\ell=v \text{ and } \mathbf{I}(X^\ell) \geq  \left[ \inf_{w \in V}c(w)\right]^2 \frac{n}{\delta}\right).
\end{multline*}
Reversing time then yields that
  \begin{multline*}
    \P\left(u \in \fB_v(v,2n) \setminus \fB_v(v,n) \text{ and } \Cap(\Gamma_v(u,\infty)) \leq \delta n \right)
  \\
    \leq 
     2\frac{c(v)}{c(u)}\sum_{\ell=0}^r\bP_v\left(X_\ell=u \text{ and } \mathbf{I}(X^\ell) \geq  \left[ \inf_{w \in V}c(w)\right]^2 \frac{n}{\delta}\right),
\end{multline*}
and summing over $u\in V$ and applying \cref{cor:Imean2} and Markov's inequality we obtain that
  \begin{multline*}
    \E  \left|\Bigl\{ u \in \fB_v(v,2n) \setminus \fB_v(v,n) : \Cap(\Gamma_{v}(u,\infty)) \leq \delta  n \Bigr\}\right|\\\preceq \sum_{\ell=0}^r \bP_v\left(\mathbf{I}(X^\ell) \geq  \left[ \inf_{w \in V}c(w)\right]^2 \frac{n}{\delta} \right)
    \preceq \sum_{\ell=0}^r \frac{\delta (\ell+1)}{n} \preceq \delta n
  \end{multline*}
  as claimed.
% \begin{align*}
% 		\\
% 		 &\hspace{1cm}= \frac{c(v)}{c(u)}\bP_v\left(\tau_u < \infty,\, n \leq |\LE (X^{\tau_v})| \leq 2n,\, \text{ and } \Cap\left(|\LE (X^{\tau_v})|\right) \leq \delta n\right),
% 	\end{align*}
% 	% 
% 	where the second equality follows from time-reversal.
% 	It follows from the estimate \eqref{eq:MFCLoopLength1} of \cref{lem:LoopLength} that 
%   % there exists a constant $C=C(G)$ such that
% 	% 
% 	\begin{multline*}
% 		\P\left(u \in \fB_v(v,2n/3) \setminus \fB_v(v,n/3) \text{ and } \Cap(\Gamma_v(u,\infty)) \leq \delta n \right)
% 		\\
% 		\preceq \bP_v\left( n/3 \leq \tau_u \leq \bubnorm{P}n,\, \text{ and } \Cap\left(|\LE (X^{\tau_v})|\right) \leq \delta n\right).
% 	\end{multline*}
% 	% 
% 	Summing over $u\in V$, we obtain that
% 	\begin{align*}
% E(n,\delta)
%      \preceq \sum_{m=n/3}^{\bubnorm{P}n} \bP_v\left(\Cap\left(|\LE (X^m) |\right) \leq \delta n \right) 
% 		\preceq \left[\sup_{m\geq n/3} \sup_{v\in V} \bP_v\left(\Cap\left(\LE(X^m)\right) \leq \bubnorm{P}^{-1} \delta m \right) \right] n,
% 	\end{align*}
% and the claim follows from \cref{prop:capacities}.
\end{proof}

\begin{proof}[Proof of \cref{lem:exponentsinduction}]
	Fix $v\in V$. Let $\sI_v$ be the $v$-wired interlacement process on $G$, let $\F_{v,t} = \langle \AB_{v,t}(\sI_v) \rangle_{t\in \R}$, and let $\fB_{v,t}(v,n)$ denote the ball of radius $n$ about $v$ in $\F_{v,t}$ for each $t\in \R$ and $n\geq0$. Recall that for each $t\in \R$, $\sigma_t(v) =\sigma_t(v,\sI_v)$ is the first time greater than or equal to $t$ such that $v$ is hit by a trajectory of $\sI_v$ at time $\sigma_t(v)$.

	For each two vertices $u$ and $v$ of $G$, every $t\in \R$ and $n \geq 0$, let $\sB_{t,n}(u,v)$ be the event that $u \in \fB_{v,t}(v,2n) \setminus \fB_{v,t}(v,n)$, and let $\sC_{t,n}(u,v)\subseteq \sB_{t,n}(u,v)$ be the event that $\sB_{t,n}(u,v)$ occurs and that $v$ is connected to $\partial \fB_{v,t}(v,3n)$ by a simple path that passes through $u$.

Let $\eps,\delta>0$.
	 If $\partial \fB_{v,0}\left(v,3n\right) \neq \emptyset$ then we must have that $\sC_{0,n}(u,v)$ occurs for at least $n$ vertices $u$, namely those vertices on the middle third of some path connecting $v$ to $\partial \fB_{v,0}(v,3n)$ in $\F_{v,0}$. Thus, it follows by the union bound and Markov's inequality that
	\[Q(3n) \leq \P(\sigma_0(v) \leq \eps ) + \frac{1}{n}\sum_{u \in V} \P\left(\sC_{0,n}(u,v) \cap \{\sigma_0(v) > \eps \}\right).\]
Let $\Gamma_{v,0}(u,\infty)$ be the future of $u$ in $\F_{v,0}$. 	 By stationarity and \cref{lem:vPastDynamics}, we have that
	\[ \P\Bigl(\sC_{0,n}(u,v) \cap \{\sigma_0(v) > \eps\}\Bigr) \leq
	\P\left(\sC_{0,n}(u,v) \cap \left\{ \Gamma_{v,0}(u,\infty) \cap \cI_{v,[-\eps,0]}  =\emptyset\right\}\right).
	\]
  (This is an inequality rather than an equality because it is possible for the path connecting $u$ to $\partial \mathfrak{B}_{v,0}(v,3n)$ to be hit without $\Gamma_{v,0}(u,\infty)$ being hit.)
	% 
	 % We have by a union bound that
	% \begin{multline*}
	% 	\P( \sC_{0,n}(u,v) \cap \{\sigma_v \geq \delta\}) \leq 
	% 	 \P(\sC_{0,n}(u,v) \cap \{ \Cap(\Gamma(u,v)) \leq cn/3 \})\\
	% 	+\P(\sC_{0,n}(u,v) \cap \{ \Cap(\Gamma(u,v)) \geq cn/3\} \cap \{ \sigma_v > \delta\}). 
	% \end{multline*}
	% The first two terms are easy to control, yielding that
	Next, observe that
	\begin{multline*}
		\P\Bigl(\sC_{0,n}(u,v)  \mid u \in \fT_v,\, \Gamma_{v,0}(u,\infty)=\gamma\Bigr)
		\\
		\leq  \P\Bigl(\partial \fP_{v,0}\left(u,n \right) \neq \emptyset \mid u \in \fT_v,\, \Gamma_{v,0}(u,\infty)=\gamma\Bigr)
		\\
		\leq \P\Bigl(\partial \fB_{u,0}(u,n) \neq \emptyset\Bigr) \leq Q(n)
	\end{multline*}
	for every simple path $\gamma$ from $u$ to $v$, where we have used \cref{lem:domination} (more specifically, \eqref{eq:dom2} with $K=\{u\}$) in the second inequality. Since the events $\sC_{0,n}(u,v)$ and $\{\Gamma_{v,0}(u,\infty) \cap  \cI_{v,[-\eps,0]} =\emptyset \}$ are conditionally independent conditional on the event $\sB_{0,n}(u,v)$ and the random variable $\Gamma_{v,0}(u,\infty)$, we deduce that
	  % and hence that
	% 
	\begin{multline*}
		\P\left(\sC_{0,n}(u,v) \cap \left\{ \Gamma_{v,0}(u,\infty) \cap \cI_{v,[-\eps,0]}  =\emptyset\right\} \mid \sB_{0,n}(u,v),\,\Gamma_{v,0}(u,\infty) \right)\\
		\leq
		 Q(n) \P\left( \Gamma_{v,0}(u,\infty) \cap \cI_{v,[-\eps,0]}  =\emptyset \mid \sB_{0,n}(u,v),\, \Gamma_{v,0}(u,\infty) \right).
	\end{multline*}
	Taking expectations over $\Gamma_{v,0}(u,\infty)$ and applying a union bound, we deduce that
	\begin{multline*}
		\P\left(\sC_{0,n}(u,v) \cap \left\{ \Gamma_{v,0}(u,\infty) \cap \cI_{v,[-\eps,0]}  =\emptyset\right\}\right) 
		\\
		\leq
		Q(n)\P\left(\sB_{0,n}(u,v) \cap \left\{ \Cap(\Gamma_{v,0}(u,\infty)) \geq \delta n \right\} \cap  \left\{\Gamma_{v,0}(u,\infty) \cap  \cI_{v,[-\eps,0]} =\emptyset \right\}\right)\\
		+
		Q(n)\P\left(\sB_{0,n}(u,v) \cap \left\{\Cap(\Gamma_{v,0}(u,\infty)) \leq \delta n \right\}\right).
	\end{multline*}
	Summing the second term over $u \in V$ yields $Q(n) E(n,\delta)$, where $E(n,\delta)$ is the quantity from \cref{lem:ln}.  
	To control the first term, we apply \eqref{eq:CapvCapComparison} to deduce that
	\begin{multline*}
	% &\P(\sB_{0,n}(u,v)\cap\{\Cap(\Gamma_{v,0}(u,\infty)) \geq cn/3)\} \cap \{\sigma_v \geq \delta\}) \\
	\P\left(\sB_{0,n}(u,v) \cap \left\{ \Cap(\Gamma_{v,0}(u,\infty)) \geq \delta n \right\} \cap  \left\{\Gamma_{v,0}(u,\infty) \cap  \cI_{v,[-\eps,0]} =\emptyset \right\}\right)\\
	\preceq e^{- \delta \eps n}\P\Bigl(\sB_{0,n}(u,v)\cap\{\Cap(\Gamma_{v,0}(u,\infty)) \geq \delta n)\}  \Bigr)
	\preceq e^{-\eps \delta n}\P\Bigl(\sB_{0,n}(u,v)\Bigr).
	\end{multline*}
	Thus, summing over $u$ we obtain that
	\begin{equation*}
	 Q(3n) \leq \P(\sigma_0(v) \leq \eps) + \frac{1}{n}Q(n) e^{-\eps \delta n} \E |\fB_{v,0}(0,2n)|
	+ \frac{1}{n}Q(n)E(n,\delta) \preceq \eps  + \left[ e^{-\eps\delta n} + \delta\right] Q(n),
	\end{equation*}
  where we have used \cref{lem:expectedballgrowth1} to bound the second term and \cref{lem:ln} to bound the third. 
  The claim now follows by taking $\delta$ to be a small constant and taking $\eps$ to be $C/n$, where $C$ is a large constant. \qedhere
	% where 
	% the second inequality follows from \cref{lem:expectedballgrowth}.
	% Optimizing over $\delta \geq 0$ by taking
	% \begin{equation}
	% 	\label{eq:optimizedelta}
	% 	\delta = \max \left\{ \frac{3}{cn }\log\left(cn Q\left(n \right)\right),\, 0\right\}
	% \end{equation}
	% yields that there exists a constant $C'=C'(G)$ such that
	% \begin{align*}
	% 	 Q(n) \leq C' n^{-1} + C' n^{-1}\log \left(n Q\left(n\right)\right) +  C' \ell(n) n^{-1} Q\left(n\right).
	% \end{align*}
	% (The stronger inequality $Q(n) \leq C' n^{-1} + C' \ell(n) n^{-1} Q(n)$ is true in the case that we choose $\delta=0$ in \eqref{eq:optimizedelta}.)
	% The result follows since $\ell(n)n^{-1} \to 0$ as $n\to\infty$ by \cref{lem:ln}, and hence $C' \ell(n) n^{-1} \leq 1/6$ for sufficiently large $n$.  \qedhere

\end{proof}

\subsection{The volume}

\begin{proof}[Proof of \cref{thm:generalexponents,thm:generalexponentsv}]
Due to the stochastic domination between the $v$-WUSF and the past of $v$ in the WUSF (\cref{lem:domination}), the desired upper and lower bounds on the probability of a large intrinsic radius follow from Proposition \ref{prop:intlower} and \cref{cor:intupperbound}. Thus, it remains to prove only the desired upper and lower bounds on the probability of a large volume. We begin with the upper bound. 
	% \paragraph{Volume upper bound.}
Using stochastic domination and taking a union bound, we have that
	\begin{align*}
	% \P(|\past_{\F}(v)| \geq n) &\leq 
	\P\bigl(|\fP(v)|\geq n^2\bigr) \leq \P\bigl(|\fT_v| \geq n^2\bigr) \leq 
	\P\Bigl(\partial \fB_v(v,n)  \neq \emptyset\Bigr) 
	+
	\P\Bigl(\big|\fB_v (v,n)\big| \geq n^2\Bigr).
	\end{align*}
	Applying the upper bound on the probability of a large intrinsic diameter from \cref{cor:intupperbound} to control the first term, and \cref{lem:expectedballgrowth1} together with Markov's inequality to control the second term, we obtain
	that
	\begin{align*}
		\P\bigl(|\fP(v)|\geq n^2\bigr) \leq  \P\bigl(|\fT_v| \geq n^2\bigr)  	\preceq n^{-1},
	\end{align*} 
for every $n\geq 1$, from which the claimed upper bounds follow immediately.
% \medskip

	We now turn to the lower bounds. 
For this, we recall the Paley-Zigmund inequality, which states that 
  \[
  \P(Z \geq \eps \E[Z]) \geq (1-\eps)^2 \frac{\E[Z]^2}{\E[Z^2]}
  \]
  for every non-negative random variable $Z$ such that $\P(Z>0) >0$  and every $\eps\in(0,1)$. Applying the Paley-Zigmund inequality to the conditional distribution of the non-negative random variable $Z$ on the event that $Z>0$ readily yields that
  \begin{equation}
    \label{eq:PaleyZigconditional}
    \P\Bigl(Z \geq \eps \E[Z \mid Z>0]  \Bigr) \geq (1-\eps)^2\frac{\E[Z]^2}{\E[Z^2]}
  \end{equation}
  for every real-valued random variable $Z$ such that $\P(Z>0) >0$ and every $\eps\in(0,1)$. 

  To apply the Paley-Zigmund inequality in our setting, we define random variables
	\[
		Z(v,n) = \left| \fP(v, 2n ) \setminus \fP(v,n ) \right| 
\quad\text{and}\quad
		Z_v(v,n) = \left| \fB_v(v,2 n ) \setminus \fB_v(v,n) \right|
	\]
	for each $v\in V$ and $n\geq 1$. 
	 \cref{thm:vmoments,lem:expectedballgrowth2} yield the estimates
      \begin{equation}
    \label{eq:rootedmomentsexponentsproof}
    \E\left[Z_v(v,n)\right]
     \asymp n,\, 
     \text{ and }
      \E \left[Z_v(v,n)^2\right] 
      \preceq n^{3}.
  \end{equation}
Similarly, we have that
	\begin{equation}
		\label{eq:unrootedmomentsexponentsproof}
		q(v) n\, \preceq
		\E\left[Z(v,n)\right]
		 \preceq n,\, 
		 \text{ and }
		  \E \left[Z(v,n)^2\right] 
		  \preceq n^{3},
	\end{equation}
where the lower bound follows from \cref{lem:expectedballgrowth2} and the upper bounds follow from \cref{thm:vmoments} and \cref{lem:domination}. 
Thus, the Paley-Zygmund inequality \eqref{eq:PaleyZigconditional} implies that
	\begin{align*}
		\P\bigl(Z\!\left(v,n\right) \geq c q(v) n^2\bigr) \geq \frac{1}{4}\frac{\E[Z(v,n)]^2}{\E[Z(v,n)^2]} \succeq q(v)^2 n^{-1}
	\end{align*}
and similarly that
	\begin{align*}
		\P(Z_v\!\left(v,n\right) \geq c n^2) \geq \frac{1}{4}\frac{\E[Z_v(v,n)]^2}{\E[Z_v(v,n)^2]} \succeq  n^{-1}.
	\end{align*}
	% in the $v$-rooted case, where we used \eqref{eq:rootedmomentsexponentsproof} and \eqref{eq:proofdiamlowerboundo} in the second inequality. 
	Since $|\fP(v)| \geq Z(v,n)$ and $|\fT_v| \geq Z_v(v,n)$, 
	it follows easily that
	\[
		\P(|\fP(v)| \geq n) \succeq q(v)^{5/2} n^{-1/2}
	% \]
	% and
	\quad \text{ and } \quad
	% \[
		\P(|\fT_v| \geq n) \succeq n^{-1/2}
	\]
 for all $n\geq 1$ as claimed.
	% as claimed. \qedhere
\end{proof}

% \begin{remark}
% With a small amount of extra work one can obtain that $\P(|\fP(v)|\geq n) \succeq q(v) n^{-1/2}$.
% \end{remark}

\subsection{The extrinsic diameter: upper bounds}

In this section we prove our results concerning the tail of the extrinsic diameter of the past.

We begin with the proofs of Theorems \ref{thm:extrinsicspeed} and \ref{thm:extrinsicspeedv}, which are straightforward. 

\begin{proof}[Proof of Theorems \cref{thm:extrinsicspeed} and \ref{thm:extrinsicspeedv}]
The lower bounds follow immediately from \cref{prop:extlower}. Since the extrinsic diameter is bounded from above by the intrinsic diameter, the upper bounds are immediate from \cref{thm:generalexponents,thm:generalexponentsv}. 
\end{proof}

Next, we deduce the upper bounds of Theorems \ref{thm:extrinsic} and \ref{thm:extrinsicv}  from the following more general bound.

\begin{prop}
\label{prop:polyupperext}
Let $G$ be a network with controlled stationary measure and $\bubnorm{P}<\infty$, and suppose that there exist $C$ and $d$ such that $|B(v,r)|\leq Cr^d$ for every $r\geq 1$ and $v\in V$. Then 
\[
\P\bigl(\diam_\ext(\fP(v)) \geq n\bigr) \leq \P\bigl(\diam_\ext(\fT_v) \geq n\bigr) \preceq n^{-2}\log n
\]
for every $n\geq 1$. 
\end{prop}
\begin{proof} The first inequality is immediate from \cref{lem:domination}, so it suffices to prove the second. 
By the union bound we have that
\begin{align*}
\P\bigl( \diam_\ext (\fT_v)\geq n\bigr) &\leq \P\bigl( \diam_\mathrm{int} (\fT_v)\geq m\bigr) +  \P\left(\max\left\{ d(v,u) : u \in \fB_v(v, m)\right\} \geq n \right)
% \P(\operatorname{diam}_\operatorname{ext}(\fT_v) \geq n) \leq \P(\diam(\fT_v) \geq n^2) + \E\left[ \max\{ d(v,u) : u \in \fB_v(v,n^2)\} \right]
\\ &\preceq m^{-1} + \P\left(\max\left\{ d(v,u) : u \in \fB_v(v, m)\right\} \geq n \right)
\\& \leq m^{-1} + \E |\{u \in \fB_v(v,m) : d(v,u) \geq n\} |.
\end{align*}
Recall that the \textbf{Varopoulos-Carne bound} \cite[Theorem 13.4]{LP:book} states that in any network,
\[
p_n(u,v) \leq 2\sqrt{\frac{c(v)}{c(u)}} e^{-d(u,v)^2/(2n)}.
\]
Write $A_n = \{u\in V : d(u,v) \geq n\} = V \setminus B(v,n-1)$ and $m'=\lceil 2\bubnorm{P} m \rceil$. Using \cref{lem:LoopLength} and the Varopoulos-Carne bound, one can easily derive that, using a time-reversal argument similar to that of \cref{lem:expectedballgrowth1},
\begin{align*}
\E |\{u \in \fB_v(v,m) : d(v,u) \geq n\} |
&= \sum_{u \in A_n} \bP_u(\tau_v< \infty, |\LE(X^{\tau_v})|\leq m) \leq \sum_{u \in A_n} 2\bP_u(\tau_v \leq m' )\\ &\leq \sum_{u \in A_n} \sum_{i=0}^{m'}2\bP_u(X_i =v ) \preceq \sum_{i=0}^{m'} \bP_v\left( d(v,X_i) \geq n \right) \\
&\leq \sum_{i=0}^{m'} \sum_{k \geq n} |\partial B(v,k)| \sup_{u \in \partial B(v,k)}p_i(v,u)
 \preceq m \sum_{k\geq n} k^d e^{-k^2/2m},
\end{align*}
so that taking $m= \eps n^2 \log^{-1} n$ for a suitably small constant $\eps>0$ yields that
\[
\P\bigl( \diam_\ext (\fT_v)\geq n\bigr) \preceq n^{-2} \log n. \qedhere
% \P(\operatorname{diam}_\operatorname{ext}(\fT_v) \geq n) \leq \P(\diam(\fT_v) \geq n^2) + \E\left[ \max\{ d(v,u) : u \in \fB_v(v,n^2)\} \right]
\]
\end{proof}

\begin{proof}[Proof of \cref{thm:extrinsic,thm:extrinsicv}]
The lower bounds both follow immediately from Proposition \ref{prop:extlower} and \cref{lem:polyLr}, while the upper bounds are immediate from \cref{prop:polyupperext}.
\end{proof}

We now wish to improve this argument and remove the logarithmic correction in the case of $\Z^d$, $d\geq 5$. To this end, let $\F_0$ be the $0$-wired uniform spanning forest of $\Z^d$, and let $\fT_0$ be the component of $0$ in $\fF_0$. (We do not consider any time parameterised forests in this section, so this notation should not cause confusion.) We write $\Lambda_m=[-m,m]^d$ and $\partial \Lambda_m = \Lambda_m \setminus \Lambda_{m-1}$. We say that a vertex $v\in \partial \Lambda_m$ is a \textbf{pioneer} if it is in $\fT_0$ and the future of $v$ (i.e. the unique path connecting $v$ to $0$ in $\fT_0$) is contained in $\Lambda_m$.

\begin{lemma}
\label{lem:pioneers}
Let $d\geq 5$, and let $\F_0$ be the $0$-wired uniform spanning forest of  $\Z^d$.  Then there exists positive constants $c_d$ and $C_d$ such that
\begin{equation}
\label{eq:pioneerbound}
\E\left[ \#\left\{ \text{pioneers in $\partial\Lambda_m\cap \fB_0(0,n)$} \right\}\right] \leq C_d \exp\left[-c_d\frac{m^2}{n}\right]
\end{equation}
for every $m,n \geq 1$.
\end{lemma}

Note that the expectation of $|\partial \Lambda_m \cap \fB_0(0,n)|$ is of order $m e^{-\Omega(m^2/n)}$. The point of the lemma is that by considering only pioneers we can reduce the expectation by at least a factor of $m$.

Before proving \cref{lem:pioneers}, let us see how it can be applied to deduce \cref{thm:extrinsicZd,thm:extrinsicZdv}.

\begin{proof}[Proof of \cref{thm:extrinsicZd,thm:extrinsicZdv}]
The lower bounds follow from \cref{prop:extlower}, and so by stochastic domination (\cref{lem:domination}) it suffices to prove the upper bound on the probability that $\fT_v$ has a large extrinsic  diameter.
Let 
$Q(n) = \P(\partial \fB_0(0,n) \neq \emptyset)$ and let
$\tilde Q(m) = \P( \fT_0 \cap \partial \Lambda_m \neq \emptyset)$. By the union bound we have that
\begin{align*}
\tilde Q(2m) \leq Q(n) + \P( \fB_0(0,n) \cap \partial \Lambda_{2m} \neq \emptyset).
\end{align*}
Suppose that $\fB_0(0,n) \cap \partial \Lambda_{2m} \neq \emptyset$, and consider a geodesic $\gamma$ in $\fT_0$ from $0$ to $\partial \Lambda_{2m}$. If $\gamma$ visits $\partial \Lambda_m$ for the first time at $v$, then we necessarily have that $v$ is a pioneer in $\partial \Lambda_m \cap \fB_0(0,n)$ and that there is a path from $v$ to $v+\partial \Lambda_m$ that is disjoint from the future of $v$. Thus, counting the expected number of such $v$ and applying the stochastic domination property, we have that, by Markov's inequality,
\[
\P\left( \fB_0(0,n) \cap \partial \Lambda_{2m} \neq \emptyset\right) \leq \tilde Q(m) \E\left[ \#\left\{ \text{pioneers in $\partial\Lambda_m\cap \fB_0(0,n)$} \right\}\right].
\]
Thus, if we take $n= \lceil \eps m^2 \rceil$ for some sufficiently small constant $\eps$, it follows from \cref{thm:generalexponentsv} and \cref{lem:pioneers} that
\begin{align*}
\tilde Q(2m) \leq \frac{C}{m^2} + \frac{1}{8} \tilde Q(m).
\end{align*}
for all sufficiently large $m$.  The proof can now be concluded via an induction similar to that used in the proof of \cref{cor:intupperbound}. (The $1/8$ here could be replaced with any number strictly smaller than $1/4$.)
\end{proof}

The proof of \cref{lem:pioneers} will come down to a few somewhat involved estimates of diagrammatic sums involving random walks on boxes with Dirichlet boundary conditions.  
We write $\preceq$ for upper bounds depending only on the dimension $d$, and, to simplify notation, use the convention that $0^{\alpha}=1$ for every $\alpha \in \R$.
We also use $\Omega$ asymptotic notation ($f=\Omega(g)$ is equivalent to $f \succeq g$), where again the implicit constants depend only on $d$.

Let us first record some basic estimates concerning the random walk on $\Z^d$. 
Let $\mathbf{G}^m_i(v,u)$ be the expected number of times that  a walk started at $v$  visits $u$  when it is stopped either at time $i$ or when it first leaves $\Lambda_m$, whichever is sooner. It follows by Gambler's ruin applied to each coordinate that
\begin{equation}
\label{eq:GamblersRuin}
\sum_{u \in \partial \Lambda_m} \mathbf{G}^{m+r}_\infty(v,u) \leq 2d (r+1)
\end{equation}
for every vertex $v\in \Z^d$ and every $m,r\geq 0$. 
 Furthermore, the maximal version of Azuma's inequality \cite[Section 2]{mcdiarmid1998concentration} implies that there exists a constant $C$ such that
 % \[
 \begin{equation}
\label{eq:Azuma}
% \bP_0\Bigl( \max_{0 \leq k \leq n} |X_k^i| \geq m\Bigr) \preceq \exp\left[ - \frac{m^2}{2n}\right] \quad \text{ and hence that } \quad 
 % \]
% for every $m,n\geq 1$, and hence that
% 
\bP_0( \tau_{\partial \Lambda_m} \leq n) \leq C \exp\left[ - \frac{m^2}{Cn}\right]
\end{equation}
for every $m,n \geq 1$.

We will also use the following estimate.
\begin{lemma}
 $\max_{w \in \partial \Lambda_m} \mathbf{G}^m_\infty(v,w) \preceq \left(m+1-\|v\|_\infty\right)^{-d+1}
$ for every $v \in \Lambda_m$.
\end{lemma}
\begin{proof}
Let $\mathbf{H}(u,w)$ be the Green's function of the simple random walk on $\mathbb{Z}^d$ killed upon exiting the half-space $\mathbb{H}^-:=\{x \in \mathbb{Z}^d : x_d < 0\}$. The elliptic Harnack inequality (see e.g.\ \cite[Chapter 7]{barlow2017random}) implies that there exists a positive constant $C$ such that if $u$ and $v$ are such that $\|u-v\|_\infty \leq \min \{u_d,v_d\}/2$ then 
\begin{equation}
\label{eq:lemEHI} C^{-1}\mathbf{H}(v,w) \leq \mathbf{H}(u,w) \leq C\mathbf{H}(v,w)
\end{equation}
for every $w \in \partial \mathbb{H}^-:=\{x \in \mathbb{Z}^d : x_d = 0\}$. Let $k\geq 0$,  let $e_k = (0,0,\ldots,0,k)$, and let $w_k \in \partial \mathbb{H}^-$ be such that $\mathbf{H}(e_k,w_k)=\max_{w\in \partial \mathbb{H}^-} \mathbf{H}(e_k,w)$. Then it follows by \eqref{eq:lemEHI} and translation symmetry that 
\[\mathbf{H}(e_k,w) = \mathbf{H}(e_k+w_k-w,w_k) \geq C^{-1}\mathbf{H}(e_k,w_k) \]
for every $w \in \partial \mathbb{H}^-$ such that $\|w-w_k\|_\infty \leq k/2$. On the other hand, the expected total time spent in $\partial \mathbb{H}^-$ before entering $\mathbb{H}^-$ is clearly of constant order, and we deduce that
\[
\mathbf{H}(e_k,w_k) \leq C \bigl|\bigl\{w \in \partial \mathbb{H}^- : \|w-w_k\|_\infty \leq k/2\bigr\}\bigr|^{-1} \sum_{v \in \partial \mathbb{H}^-} \mathbf{H}(e_k,v) \preceq (k+1)^{-d+1}.
\]
This immediately implies the claim by translation symmetry.
\end{proof}

Now, since the expected number of times the walk spends in $\partial\Lambda_m$ before leaving $\Lambda_m$ is order 1, we deduce that
\begin{equation}
\label{eq:EHI}
\sum_{w\in \partial \Lambda_m}  \mathbf{G}^m_\infty(v,w)^2 \leq \max_{w \in \partial \Lambda_m} \mathbf{G}^m_\infty(v,w) \sum_{w\in \partial \Lambda_m}  \mathbf{G}^m_\infty(v,w) \preceq \left(m+1-\|v\|_\infty\right)^{-d+1}
\end{equation}
for every $m\geq 0$ and $v\in \Lambda_m$, and similarly that
\begin{equation}
\label{eq:EHI2}
\sum_{w\in \partial \Lambda_m}  \mathbf{G}^m_\infty(v,w)\mathbf{G}^m_{\infty}(u,w)  \preceq \left(m+1-\|v\|_\infty \wedge \|u\|_\infty\right)^{-d+1}
\end{equation}
for every $m\geq 0$ and $u,v \in \Lambda_m$.

We now turn to the proof of \cref{lem:pioneers}.

\begin{proof}[Proof of \cref{lem:pioneers}]

 % denote the Green's function on the box $\Lambda_m$ with Dirichlet boundary conditions. 

We begin with a proof that works only for $d>5$, and then show how it can be modified to obtain a proof for $d\geq 5$. (In fact, the modified proof works for all $d>9/2$.) The variable names we use will be somewhat idiosyncratic; this is to avoid renaming them in the proof for $d=5$.
Let $v \in \partial\Lambda_m$, let $X$ be a random walk started at $v$, and consider sampling $\F_0$ using Wilson's algorithm, starting with the walk $X$. Write $\tau_5$ for the first time that $X$ hits the origin. In order for $v$ to be a pioneer and be in $\fB_0(0,n)$, $X$ must hit $0$, and the loop-erasure $\LE(X^{\tau_5})$ must be contained in $\Lambda_m$. Applying \cref{lem:LoopLength}, we have furthermore that $\tau_5\leq 2\bubnorm{P} n$ with probability at least $1/2$ conditional on this event. Let $n'=\lceil2\bubnorm{P}n\rceil$, and let $\sA_{n,m}(v)$ be the event that $X$ hits $0$, that $\LE(X^{\tau_5})$ is contained in $\Lambda_m$ and that $\tau_5 \leq n'$. Thus, we have that
\[\sum_{v\in \partial \Lambda_m} \P\left(\text{$v$ a pioneer in $\partial\Lambda_m \cap \fB_0(0,n)$}\right) \leq 2\sum_{v\in \partial \Lambda_m} \bP_v\left(\sA_{n,m}(v)\right).
\]
Let $R_2\geq 0$ be maximal such that $X$ visits $\partial \Lambda_{m+R_2}$ before it first visits the origin, let $0 \leq \tau_4 \leq \tau_5$ be the first time that $X$ visits $\partial \Lambda_{m+R_2}$, and let $\tau_3$ be the first time $i$ such that $X_i \in \Lambda_m$ and $X_i = X_j$ for some $j\geq \tau_4$, and let $\tilde \tau_3$ be the first time after $\tau_4$ that $X$ visits $X_{\tau_3}$. Clearly the times $\tau_3$ and $\tilde \tau_3$ must exist and satisfy $0 \leq \tau_3 \leq \tau_4 \leq \tilde \tau_3 \leq \tau_5$ on the event $\sA_{n,m}(v)$. (Note that it is possible that $R_2=\tau_3=\tau_4=\tilde\tau_3=0$.)
Considering the possible choices for $r_2=R_2$ and for the vertices $u = X_{\tau_3}$ and $w= X_{\tau_4}$  yields that
\begin{align*}
 \P\left(\sA_{n,m}(v)\right) &\leq  
     \sum_{u \in \Lambda_m}  \sum_{r_2 \geq 0} \sum_{w\in \partial \Lambda_{m+r_2}} \\
     &\hspace{1.5cm}\bP_v\Bigl( R_2=r_2, X_{\tau_3}=w,X_{\tau_4}=v, \text{ and } |\tau_3|, |\tau_4-\tau_3|, |\tilde\tau_3-\tau_4|, |\tau_5-\tilde\tau_3| \leq n'\Bigr)
\end{align*}
\begin{align*}
\phantom{ \P\left(\sA_{n,m}(v)\right)}
&\leq      \sum_{k_2=0}^m \sum_{u \in \partial\Lambda_{m-k_2}}  \sum_{r_2 \geq 0} \sum_{w\in \partial \Lambda_{m+r_2}} \mathbf{G}^{m+r_2}_{n'}(v,u)\mathbf{G}^{m+r_2}_{n'}(u,w)^2\mathbf{G}^{m+r_2}_{n'}(u,0).
\end{align*}
Applying \eqref{eq:EHI} and reversing time yields that
\begin{equation*}
\sum_{v\in \partial \Lambda_m} \P\left(\sA_{n,m}(v)\right)
\preceq  \sum_{k_2=0}^m \sum_{r_2\geq 0} (r_2+k_2+1)^{-d+1} \sum_{u \in \partial \Lambda_{m-k_2}} \mathbf{G}^{m+r_2}_{n'}(0,u)   \sum_{v\in \partial \Lambda_m} \mathbf{G}^{m+r_2}_{n'}(u,v).
\end{equation*}
Applying \eqref{eq:GamblersRuin} and \eqref{eq:Azuma} we have that
\begin{align*}
\sum_{v\in \partial \Lambda_m}\mathbf{G}^{m+r_2}_{n'}(u,v) &\leq \bP_u\Bigl(\tau_{\partial\Lambda_m} \leq n'\Bigr) \max_{w \in \partial \Lambda_m} \sum_{v\in \partial \Lambda_m}\mathbf{G}^{m+r_2}_{\infty}(w,v) \preceq (r_2+1) \exp\left[ - \Omega(k_2^2/n)\right],
\end{align*}
for every $u\in \partial \Lambda_{m-k_2}$, and similarly that
\begin{multline*}
\sum_{u\in \partial \Lambda_{m-k_2}}
\mathbf{G}^{m+r_2}_{n'}(0,u) \leq \bP_0\Bigl(\tau_{\partial\Lambda_{m-k_2}} \leq n'\Bigr) \max_{w \in \partial \Lambda_{m-k_2}} \sum_{u\in \partial \Lambda_{m-k_2}}\mathbf{G}^{m+r_2}_{\infty}(w,u)\\ \preceq (r_2+k_2+1) \exp\left[ - \Omega((m-k_2)^2/n)\right].
\end{multline*}
Thus,  we deduce that
\begin{multline*}
\sum_{v\in \partial \Lambda_m} \P\left(\sA_{n,m}(v)\right)
\preceq \sum_{k_2=0}^m \sum_{r_2\geq 0} (r_2+1)(r_2+k_2+1)^{-d+2}   \exp\left[-\Omega\left(\frac{(m-k_2)^2+k_2^2}{n}\right)\right]\\
\preceq \exp\left[-\Omega(m^2/n)\right]\sum_{k_2=0}^m (k_2+1)^{-d+4} \preceq \exp\left[-\Omega(m^2/n)\right]
\end{multline*}
as claimed. (In five dimensions we would obtain an unwanted logarithmic correction to this bound since $\sum_{k_2\geq 0} (k_2+1)^{-d+4}$ diverges.)

\begin{figure}
\centering
\includegraphics[width=0.9\textwidth]{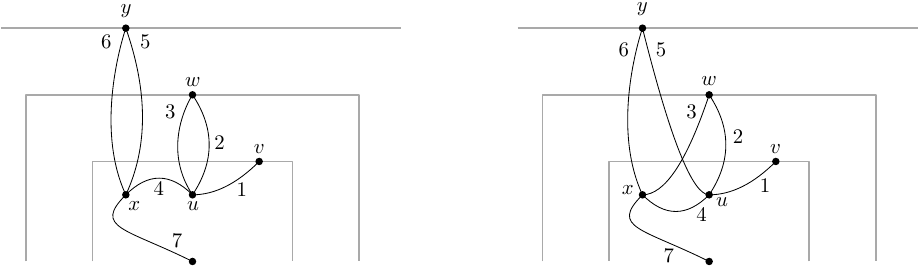}
\caption{Schematic illustration of the two types of paths contributing to the sums estimated in the five dimensional case. In each case, the numbers $1,\ldots,7$ indicate the order in which different segments are traversed by the path. The left figure corresponds to the case $\tilde \tau_1 \leq \tau_3$, while the right figure corresponds to the case $\tilde \tau_1 \geq \tau_3$.}
\label{fig:loops}
\end{figure}

To obtain a corresponding bound in $d=5$ dimensions, we sum over two loops instead of one.  We will be somewhat brief, as a fully detailed treatment of the calculations involved would be quite long.  Let $R_2, \tau_3, \tilde \tau_3, \tau_4,\tau_5$ be as above.
Let $R_1$ be maximal such that $X$ visits $\partial\Lambda_{m+R_1}$ before time $\tau_3$, let $\tau_2$ be the first time that $X$ visits $\partial\Lambda_{m+R_1}$, and let $\tau_1$ be the first time $i$ such that $X_i \in \Lambda_m$ and $X_i=X_j$ for some $j\geq \tau_2$. 
On the event $\sA_{n,m}$, we must have that $\tau_1$ exists and that $0\leq \tau_1\leq \tau_2$. 
Let $\tilde \tau_1$ be the first time $X$ visits $X_{\tau_1}$ after time $\tau_2$. It follows from the definitions that \[0\leq \tau_1 \leq \tau_2 \leq \tilde \tau_1,\tau_3 \leq \tau_4 \leq \tilde \tau_3\leq \tau_5\] on the event $\sA_{n,m}(v)$, but it is possible for $\tilde \tau_1 $ and $\tau_3$ to occur in either order. See \cref{fig:loops} for an illustration. We bound the contribution of the two possibilities separately. In the first case, we have that
% Thus, arguing as above, we obtain that
\begin{multline*}
 \sum_{v\in \partial \Lambda_m}\P\left(\sA_{n,m}(v) , \tilde \tau_1 \leq \tau_3 \right)  \\
\preceq \sum_{0\leq k_1,k_2\leq m} \sum_{r_2 \geq r_1 \geq 0} \sum_{v\in \partial \Lambda_m}  \sum_{u \in \partial\Lambda_{m-k_1}}   \sum_{x \in \partial\Lambda_{m-k_2}}   \sum_{y\in \partial \Lambda_{m+r_2}} \sum_{w\in \partial \Lambda_{m+r_1}}\\  
\hspace{5.2cm}
\mathbf{G}^{m+r_1}_{n'}(v,u) \mathbf{G}^{m+r_1}_{n'}(u,w)^2 \mathbf{G}^{m+r_1}_{n'}(u,x)\mathbf{G}_{n'}^{m+r_2}(x,y)^2 \mathbf{G}_{n'}^{m+r_2}(x,0) ,
% \preceq
\end{multline*}
where we have written the terms in the same order as the corresponding path segments are traversed on the left of \cref{fig:loops}. Performing a similar calculation to that done in the case $d> 5$, above, we obtain that
\begin{multline*}
\sum_{v\in \partial \Lambda_m}\P\left(\sA_{n,m}(v) , \tilde \tau_1 \leq \tau_3 \right)\preceq\\ 
e^{-\Omega(m^2/n)}\sum_{k_1\geq 0}   \sum_{k_2\geq 0}  \sum_{r_2 \geq r_1 \geq 0}   
(k_2+r_2+1)(k_2+r_2+1)^{-d+1}(k_1+r_1+1)(k_1+r_1+1)^{-d+1}(r_1+1).
 % m^{d-1} (k_1+r_1)^{-d+2}(k_2+r_2)^{-d+2}(k_2+r_1)
 % \label{eq:diagram1}
\end{multline*}
We may then show that this sum is finite by computing that
\begin{align*}
\sum_{k_1\geq 0}   \sum_{k_2\geq 0}  \sum_{r_2 \geq r_1 \geq 0}   
(k_2+r_2+1)^{-d+2}(k_1+r_1+1)^{-d+2}(r_1+1) &\preceq 
  \sum_{r_2 \geq r_1 \geq 0}   
(r_2+1)^{-d+3}(r_1+1)^{-d+4}\\ &\preceq \sum_{r_1\geq 0} (r_1+1)^{-2d+8} \preceq 1
\end{align*}
as desired.
Similarly, in the second case, we have that
\begin{multline*}
 \sum_{v\in \partial \Lambda_m }\P\left(\sA_{n,m}(v) , \tau_3 \leq \tilde \tau_1\right)  \preceq 
 \sum_{0\leq k_1,k_2\leq m} \sum_{r_2 \geq r_1 \geq 0} \sum_{v\in \partial \Lambda_m}  \sum_{u \in \partial\Lambda_{m-k_1}}   \sum_{x \in \partial\Lambda_{m-k_2}}   \sum_{y\in \partial \Lambda_{m+r_2}} \sum_{w\in \partial \Lambda_{m+r_1}}\\
 \hspace{1.6cm} \mathbf{G}^{m+r_1}_{n'}(v,u) \mathbf{G}^{m+r_1}_{n'}(u,w) \mathbf{G}^{m+r_1}_{n'}(w,x) \mathbf{G}^{m+r_2}_{n'}(x,u) \mathbf{G}^{m+r_2}_{n'}(u,y) \mathbf{G}^{m+r_2}_{n'}(y,x) \mathbf{G}^{m+r_2}_{n'}(x,0),
\end{multline*}
where we have written the terms in the same order as the corresponding path segments are traversed on the right of \cref{fig:loops}.
% \begin{multline*}
%  \sum_{v\in \partial \Lambda_m }\P\left(\sA_{n,m}(v) , \tau_3 \leq \tilde \tau_1\right)  \preceq 
%  \sum_{0\leq k_1,k_2\leq m} \sum_{r_2 \geq r_1 \geq 0} \sum_{v\in \partial \Lambda_m}  \sum_{u \in \partial\Lambda_{m-k_1}}   \sum_{x \in \partial\Lambda_{m-k_2}}   \sum_{y\in \partial \Lambda_{m+r_2}} \sum_{w\in \partial \Lambda_{m+r_1}}\\
%  % \sum_{v\in \partial \Lambda_m} \sum_{k_1=0}^m \sum_{u \in \partial\Lambda_{m-k_1}}  \sum_{k_2=0}^m \sum_{u \in \partial\Lambda_{m-k_2}}  \sum_{r_2 \geq r_1 \geq 0} \sum_{w_2\in \partial \Lambda_{m+r_2}} \sum_{w_2\in \partial \Lambda_{m+r_2}}\\ 
% % 
% \hspace{1.6cm} \mathbf{G}_{n'}^{m+r_2}(0,x)\mathbf{G}_{n'}^{m+r_2}(x,y)\mathbf{G}_{n'}^{m+r_2}(y,u)\mathbf{G}^{m+r_2}_{n'}(u,x)\mathbf{G}^{m+r_1}_{n'}(x,w)\mathbf{G}^{m+r_1}_{n'}(w,x)\mathbf{G}^{m+r_1}_{n'}(u,v).
% \end{multline*}
% Once again, reversing time and applying \eqref{eq:EHI2} and \eqref{eq:Azuma} as before we obtain that
A similar calculation to above but using \eqref{eq:EHI2} instead of \eqref{eq:EHI} yields that
\begin{multline*}
 \sum_{v\in \partial \Lambda_m }\P\left(\sA_{n,m}(v) , \tau_3 \leq \tilde \tau_1\right) \preceq
\\
e^{-\Omega(m^2/n)}\sum_{k_1\geq 0}   \sum_{k_2\geq 0}  \sum_{r_2 \geq r_1 \geq 0}   
(k_2+r_2+1)(k_1\vee k_2 +r_2+1)^{-d+1}(k_1+r_2+1)(k_1\vee k_2 +r_1+1)^{-d+1}(r_1+1).
% \label{eq:diagram2}
 % m^{d-1} (k_1+r_1)^{-d+2}(k_2+r_2)^{-d+2}(k_2+r_1)
\end{multline*}
As before, we need to show that this sum is finite. To do this, we rewrite the sum in terms of $a=k_1 \wedge k_2$ and $b=k_1 \vee k_2$ to obtain that
\begin{multline*}
\sum_{k_1\geq 0}   \sum_{k_2\geq 0}  \sum_{r_2 \geq r_1 \geq 0}   
(k_2+r_2+1)(k_1\vee k_2 +r_2+1)^{-d+1}(k_1+r_2+1)(k_1\vee k_2 +r_1+1)^{-d+1}(r_1+1)
\\=  \sum_{b\geq a \geq 0} \sum_{r_2 \geq r_1 \geq 0} (a+r_2+1)(b+r_2+1)^{-d+2}(b+r_1+1)^{-d+1}(r_1+1)\\
 \preceq \sum_{b\geq 0} \sum_{r_2 \geq r_1 \geq 0} (b+1)(b+r_2+1)^{-d+3}(b+r_1+1)^{-d+1}(r_1+1).
\end{multline*}
Considering the contribution to this sum from the three cases $b\leq r_1$, $r_1<b<r_2$, and $b\geq r_2$ yields that
\begin{align*}
\sum_{b\geq 0} &\sum_{r_2 \geq r_1 \geq 0} (b+1)(b+r_2+1)^{-d+3}(b+r_1+1)^{-d+1}(r_1+1) \\
&\preceq \sum_{r_2\geq r_1\geq 0} \Biggl[\sum_{b=0}^{r_1} (b+1) (r_2+1)^{-d+3} (r_1+1)^{-d+2} + \sum_{b=r_1}^{r_2} (r_2+1)^{-d+3} (b+1)^{-d+2} (r_1+1)\\ & \hspace{10.7cm}+ \sum_{b \geq r_2} (b+1)^{-2d+5}(r_1+1) \Biggr]
\end{align*}
\begin{align*}
&\preceq  
\sum_{r_2\geq r_1 \geq 0} \left[ (r_2+1)^{-d+3} (r_1+1)^{-d+4} +  (r_2+1)^{-d+3} (r_1+1)^{-d+4} + (r_2+1)^{-2d+6}(r_1+1) \right]\\&\preceq \sum_{r_1\geq 0}  (r_1+1)^{-2d+8} \preceq 1
\end{align*}
as desired. This concludes the proof.
\end{proof}

\section{Spectral dimension, anomalous diffusion}
\label{sec:AlexanderOrbach}

In this section, we apply the estimates of \cref{sec:volume} together with the intrinsic diameter exponent \cref{thm:generalexponentsv} to deduce \cref{thm:AlexanderOrbach}. This will be done via an appeal to the following theorem of Barlow, J\'arai, Kumagai, and Slade \cite{BJKS08}, which gives a sufficient condition for Alexander-Orbach behaviour. See \cite{BJKS08} for quantitative versions of the theorem, and \cite{kumagai2008heat} for generalizations.

We recall that the \textbf{effective conductance} between two disjoint finite sets $A,B$ in a finite network $G$ is defined to be
\[
\Ceff(A\leftrightarrow B ;\, G) =  \sum_{v\in A}c(v)\mathbf{P}_v\bigl(\tau_B<\tau_A^+\bigr).
\]
The \textbf{effective resistance} $\Reff(A\leftrightarrow B ;\, G):= \Ceff(A\leftrightarrow B ;\, G)^{-1}$ is defined to be the reciprocal of the effective conductance. We also define the effective resistance and conductance by the same formulas when $A$ and $B$ are finite subsets of an infinite network $G$ that are such that every transient path from $A$ must pass through $B$. For further background on effective conductances and resistances see e.g.\ \cite[Chapters 2 and 9]{LP:book}.

\begin{theorem}[Barlow, J\'arai, Kumagai, and Slade 2008]
	\label{thm:BJKSquenched}
	Let $(G,\rho)$ be a random rooted graph, and suppose that there exist positive constants $C$, $\gamma$ and $N$ such that
	\begin{equation}
  \label{eq:BJKSgrowth}
		\P\left(\lambda^{-1} n^2 \leq |B(\rho,n)| \leq \lambda n^2 \right)
		\geq 1 - C \lambda^{-\gamma}
	\end{equation}
	and 
	\begin{equation}
  \label{eq:BJKSresistance}
		\P\left( \Reff\left(v \leftrightarrow \partial B(\rho,n);\, G\right) \geq \lambda^{-1}n \right)
		\geq 1 - C \lambda^{-\gamma}
	\end{equation}
	for all $n\geq N$.
	Then $d_s(G)=4/3$ and $d_w(G)=3$ almost surely. In particular, the limits defining both dimensions exist almost surely.
\end{theorem}

The estimate \eqref{eq:BJKSgrowth} has already been established in \cref{cor:exponentialproblargevolume,lem:fullvolumelower}. Thus, to apply \cref{thm:BJKSquenched}, it remains only to prove an upper bound on the probability that the effective conductance is large. The following lemma will suffice.

\begin{lemma}
\label{lem:expectedconductance}
Let $G$ be a network with controlled stationary measure such that $\bubnorm{P}<\infty$. Then 
\[\E\Bigl[ \Ceff(v \leftrightarrow \partial \fB(v,n) ;\, \F)  \Bigr] \preceq n^{-1} \]
for all $n \geq 1$.
\end{lemma}

We begin with the following deterministic lemma. Arguments of the form used to derive this lemma are well known, and a similar bound has appeared in \cite[Lemma 4.5]{BarKum06}.

\begin{lemma}
\label{lem:treeresistance}
Let $T$ be a tree, let $v$ be a vertex of $T$, and let $N_v(n,k)$ be the number of vertices $u \in \partial B(v,k)$ such that $u$ lies on a geodesic in $T$ from $v$ to $\partial B(v,n)$. Then
\[\Ceff\left( v \leftrightarrow \partial B(v,n) ; \, T \right) \leq \frac{1}{k}N_v(n,k)\]
for every $1\leq k \leq n$.
\end{lemma}

\begin{proof}
We use the extremal length characterisation of the effective resistance \cite[Exercise 2.78]{LP:book}. Given a graph $G$ and a function $m:E\to[0,\infty)$, we define the $m$-length of a path in G by summing $m$ over the edges in the path, and define the $m$-distance $d_m(A,Z)$ between two sets of vertices $A$ and $Z$ to be the minimal $m$-length of a path connecting $A$ and $Z$. If $G$ is finite, then we have that
\[
\Ceff( A \leftrightarrow B) = \inf\left\{ \sum_{e} m(e)^2 :   d_m(A,Z) \geq 1 \right\}.
\]
We now apply this bound with $G$ equal to the (subgraph of $T$ induced by the) ball $B(v,n)$ in $T$.
If we set $m(e)=1/k$ if $e$ lies on the first $k$ steps of some geodesic from $v$ to $\partial B(v,n)$ and set $m(e)=0$ otherwise, then we clearly have that $d_m(v,\partial B(v,n))=1$ and, since $N_v(n,k)$ is increasing in $k$,
\[\sum_{e}m(e)^2 = \frac{1}{k^2} \sum_{r=1}^k N_v(n,r) \leq \frac{1}{k} N_v(n,k)\]
as claimed.
\end{proof}

\begin{proof}[Proof of \cref{lem:expectedconductance}]
For each $0 \leq m \leq n$, let $K(n,m)$ be the set of vertices $u\in \partial \fB(v,m)$ such that $u$ lies on a geodesic in $\F$ from $v$ to $\partial \fB(v,n)$, and let  $K'(n,m)$ be the set of vertices $u$ in $\partial \fB(v,m)$ such that 
$\fP(u,n-m) \neq \emptyset$.  Note that $K(n,m) \setminus K'(n,m)$ contains at most one vertex, namely the unique vertex in $\partial \fB(v,m)$ that lies in the future of $v$. 
Thus, by \cref{lem:treeresistance}, we have that
\[\Ceff\left(v\leftrightarrow \partial \fB(v,n) ;\F \right) \leq  \frac{1}{m} |K(n,m)| \leq \frac{1}{m} (|K'(n,m)|+1) \]
for every $1\leq m \leq n$, and hence that
\begin{equation}
\label{eq:conductancelanes}
\Ceff(v\leftrightarrow \partial \fB(v,3n) ;\F ) \preceq  \frac{1}{n^2} \sum_{m=n}^{2n} (|K'(3n,m)|+1) \preceq \frac{1}{n} + \frac{1}{n^2} \sum_{m=n}^{2n} |K'(3n,m)|
\end{equation}
for each $n\geq 1$.
% Thus, it suffices to prove that 
% \[\E\left[\sum_{m=0}^{2n} Z_m \right] \preceq 1.\]
Now, for each vertex $u$ of $G$ and $1\leq m \leq n$, let $\sA_{n,m}(v,u)$ be the event that $u\in K'(n,m)$. By   
  \cref{lem:domination} (more specifically, \eqref{eq:dom1} applied with $K=\{u,v\}$) and \cref{thm:generalexponentsv}, we have that
 \[
\P(\sA_{n,m}(v,u)) \preceq \frac{1}{n+1-m}\P\left(u \in \partial \fB(v,m)\right), 
 \]
 for $n \geq m+1$.
Summing over $u$, we obtain that
\begin{align*}
\label{eq:expectedlanes}
\E\Bigl[\Ceff(v\leftrightarrow \partial \fB(v,3n) ;\F )\Bigr] &\preceq
\frac{1}{n}+\frac{1}{n^2} \sum_{m=n}^{2n} \E |K'(3n,m)| 
\preceq \frac{1}{n}+\frac{1}{n^2}\sum_{m=n}^{2n}\sum_{u\in V} \P(\sA_{3n,m}(v,u))
\\ &\preceq \frac{1}{n}+\frac{1}{n^3}\sum_{m=n}^{2n}\sum_{u\in V} \P(u \in \partial \fB(v,m)) \leq \frac{1}{n}+ \frac{1}{n^3} \E|\fB(v,2n)| \preceq \frac{1}{n},
\end{align*}
where we have used \cref{thm:moments} in the final inequality. This establishes the claim.
\end{proof}

\begin{proof}[Proof of \cref{thm:AlexanderOrbach}]
The claim that $d_f(T)=2$ follows from \cref{cor:quenchedvolumeupper,cor:quenchedvolumelowerbound}. The remaining claims follow immediately by applying \cref{thm:BJKSquenched}, the hypotheses of which are met by \cref{cor:exponentialproblargevolume,lem:fullvolumelower,lem:expectedconductance}.
\end{proof}

\section{Applications to the Abelian sandpile model}
\label{sec:sandpile}

Let $G$ be a transient graph and let $\Eta$ be a uniform infinite recurrent sandpile on $G$, as defined in \cref{subsec:sandpileintro}. Let $\fF$ be the wired uniform spanning forest of $G$, let $\fF_v$ be the $v$-wired uniform spanning forest of $G$ for each vertex $v$ of $G$, let $\fT_v$ be the component of $v$ in $\F_v$, and let $\mathbf{G}$ be the Greens function on $G$. 

Given a recurrent sandpile configuration $\eta$ and a vertex $v$ of $G$, we write $\operatorname{Av}_v(\eta,u)$ for number of times $u$ topples if we add a grain of sand to $\eta$ at $v$ and then stabilize, so that $|\Av_v(\eta)|=\sum_{u\in V} \Av_v(\eta,u)$. 
We recall the following relationships between these objects:
\begin{enumerate}
  \item \textbf{Dhar's formula} \cite{Dhar90} states that the expected number of times $u$ topples when we add a grain of sand at $v$ is given by the Greens function. That is,
\begin{equation}
\label{eq:Dhar}
\E\left[ \operatorname{Av}_v(\Eta,u) \right] = \frac{\mathbf{G}(v,u)}{c(u)}.
\end{equation}
See also \cite[Section 2.3]{MR3857602}. (Note that the right hand side is also the Green's function for \emph{continuous time random walk}.)
\item The avalanche cluster at $v$  approximately stochastically dominates the past of $v$ in the WUSF. More precisely, for any increasing Borel set $\sA \subseteq \{0,1\}^V$, we have that
\begin{equation}
\label{eq:sandpilelowerbound}
\P\left( \AvC_v(\Eta) \in \sA \right)
\geq \frac{1}{\deg(v)} \P\left( \fP(v) \in \sA \right).
\end{equation}
This follows from the discussion in \cite[Section 2.5]{bhupatiraju2016inequalities}, see also equation $(3.2)$ of that paper.
\item The diameter of the avalanche cluster at $v$ is approximately stochastically dominated by the diameter of the component of $v$ in the $v$-WUSF. More precisely, we have that
\begin{equation}
\label{eq:sandpileupperbound}
\P\left(\diam_\ext\left[ \AvC_v(\Eta) \right] \geq r\right) \leq \frac{\mathbf{G}(v,v)}{c(v)} \P\left(\diam_\ext\left[ \fT_v \right] \geq r\right).
\end{equation}
This follows from \cite[Lemma 2.6]{bhupatiraju2016inequalities}.
\end{enumerate}
See \cite{bhupatiraju2016inequalities,MR3857602} for detailed discussions of these properties.

We now apply these relations to deduce \cref{thm:sandpile,thm:sandpilepolynomial,thm:sandpilenonamenable} from the analogous results concerning the WUSF and $v$-WUSF.

\begin{proof}
The lower bounds all follow immediately from \eqref{eq:sandpilelowerbound} together with the corresponding statements for the WUSF, which are given in \cref{thm:generalexponents,thm:extrinsicZd,thm:extrinsic,thm:extrinsicspeed}. Similarly, the upper bounds on the extrinsic radius of the avalanche follow from \cref{eq:sandpileupperbound} and the corresponding statements for the $v$-WUSF. Thus, it remains only to prove the upper bound on the probability of a large number of topplings. For this, we apply a union bound and Dhar's formula to obtain that
\begin{align*}
\P\Bigl(|\Av_v(\Eta)| \geq n\Bigr) &\leq \P\left( \diam_\ext\left[ \AvC_v(\Eta) \right]\geq m \right) + \frac{1}{n}\sum_{u\in B(v,m)} \frac{\mathbf{G}(v,u)}{c(u)}.
% &\leq \P\left( \diam_\ext\left[ \AvC_v(\Eta) \right]\geq m \right) + \frac{1}{n}L(m).
\end{align*}
Under the hypotheses of \cref{thm:sandpile,thm:sandpilepolynomial}, the second term on the right is $O(m^2)$ by \eqref{eq:Occupation}, while under the hypotheses of \cref{thm:sandpilenonamenable} it is $O(m)$ by definition of uniform ballisticity. Thus, the claimed upper bounds follow by applying \cref{thm:extrinsicZdv,thm:extrinsicv,thm:extrinsicspeedv} as appropriate to bound the first term on the right, taking $m=\lceil n^{1/4} \rceil$ in the case of $\Z^d$ (where $d\geq 5$),  taking $m =\lceil n^{1/4} \log^{1/4} n \rceil$ in the case of \cref{thm:sandpilepolynomial}, and taking $m=\lceil n^{1/2} \rceil$ in the uniformly ballistic case.
\end{proof}

\subsection*{Acknowledgments}
We thank Martin Barlow, Antal J\'arai, and Perla Sousi for helpful discussions, and thank Russ Lyons for catching some typos. 
I also thank the two anonymous referees for their close and careful reading of the paper; their comments and suggestions have greatly improved the paper.
Much of this work took place while the author was a PhD student at the University of British Columbia, during which time he was supported by a Microsoft Research PhD Fellowship.

\phantomsection

\addcontentsline{toc}{section}{References}

\footnotesize{
	\bibliographystyle{abbrv}
	\bibliography{unimodular}
 }
\end{document}